\documentclass[english,reqno,11pt]{amsart}
\usepackage{amsmath, amsthm, amsfonts}
\usepackage{hyperref} 
\usepackage{hyperref}
\hypersetup{
	colorlinks=true,
		citecolor=blue!60!black,
		linkcolor=red!60!black,
		urlcolor=green!40!black,
		filecolor=yellow!50!black,
	breaklinks=true,
	pdfpagemode=UseNone,
	bookmarksopen=false,
}

\usepackage[many]{tcolorbox}

\usepackage{amsmath,amsthm,amssymb}
\usepackage{amscd,indentfirst,epsfig}
\usepackage{latexsym}
\usepackage{times}
\usepackage{enumerate}
\usepackage{mathrsfs}
\usepackage{stmaryrd}
\usepackage{amsopn}
\usepackage{amsmath}
\usepackage{amssymb,dsfont,mathtools}
\usepackage{amsfonts,bm}
\usepackage{amsbsy,amsmath}
\usepackage{amscd}
\usepackage{xcolor}
\usepackage[mono=false]{libertine}
\usepackage[T1]{fontenc}
\usepackage{amsthm}
\usepackage[normalem]{ulem} 
\usepackage[cal=euler, scr=boondoxo]{}
\usepackage{microtype}

\usepackage{numprint}

\makeatletter
\def \leq {\leqslant}
\def \le {\leq}
\def \geq {\geqslant}
\def\e{\varepsilon}
\def \ge {\geq}

\def\R{\mathbb R}
\def\S{\mathbb S}

\def\N{\mathbb N}

\def \H {\mathbb{H}}

\def \M {\mathcal{M}}

\def\g{\gamma}
\def \ds {\displaystyle}

\def \d {\mathrm{d}}
\def \lm {\bm{m}}
\def \lM {\mathds{M}}
\def \lD {\mathds{D}}

\def \Q {\mathcal{Q}}

\def\fet{f_{\ast}}

\def\vet{v_{\ast}}

\def \var {\dd}

\def\grad{\nabla}

\newtheorem{theo}{Theorem}[section]

\newtheorem{prop}[theo]{Proposition}
\newtheorem{cor}[theo]{Corollary}
\newtheorem{lem}[theo]{Lemma}

\newtheorem{defi}[theo]{Definition}
\newtheorem{rmq}[theo]{Remark}

\def \dd {\bm{\varepsilon}}

\newcommand{\vertiii}[1]{{\left\vert\kern-0.25ex\left\vert\kern-0.25ex\left\vert #1  %%
    \right\vert\kern-0.25ex\right\vert\kern-0.25ex\right\vert}}                      %%
\newcommand{\verti}[1]{{\left\vert\kern-0.25ex\left\vert\kern-0.25ex\left\vert #1    %%
    \right\vert\kern-0.25ex\right\vert\kern-0.25ex\right\vert}}	

\def \ind {\mathbf{1}}
\numberwithin{equation}{section}

\usepackage{natbib}
\usepackage[many]{tcolorbox}

\title[Landau equation under Prodi-Serrin like criteria]{A priori estimates for solutions to Landau equation under Prodi-Serrin like criteria}
\author{R. Alonso}
\address{$^1$Texas A\&M University at Qatar, Division of Arts and Sciences, Education City, Doha, Qatar.}
\email{ricardo.alonso@qatar.tamu.edu}

\author{V. Bagland}
\address{Universit\'{e} Clermont Auvergne, LMBP, UMR 6620 - CNRS,  Campus des C\'ezeaux, 3, place Vasarely, TSA 60026, CS 60026, F-63178 Aubi\`ere Cedex,
France.}\email{Veronique.Bagland@uca.fr}

\author{L. Desvillettes}
\address{Universit\'e de Paris and Sorbonne Universit\'e, CNRS, IUF, IMJ-PRG, F-75006 Paris, France} \email{desvillettes@imj-prg.fr}

\author{B.  Lods}
\address{Universit\`{a} degli
Studi di Torino \& Collegio Carlo Alberto, Department of Economics, Social Sciences, Applied Mathematics and Statistics ``ESOMAS'', Corso Unione Sovietica, 218/bis, 10134 Torino, Italy.}\email{bertrand.lods@unito.it}

\begin{document}

 \maketitle

\begin{abstract}

In this paper, we introduce Prodi-Serrin like criteria which enable to provide a priori estimates for the solutions to the  spatially homogeneous  Landau equation for all classical soft potentials and dimensions $d \ge 3$. The physical case of Coulomb interaction in dimension $d=3$ is included in our analysis, which generalizes the work of \cite{silvestre}. Our approach is quantitative and does not require a preliminary  knowledge of elaborate tools for nonlinear parabolic equations.\end{abstract}

\tableofcontents

\section{Introduction}

\subsection{The spatially homogeneous Landau equation}

In this work, we are interested in the regularity properties of the solutions to the (spatially homogeneous) Landau equation:
\begin{subequations}\label{eq:Landau}
\begin{equation}\label{eq:Landau1}
\partial_{t} f(t,v) = \Q(f(t,\cdot))(v),  
%{\grad}_v \cdot \int_{\R^{d}} |v-\vet|^{\g+2} \, \Pi(v-\vet) 
%\Big\{f(t,\vet) \grad f(t,v) - f(t,v) {\grad f}(t,\vet) \Big\} \, \d\vet
%\,,\\
\qquad t\geq0, \quad v \in \R^{d},
\end{equation}
supplemented with the initial condition
\begin{equation}\label{eq:Init}
f(t=0,v)=f_{\mathrm{\rm in}}(v), \qquad v \in \R^{d},\end{equation}
\end{subequations}
where $\Q$ here above denotes the (quadratic) Landau collision operator
\begin{equation}\label{op:Landau}
\Q(f)(v) := {\grad}_v \cdot \int_{\R^{d}} |v-\vet|^{\g+2} \, \Pi(v-\vet) 
\Big\{ \fet \grad f - f {\grad f}_\ast \Big\}
\, \d\vet\, ,
\end{equation}
with the usual shorthands $f:= f(v)$, $f_* := f(v_*)$,
%where $\g \geq -d$ 
and where $\Pi(z)$ is the orthogonal projector onto $z^{\bot}$: 
$$\Pi(z) :=\mathrm{Id}-\frac{z \otimes z}{|z|^{2}}.$$ 

In the  present contribution we obtain \emph{new results} for the Landau equation in the case of so-called very  soft potentials
$$-d \leq \gamma < -2, $$
including the physically relevant Coulomb case $d=3$, $\g = -3$, in which the Landau operator models collisions between charged particles, see  \cite{Lif}. We nevertheless point out that the results presented here still hold when $-2 \leq \g <0$ (so-called moderately soft potentials), and enable to get \emph{new proofs of already known results} 
(see for example 
\cite{Wu,DesvJFA} and \cite{ABL}). The case of Maxwell molecules ($\g=0$) and hard potentials ($\g \in (0,1]$) is quite different (see e.g. \cite{DeVi1}), and is not considered in this work. 
\par

 \subsection{Notations} For $k \in \R $ and $ p\geq 1$, we define the Lebesgue space {$L^{p}_{k}=L^{p}_{k}(\R^{d})$} through the norm
$$\displaystyle \|f\|_{L^p_{k}} := \left(\int_{\R^{d}} \big|f(v)\big|^p \, 
\langle v\rangle^{k} \, \d v\right)^{\frac{1}{p}}, \qquad L^{p}_{k}(\R^{d}) :=\Big\{f\::\R^{d} \to \R\;;\,\|f\|_{L^{p}_{k}} < \infty\Big\}\, ,$$
where $\langle v\rangle :=\sqrt{1+|v|^{2}}$, $v\in \R^{d}.$ For $k=0$, we simply denote $\|\cdot\|_{p}$ the $L^{p}$-norm. We also denote the homogeneous Sobolev space $\dot{\H}^{m}$ through the norm
$$\|f\|_{\dot{\H}^{m}}^{2}=\int_{\R^{d}}\left|\xi\right|^{2m}\left|\mathcal{F}[f](\xi)\right|^{2}\d \xi, \qquad m \in \N$$
%and
%\int_{\R^{d}}\left|\nabla f(v)\right|^{2}\d v, \qquad m \in \N$$
and the weighted homogenous Sobolev space $\dot{\H}^{m}_{k}$ through the norm
$$\|f\|_{\dot{\H}^{m}_{k}}^{2}=\int_{\R^{d}}\left|\xi\right|^{2m}\left|\mathcal{F}[\langle \cdot \rangle^{\frac{k}{2}}f](\xi)\right|^{2}\d\xi, \qquad m \in \N, \quad k \in \R$$
where $\mathcal{F}[f]$ denotes the Fourier transform of $f$.  
Given $k \in \R$ and $f \in L^{1}_{k}(\R^{d})$, we also define the statistical moments as
$$
\lm_{k}(f)=\int_{\R^{d}}f(v)\langle v\rangle^{k}\d v\,.
$$
In relation to the coefficients of the Landau equation, we introduce, for $\gamma \in [-d,0)$,
\begin{equation*}
\left\{
\begin{array}{rcl}
a(z) & := & \left(a_{i,j}(z)\right)_{i,j} \quad \mbox{ with }
\quad a_{i,j}(z) 
= |z|^{\g+2} \,\left( \delta_{i,j} -\frac{z_i  z_j}{|z|^2} \right),\medskip\\
 b_i(z) & := & \sum_k \partial_k a_{i,k}(z) = -(d-1) \,z_i \, |z|^\g,  \medskip\\
 c(z) & := & -\sum_{k,l} \partial^2_{kl} a_{k,l}(z).
\end{array}\right.
\end{equation*}
Therefore,
$$c(z)=\begin{cases}(d-1) \,(\g+d) \, |z|^\g, \qquad &\text{ for } \g \in (-d,0),\\
(d-1)(d-2)|\S^{d-1}|\delta_{0}(z), \qquad &\text{ for } \g=-d.\end{cases}
$$
For any $f \in L^{1}_{2+\g}(\R^{d})$, we define then the matrix-valued mapping  $\mathcal{A}[f]$ given by
\begin{equation}\label{eq:Af}
\mathcal{A}[f] :=\big(\mathcal{A}_{ij}[f]\big)_{ij}:=\big(a_{ij}\ast f\big)_{ij} .\end{equation}
In the same way, we define $\bm{b}[f] \in \R^{d}$ and $\bm{c}_{\g}[f] \in \R$  by 
$$\bm{b}_{i}[f] := b_{i} \ast f , \qquad \qquad i=1,\ldots,d\,; \qquad  \bm{c}_{\g}[f]  := c \ast f\,.$$
We emphasize the dependency with respect to the parameter $\g$ in $\bm{c}_{\g}[f]$ since, in several places, we apply the same definition with $\g+1$ replacing $\g$. Notice that
\begin{equation}\label{eq:cgCoul}
\bm{c}_{\g}[f](v)=\begin{cases}(d-1)\,(\g+d)\ds\int_{\R^{d}}|v-\vet|^{\g}f(\vet)\d\vet \qquad &\text{ for } \g \in (-d,0),\\

c_{d}\,f(v)\qquad &\text{ for } \g=-d,\end{cases}
\end{equation}
where $c_{d}:=(d-1)(d-2)|\S^{d-1}|.$

\smallskip
\noindent
With these notations, the  Landau  equation can then be written alternatively under the form
\begin{equation}\label{LFD}
\left\{
\begin{array}{ccl}
\;\partial_{t} f &= &\grad \cdot \big(\,\mathcal{A}[f]\, \grad f
- \bm{b}[f]\, f\big)=\mathrm{Tr}\left(\mathcal{A}[f] D^{2}f\right) + \bm{c}_{\g}[f]f , \medskip\\
\;f(t=0)&=&f_{\mathrm{\rm in}}\, ,
\end{array}\right.
\end{equation}
where $D^{2}f$ is the Hessian matrix of $f$. {Notice that, with respect to previous existing works on the field, we adopt here a convention in which the term $\bm{c}_{\g}[f]$ in \eqref{LFD} is nonnegative. With such a convention, $\bm{c}_{\g}[f]=-\nabla \cdot \bm{b}[f].$}

%In the above formulation, the Landau equation can be seen as a special case of nonlinear parabolic equation with critical nonlinearity. 
\subsection{Solutions to the (spatially homogeneous) Landau equation} 

We discuss here some of the results existing in the literature for the Landau equation  in dimension $d=3$. 
\medskip

In the case of hard potentials ($\g \in (0,1]$), existence, uniqueness, appearance of smoothness and of moments is known \cite{DeVi1}. Exponential convergence towards equilibrium also holds, see \cite{ABDL2} and \cite{Desv_Braga}. The situation for Maxwell molecules ($\g=0$) and moderately soft potentials ($\g \in [-2, 0)$), is almost identical in terms of existence, uniqueness and appearance of smoothness. The main difference concerns the statistical moments, which are propagated (and grow at most linearly with time) 
%in the case of moderately soft potentials)
 but are not created (see \cite{Wu, DesvJFA, DeVi1}). Moreover,  the speed of convergence towards equilibrium is
% not exponential anymore when $\g <0$. 
proved under suitable assumptions on the initial datum to be "stretched exponential"  (see Theorem 1.4 of \cite{kleber} for the case $-1 < \g <0$ or Section 6 of \cite{ABDL1} for the more general case of Landau-Fermi-Dirac equation, covering in particular the Landau case).  A systematic study of the Landau equation for moderately soft potentials, including in particular the ``critical case'' $\g=-2$ has been addressed recently with techniques partly similar to those of the present paper in \cite{ABL}. 
\medskip

The case of  very soft potentials ($\g \in [-3,-2)$), which includes the physically relevant case of Coulomb potential ($\g = -3$) is different. Up to the appearance of the very recent paper \cite{GS}, only weak solutions (including H-solutions) were known to exist  and uniqueness was an open problem as discussed in \cite{DesvJFA}. Notice that stretched exponential convergence to equilibrium still holds \cite{CDH}. 

\medskip

Focusing on Coulomb interactions $\g=-3$, the main \emph{a priori} estimate in Lebesgue spaces states that H-solutions $f$ satisfy
$$f \in  L^1([0,T]; L^{3}_{3\g}(\R^3)),$$
see \cite{DesvJFA}. We point out the very recent improvement in terms of moments for Coulomb potential $\g=-3$ which shows that $f \in L^{1}([0,T];L^{3}_{s}(\R^{3}))$ with $s \geq -5$, see \cite{Ji}. Such an estimate is deduced from a careful study of the entropy production of the Landau collision operator and can be interpolated with the energy estimate  stating that $f \in L^{\infty}([0,T]; L^1_{2}(\R^3))$. 
\medskip

As far as regularity of solutions is concerned, it is possible to get a bound on the Hausdorff dimension of the times in which the solution might be singular (see \cite{GGIV} for the Coulomb case and \cite{GIV} for general very soft potentials). Various perturbative results are also available in the literature: we refer to \cite{GGL} for a recent construction of close-to-equilibrium solutions in the case of Coulomb interactions while local in time solutions for large data and global in time solutions for data close to equilibrium have been recently obtained in \cite{DesvHe}.  
\medskip

Finally, in \cite[Theorem 3.8]{silvestre}, conditional results show that if the solution lies in $L^{\infty}([0,T]; L^{p}_{\kappa}(\R^3))$ for $p>\frac32$ and $\kappa$ sufficiently large, then it is bounded and consequently smooth. We refer also to  \cite{BP} for a local version of a close result.  We discuss the consequences of the new work \cite{GS} on the results presented in this paper in Section \ref{sec:GS}.
\medskip

We end this  description of the literature on the Landau equation with very interesting results about a related model, known as the \emph{isotropic Landau equation}, introduced in \cite{Krieger} and investigated later in \cite{GKS,GG1,GG2}. For such a model, the projection $\Pi(z)$ in the Landau equation is simply replaced with $\Pi(z)=\mathrm{Id}$ and this simplification yields striking results. In particular, there exists $\g_{*} \in (-\frac{5}{2},-2)$ such that, for $\g \in (\g_{*},-2)$, $L^{p}$-norms are \emph{propagated} for $p > \frac{d}{d+\g+2}$ and become instantaneously bounded. Therefore, classical solutions of the isotropic Landau equation remain smooth for every ﬁnite time (see  Theorem 1.1 of \cite{GG2}). Similar results have been obtained recently for an isotropic version of the Boltzmann equation, see \cite{snelson}.

%The main result of \cite{silvestre}, as far as Landau's equation is concerned, is that provided 
%\begin{equation}\label{eq:silvestre}
%\sup_{t \in [0,T]}\|f(t)\|_{L^{p}_{\kappa}}<\infty
%\end{equation}
%with $p > \frac{3}{2}$ and $\kappa$ large enough, see \cite[Theorem 3.8]{silvestre} for a precise statement, then  solutions to \eqref{eq:Landau} are bounded.
%\begin{itemize}
%\item Review the results for spatially homogeneous Landau equation for \emph{moderately} soft potentials $-2 \leq \g <0.$ Insist that, in this case, solutions to \eqref{eq:Landau}, are known to be unique, smooth, and even the long-time behavior of solutions is understood.
%\item Point out the main difference with the case $-d \leq \g < -2$ for which only conditional results exist and regularity/uniqueness is an open problem. Case of particular interest is that of Coulomb interactions $\g=-d.$ Only physically realistic interactions \cite{chiara}.
%\end{itemize}

\medskip

\subsection{Link between the Landau equation and the Navier-Stokes equation}

Some striking analogies between the spatially homogeneous Landau for very soft potentials and the incompressible $3d$-Navier-Stokes have been observed in recent contributions. More precisely, recall the incompressible $3d$-Navier-Stokes with viscosity $\bm{\nu} >0$:
\begin{equation}\label{eq:NS}
\partial_{t}u + u \cdot \nabla u - \bm{\nu}\,\Delta u+\nabla p=0, \qquad 
\nabla \cdot u=0,\qquad 
u(t=0)=u_{\rm in}, \qquad \nabla \cdot u_{\rm in}=0,\,
\end{equation}
where $u=u(t,x) \in \R^{3}$. It appears that the $H$-solutions constructed in \cite{villH}
 share several common points with Leray solutions to the incompressible $3d$-Navier-Stokes equations  \cite{Leray,Ozan}:
\begin{itemize}
\item One first evidence is the contribution of the third author about the role of entropy dissipation, showing that $H$-solutions to Landau equation satisfy \eqref{eq:fish} for any $T >0$ providing an a priori estimate for solutions $f$ to \eqref{eq:Landau}. In particular, $H$-solutions to Landau equation are weak-solutions just like Leray solutions to Navier-Stokes equations are weak solutions and the  entropy $H$ plays for the Landau equation a similar role as the energy $\|u(t)\|^2_{2}$ for the Navier-Stokes equation. Note that in this analogy, $\sqrt{f}$ plays the role of $u$, since on one hand $\sqrt{f} \in L^{\infty}_t(L^2_v)$, and on the other hand $\sqrt{f}$ belongs to some weigthed $L^2_t(H^1_v)$ space.
\item Based on this estimate, it has been shown in \cite{GGIV} that the Hausdorff dimension of the set of time singularities is at most $\frac{1}{2}$ for solutions to \eqref{eq:Landau} in the Coulomb case ($\g=-3$, $d=3$). This result has been extended to  soft potentials $-3  < \g < -2$ in \cite{GIV} and is the 
 analogue of the celebrated Caffarelli-Kohn-Nirenberg Theorem \cite{CKN,OzanCK} for the Navier-Stokes equations.
\item Finally, a new monotonicity formula for a functional involving entropy, regularity and moments has been obtained in \cite{DesvHe}. This formula entails in particular that, if 
$$H(f_{\rm in})\left(\|f_{\rm in}-\M\|_{\dot{\H}^{1}}+C\right) \leq \frac{5}{2},$$
then no blow-up can occur for solutions to \eqref{eq:Landau}. Here,  $H(f_{\rm in})$ is the entropy of the initial datum, $\M$ is the associated Maxwellian steady state, and $C >0$ is some fixed universal constant. Such a result  can be compared to the result of \cite{Leray} for Navier-Stokes equation which shows that, provided $\|u_{\rm in}\|_{2}\,\|\nabla u_{\rm in}\|_{2} \ll 1,$ the Navier-Stokes equations admit global smooth solutions (see also \cite{Ozan}). Regularity of solutions after a given (sufficiently large)  time can also be considered as a common feature of both equations.
\end{itemize}

An important criterion about the existence of global classical solutions to 3d-incompressible Navier-Stokes equation is the celebrated \emph{Prodi-Serrin} criterion which reads as follows (see \cite[Theorems 4.8 and 4.9]{vicol}):
\begin{theo}[\textit{\textbf{Prodi-Serrin criterion for 3d-incompressible Navier-Stokes equation}}]\label{theo:PSNS}
Consider two Leray-Hopf weak solutions $u ,\tilde{u}$ of the Navier-Stokes equation with $u(0)=\tilde{u}(0)=u_{\rm in}$. If 
\begin{equation}\label{eq:PSNS}
u \in L^{r}([0,T];L^{q}(\R^{3})), \qquad with \qquad \frac{2}{r}+\frac{3}{q}=1, \quad q \in (3,\infty]
\end{equation}
then $u=\tilde{u}$ on $[0,T).$ Moreover, for $q=3$ (and $r=\infty$), there exists a universal constant $\delta >0$ such that, if 
$$\sup_{t \in [0,T]}\|u(t)\|_{3} \leq \delta \bm{\nu}$$ then $u=\tilde{u}$ on $[0,T).$ Finally, in both cases, $u$ is a strong solution on $[0,T]$ provided $u_{\rm in} \in H^{1}(\R^{3})$ whereas, if $u_{\rm in} \in L^{2}(\R^{3})$, then $u $ is a strong solution solution on $[t_{*},T]$ for any $t_{*} >0.$
\end{theo}

We refer to \cite{vicol,lemarie} for more details on the notions of Leray-Hopf weak solutions in the above Theorem. The aforementioned result has been proven first in \cite{prodi} (with some additional assumption) for $q >3$ and derived in the present form in \cite{serrin}, still for $q >3.$ The statement here above for $q=3$ is a very special case called the \emph{endpoint Prodi-Serrin} criterion. It is somehow possible to remove the smallness assumption on $\|u\|_{L^{\infty}([0,T]\;;L^{3}(\R^{3}))}$ in that case, but the corresponding result reads then  a bit differently: the breakthrough result \cite{seregin} dealing with the case $r=\infty,q=3$ shows indeed smoothness and uniqueness in some indirect way. Precisely, result of \cite{seregin} rather reads as follows: if  $u$ is a classical solution to Navier-Stokes equation whose maximal time of existence $T$ is finite, then
$$\limsup_{t \to T}\left\|u(t)\right\|_{3}=+\infty.$$
We also refer the reader to Chapters 12 and 15 of \cite{lemarie}.  Similar criteria have been established for other types of problems, spanning from the study of  MHD equations \cite{Jia} to stochastic differential equations \cite{Neves} ; \cite{KR,RZ23}. We notice that the proof of Theorem \ref{theo:PSNS} is fully quantitative whereas the endpoint proof of \cite{seregin} is obtained by contradiction and thanks to a compactness argument (we refer to the recent contributions \cite{tao1,palasek} for a quantitative approach of the same result).

%The Prodi-Serrin criterion shows in particular that the mixed Lebesgue spaces
%$$L^{r}([0,T];L^{q}(\R^{3})), \qquad \frac{2}{r}+\frac{3}{q}=1$$
%are actually critical spaces for the Navier-Stokes equations and some extra integrability in such spaces are enough to guarantee the uniqueness in the class of Leray-Hopf solutions to Navier-Stokes equations. %
\subsection{Main results of the paper}

In the present paper, we intend to obtain a result which extends the bounds obtained in \cite{silvestre}, and which is the equivalent for the Landau equation of the Prodi-Serrin result for the Navier-Stokes equation. As mentioned already, our results cover  all dimensions $d \ge 3$, any soft potentials $\g \in [-d, 0)$, and are valid for a suitable range  of admissible $L^r_t(L^q_v)$ spaces (with exception of the end-point estimate $r=\infty$, see Section \ref{sec:endpoint}). 

As it is the case for the Navier-Stokes equation, the Prodi-Serrin criteria that we obtain are actually the fundamental conditional assumptions which allow to prove the propagation and appearance of $L^{p}$-norms for suitable $p$, for the solutions to the spatially homogeneous equation. Such propagation/appearance of $L^{p}$-norms can be seen as the main result of the present contribution and such conditional results are completely new to our knowledge (for $r<\infty$). We will see that they are also the cornerstone for uniqueness and further smoothness of solutions to \eqref{eq:Landau}. 

For the clarity of presentation, and due to the physical relevance of the Coulomb case, we distinguish the two cases $\g=-d$ and $\g \in (-d,0)$, beginning with the Prodi-Serrin criterion for Landau equation in the Coulomb case $\g =-d$.
 
\begin{theo}[\textit{\textbf{Prodi-Serrin criterion for Landau equation and Coulomb interaction}}]\label{theo:Coul}
We consider an integer dimension $d \ge 3$ and let $f_{\rm in}$ and $f=f(t,v)$ define a  solution to Eq. \eqref{eq:Landau} in the sense of Definition \ref{defi:weak} with 
$$\g=-d.$$ Assume that the solution $f$ satisfies one of the following conditions:
\begin{subequations}\label{eq:PS-Cin}
\begin{equation}\label{eq:PS-Cin1}
f \in L^{1}([0,T]\,;\,L^{\infty}(\R^{d})),
\end{equation}
or
\begin{equation}\label{eq:PS-Cinr} \langle \cdot \rangle^{d}f \in L^{r}([0,T]\,;\,L^{q}(\R^{d})), \qquad \frac{2}{r}+\frac{d}{q}=2, \quad r \in (1,\infty), \;\;q \in \left(\frac{d}{2},\infty\right).\end{equation}
\end{subequations}
Then, the following holds
\begin{enumerate}[(a)]
\item \textbf{{(Propagation of $L^{p}$-norms).}} Given $p \in (1,\infty)$, if 
$f_{\rm in} \in L^{p}(\R^{d})$ then 
$$\sup_{t \in [0,T]}\|f(t)\|_{p} \le  \bm{C}_{p,q,T}\,\|f_{\rm in}\|_{p }$$
where  $\bm{C}_{p,q,T}$ is  an \emph{explicit} constant depending on $p,d,q$ and the norms $\|f\|_{L^{1}([0,T]\,;\,L^{\infty}(\R^{d}))}$ or $\|\langle \cdot \rangle^{d}f\|_{L^{r}([0,T]\,;\,L^{q}(\R^{d}))}$ 
 according to the above condition \eqref{eq:PS-Cin}.
\medskip

\item \textbf{{(Appearance of $L^{p}$-norms).}} Given $p \in (1,\infty)$, assume that  
$$f_{\rm in} \in L^{1}_{\nu_{p}}(\R^{d}), \qquad \nu_{p}= \frac{d^{2}}{2}\frac{p-1}{p}$$
then
$$\|f(t)\|_{p} \le C_{p,T}(f_{\rm in})\,t^{-\frac{d}{2}(1-\frac{1}{p})}, \qquad \forall t \in (0,T]$$
for some \emph{explicit} positive constant $C_{p,T}(f_{\rm in})$ which depends on {$T,p,d,q$}, the initial datum $f_{\rm in}$ through the quantifies  $\varrho_{\rm in}$, $E_{\rm in}$, $H(f_{\mathrm{\rm in}})$ defined in Definition \ref{weak-sol}, $\|f_{\rm in}\|_{L^{1}_{\nu_{p}}}$, and the norms associated to the above condition \eqref{eq:PS-Cin}.
\end{enumerate} 
\end{theo}

\begin{rmq} We emphasise here the different nature of the two Prodi-Serrin criteria in \eqref{eq:PS-Cin}. The assumption \eqref{eq:PS-Cin1} does not require any moment assumption on $f$ whereas the Prodi-Serrin criterion \eqref{eq:PS-Cinr} involves the weighted solution $\langle \cdot \rangle^{d}f$, see Remark \ref{nb:PSineq} for more details. We also point out that, under the assumption $f \in L^{1}([0,T]\,,\,L^{\infty}(\R^{d}))$, the conclusion of point $(a)$ still holds for $p=\infty$ (see Prop. \ref{without_dissip}).
 \end{rmq}

The full proof of the above Theorem is presented in Section \ref{sec:Coulomb} and is based on several intermediate results. Typically, Proposition \ref{without_dissip} shows the propagation of $L^{p}$-norm under the $L^{1}_{t}(L^{\infty}_{v})$ criterion while Proposition \ref{eq_with_dissip} shows the appearance of $L^{p}$-norms under this assumption. The general $L^{r}_{t}(L^{q}_{v})$ case is dealt with in Proposition \ref{mps_coul} and Corollary \ref{cor:LinftyLP} for appearance and propagation of $L^{p}$-norms respectively. Note also that constants are explicitly given in those propositions.
\medskip

An analogue of the above result holds true for general soft-potentials $\g \in (-d,0)$

\begin{theo}[\textit{\textbf{Prodi-Serrin criterion for Landau equation with soft potentials}}] 
\label{827}

We consider an integer dimension $d \ge 3$ and let $f_{\rm in}$ and $f=f(t,v)$ define a  solution to Eq. \eqref{eq:Landau} in the sense of Definition \ref{defi:weak} with 
$$-d < \g <0.$$ Assume that the solution $f$ satisfies one of the following conditions:
\begin{subequations}\label{eq:PS-Cing}
\begin{equation}\label{eq:PS-Cing1}
f \in L^{1}([0,T]\,;\,L^{\frac{d}{d+\g}}(\R^{d}))
\end{equation}
or
\begin{equation}\label{eq:PS-Cingr}
\langle \cdot \rangle^{|\g|}f \in L^{r}([0,T]\,;\,L^{q}(\R^{d})), \qquad \frac{2}{r}+\frac{d}{q}=2+d+\g, \end{equation}\end{subequations}
where $r \in (1,\infty),$ $q \in \left(\max\left(1,\frac{d}{d+\g+2}\right), \frac{d}{d+\g} \right).$
Then, the following holds
\begin{enumerate}[(a)]
\item \textbf{{(Propagation of $L^{p}$-norms).}} If $f_{\rm in} \in L^{p}(\R^{d})$ with $p \in (1,\frac{d}{d+\g})$ under condition \eqref{eq:PS-Cing1} or $p \in (1,\infty)$ under condition \eqref{eq:PS-Cingr}, then 
$$\sup_{t \in [0,T]}\|f(t)\|_{p} \le  \bm{C}_{p,\g,q,T}\,\|f_{\rm in}\|_{p }$$
where  $\bm{C}_{p,\g,q,T}$ is  an \emph{explicit} constant depending on $p,d,q,\g$ {the initial datum $f_{\rm in}$ through  $\varrho_{\rm in}$, $E_{\rm in}$, $H(f_{\mathrm{\rm in}})$,} and the norms associated to the above conditions \eqref{eq:PS-Cing}.
\medskip

\item \textbf{{(Appearance of $L^{p}$-norms).}} Let $p \in (1,\frac{d}{d+\g})$ under condition \eqref{eq:PS-Cing1} or $p \in (1,\infty)$ under condition \eqref{eq:PS-Cingr}. Assume that  
$$f_{\rm in} \in L^{1}_{\nu_{\g,p}}(\R^{d}), \qquad \nu_{\g,p}:= \frac{|\g| d}{2}\left(1-\frac{1}{p}\right)\,,$$
then
$$\|f(t)\|_{p} \le C_{p,\g,q,T}(f_{\rm in})\,t^{-\frac{d}{2}(1-\frac{1}{p})}, \qquad \forall t \in (0,T]$$
for some \emph{explicit} positive constant $C_{p,\g,q,T}(f_{\rm in})$ which depends on  {$d,\g,T,p,q$}, the initial datum $f_{\rm in}$ through  $\varrho_{\rm in}$, $E_{\rm in}$, $H(f_{\mathrm{\rm in}})$, $\|f_{\rm in}\|_{L^{1}_{\nu_{\g,p}}}$, and the norms associated to the above condition \eqref{eq:PS-Cing}.
\end{enumerate}  
\end{theo}
 
\begin{rmq} Notice that the link $\frac{2}{r}+\frac{d}{q}=2+d+\g$, with $r,q >1$ implies 
$$\max\left(1,\frac{d}{d+2+\g}\right) < q < \frac{d}{d+\g},$$
so that this condition is automatically satisfied in Theorem \ref{827}.
\end{rmq}  
 
\begin{rmq}\label{nb:PSineq} As in the Coulomb case, one notices the discrepancy between the two sets of assumptions for our Prodi-Serrin like criteria. In particular, the assumption \eqref{eq:PS-Cing1} does not require any moment assumption on $f$ whereas the Prodi-Serrin criterion \eqref{eq:PS-Cingr} applies to the weighted solution $\langle \cdot \rangle^{|\g|}f$ instead of the mere solution $f$. We point out that, by using a simple interpolation argument, we can reformulate the Prodi-Serrin assumption with an \emph{inequality} linking the two parameters $(r,q)$.  Namely, if one assumes that $\g+2<0$ and
$$f \in L^{r_{0}}([0,T]\,,\,L^{q_{0}}(\R^{d})) \qquad \frac{2}{r_{0}} + \frac{d}{q_{0}} < d+\g+2$$ 
then given $\max\left(1,\frac{d}{d+\g+2}\right) < q < \infty$, assuming 
$$f_{\rm in} \in L^{1}_{\mu}(\R^{d}), \qquad \mu=|\g|\frac{q(q_{0}-1)}{q_{0}-q}$$
one  can find $r >1$ such that
$$\langle \cdot \rangle^{|\g|}f \in L^{r}([0,T]\,,\,L^{q}(\R^{d})), \qquad \frac{2}{r}+\frac{d}{q}=d+\g+2.$$
We also refer to \cite{GGL} for consideration on another kind of inequality relating the parameters $r$ and $q.$\end{rmq}
 
The above result is similar to its analogue for the Coulomb case but we point out here that, dealing with soft-potentials $\g \in (-d,0)$ leads to additional technical difficulties due to the fact that the lowest order term $\bm{c}_{\g}[f]$ in \eqref{eq:Landau} is a convolution. To overcome this difficulty, we give here a general estimate involving such a term, which is a variant of estimates introduced in \cite{GG}, who named it $\e$-Poincar\'e inequality. For moderately soft potentials, the $\e$-Poincar\'e inequality  has been already used to show the appearance of $L^{p}$-norms in the contributions \cite{ABDL1, ABL}. We propose here a version, possibly sharp, of such a functional inequality, which involves suitable $L^{q}$-space estimates:%

\begin{prop}[\textit{\textbf{General $\e$-Poincar\'e inequality}}]\label{prop:GG} Assume that $d\in\N$, $d\ge 3$, $-d \le \g < 0$, and
$$ \max\left(1, \frac{d}{d+2+\g}\right) < q  < \frac{d}{d+\g},  \qquad {\left( \frac{d}{d+\g} = \infty \quad {\hbox{ if }} \quad \g = -d \right)} .$$
Then there exists $C_{0} >0$ depending only on $d, \g, q$ such that, for any $\e>0$ (and suitable functions $\phi$ and $g\ge0$),
\begin{multline}\label{eq:estimatcLP}
\int_{\R^{d}}\phi^{2}\bm{c}_{\g}[g]\d v \leq \e\,\int_{\R^{d}}\left|\nabla \left(\langle v\rangle^{\frac{\g}{2}}\phi(v)\right)\right|^{2}\d v \\
+C_{0}\left( \|g\|_{{1}}  +\e^{-\frac{s}{1-s}}\left\|\langle \cdot \rangle^{|\g|}g\right\|_{q}^{\frac{1}{1-s}}\right)\int_{\R^{d}}\phi^{2}\langle v\rangle^{\g}\d v,
\end{multline}
where $s \in (0,1)$ is given by $s=\dfrac{d-q(d+\g)}{2q}.$
\end{prop} 

\begin{rmq} With the above notations, if one sets $r=\frac{1}{1-s}$, one notices that $\frac{2}{r}+\frac{d}{q}=d+\g+2$. When applied to the solution $g=f(t)$ to the Landau equation \eqref{eq:Landau}, one sees therefore that the integrability in time of $\left\|\langle \cdot \rangle^{|\g|}f(t)\right\|_{q}^{\frac{1}{1-s}}$ exactly corresponds to our Prodi-Serrin criterion \eqref{eq:PS-Cingr}.\end{rmq}

The proof of this fundamental functional inequality is given in Section \ref{sec:eps}. This is the main tool in the proof of Theorem \ref{827}, under assumption \eqref{eq:PS-Cingr}. The proof of Theorem \ref{827} is then given in Section \ref{sec:prodi}, in which Proposition \ref{without_dissip-g} shows the propagation of $L^{p}$-norm under condition \eqref{eq:PS-Cing1} while Proposition \ref{prop:L1-criterion} shows the appearance of $L^{p}$-norms under this same condition. The general $L^{r}_{t}(L^{q}_{v})$ case \eqref{eq:PS-Cingr} is dealt with in Proposition \ref{mps_g} and Corollary \ref{cor:propagLP}, for appearance and propagation of $L^{p}$-norms respectively. The  constants in Theorem \ref{827} are made explicit in those results.

\subsection{Relevance of our main results}

We chose to present the main results of the paper as conditional theorems ensuring the propagation and appearance of $L^{p}$-norms for the solutions to the Landau equation. Besides their specific interest, such propagation and appearance results are the fundamental tools which allow to deduce both \emph{uniqueness} and \emph{regularity} of solutions to the Landau equation. Precisely, as already observed in \cite{chern} in the Coulomb case, suitable bounds on $L^{p}$-norms for solutions to Landau equation imply uniqueness of the solution. While the result is stated in \cite{chern} as a consequence of $L^{\infty}([0,T]\,,\,L^{p}(\R^{d}))$ norms with $p > \frac{3}{2}$ in the Coulomb case $\g=-3,d=3$, our main results Theorem \ref{theo:Coul} and \ref{827} can be seen as suitable conditional result yielding such bounds valid for any dimension $d \geq 3$ and for any choice $\g \in [-d,0)$. Then, as in \cite{chern}, such bounds yield uniqueness of suitable strong solutions. This question is addressed in Section \ref{sec:unique} and the main result of this section can then be stated as follows:

\begin{theo}\label{theo:unique} We consider an integer dimension $d \ge 3$. Let  $\g \in [-d, -2)$ and assume that $2 > \frac{d}{d+\g+2}.$ Let $f=f(t,v), g=g(t,v)$ define two solutions to Eq. \eqref{LFD} in the sense of Definition \ref{defi:weak}, 
%\ref{defi-strong}, 
with 
  $$f(0,\cdot) = g(0,\cdot) = f_{\rm in} \in {L^{2}_{k+2|\g|}(\R^{d})\cap L^{1}_{N}(\R^{d})},$$
where {$k>d$, and $ N$} is defined in Proposition \ref{theo:chern2}. If $f,g$ satisfy one of the two assumptions  \eqref{eq:PS-Cing1} or \eqref{eq:PS-Cingr} if $\g \in (-d,0)$, or \eqref{eq:PS-Cin1} or \eqref{eq:PS-Cinr} in the Coulomb case $\g=-d$,
 then $$f(t)=g(t) \qquad \forall t \in [0,T].$$
\end{theo}

\begin{rmq} We restrict our uniqueness result to the case of very soft potentials $-d \leq \g < -2$ only, since the case of moderate soft potentials $\g \in [-2,0)$ is well-understood (see \cite{ALL,Wu} and the recent contribution \cite{ABL}). Notice that 
because of the condition $2 > \frac{d}{d+\g+2}$,   our analysis is restricted in the Coulomb case to the physical dimension $d=3$. \end{rmq}

The above result can be seen as a generalisation and an improvement on existing conditional uniqueness results for Landau equation. As mentioned already, the uniqueness of weak solutions with bounded mass, momentum, energy, and entropy which further lie in $L^{1}([0, T]; L^{\infty}(\R^{d}))$ (i.e. under condition \eqref{eq:PS-Cin1})  has been established in the Coulomb case in \cite{fournier} via probabilistic arguments,
 whereas the results of \cite{chern} provide uniqueness of solutions to \eqref{eq:Landau} under some $L^{\infty}([0,T]\,,\,L^{p}(\R^{d}))$ bound (and decay) with $p >\frac{3}{2}$. \textit{\textbf{Our uniqueness result generalises these two results and extends them to general soft potentials $\g \in [-d,0)$, $d \geq 3$, and the whole range $L^r([0,T]; L^q(\R^d))$, $\frac2r + \frac{d}q = 2 + d + \g$, with $r \in [1, \infty)$. }}
\medskip

Besides the above uniqueness result, our two main results also allow to deduce suitable \emph{unique continuation and regularity results} for the solutions to Landau equation. Even though we do not elaborate on this point in the present contribution, we mention here that, due to the parabolic nature of the Landau equation, the appearance of $L^{p}$-bounds is of paramount important in order to deduce regularity estimates:
\begin{enumerate}[(1)]
\item Elaborating on the pioneering work of  \cite{degiorgi}  (suitably adapted to spatially homogeneous kinetic equations in \cite{ricardo}), it is a well-established fact that the appearance of (uniform in time) $L^{p}$-bounds (for $p$ large enough) yields the \emph{appearance of $L^{\infty}$-bounds} for solutions to Landau equation. Such a $L^{p}-L^{\infty}$ De Giorgi's argument has been already introduced in the study of Landau equations in \cite{ABDL1,GGL,ABL}. More specifically, one can deduce from Theorems \ref{theo:Coul} or \ref{827}
the following: if $f=f(t,v)$ is a solution to the Landau equation \eqref{eq:Landau} satisfying \eqref{eq:PS-Cin}  {or \eqref{eq:PS-Cingr}} (according to $\g=-d$ or $-d < \g <0$), and suitable $L^{1}$-moments estimates, then for any $t_{\star} \in (0,T)$ there exists $K(t_{\star},T) >0$ such that
\begin{equation}\label{eq:Linfp-intro}
\sup_{t \in [t_{*},T]}\left\|f(t, \cdot)\right\|_{ {\infty}} \leq K(t_{*},T).
\end{equation}
A full proof of such an estimate is given in the case of moderately soft potentials $\g \in [-2,0)$ in \cite{ABL} and in \cite{ABDL} for general $-d \leq \g < -2$. This appearance of $L^{\infty}$-bounds (uniform in time on $[t_{*},T]$) is close to results presented in \cite{silvestre}. While the approach of the latter is based upon general results regarding strong solutions to parabolic equations, the approach in \cite{ABDL,ABL} is, as said, based upon an adaptation of the  approach of  \cite{degiorgi} to spatially homogeneous kinetic equations introduced in \cite{ricardo}. 
\item Suitable $L^{\infty}$-bounds like \eqref{eq:Linfp-intro} can then be combined with parabolic regularising effect to deduce the smoothness of the solution. We refer to \cite[Proposition A.1]{GIV} for a full proof of the fact that, in dimension $d=3$ and for any $\g \in [-3,-2)$, for any weak solution $f=f(t)$ to \eqref{eq:Landau} in the sense of Definition \ref{defi:weak},  it holds
$$f \in L^{\infty}([t_{*},T]\,;\,L^{\infty}(\R^{d})) \implies f \in \mathscr{C}^{\infty}((t_{*},T] \times \R^{d}).$$ 
Therefore, our main results Theorems  \ref{theo:Coul} and \ref{827}, provide conditional results  (in the form of the assumptions \eqref{eq:PS-Cin} or \eqref{eq:PS-Cing}) ensuring the smoothness of the solution to the Landau equation. We recall also that such smoothness provides also \emph{unique continuation} criteria of the solution. We refer to the recent contribution \cite{SnSo} for more details about this question.
\end{enumerate}

The above two points illustrate the importance of devising sufficient conditions yielding the appearance and/or propagation of $L^{p}$-norms for solutions to the Landau equation and, on this basis, we strongly believe that the Prodi-Serrin criteria \eqref{eq:PS-Cin} (in the case $\g=-d$) or \eqref{eq:PS-Cing} (in the case $\g \in (-d,0)$) provide a significant contribution to the conditional uniqueness and regularity of Landau equation, yielding an analogue of the Theorem \ref{theo:PSNS} for Navier-Stokes equation.

We end this description mentioning that the method and ideas developed in this paper can be adapted to derive Prodi-Serrin criterion for the spatially homogeneous Boltzmann equation with soft potentials without cut-off assumptions. We refer to the work in progress \cite{Pierre} for further details.

\subsection{About the class of solutions}

Our main results Theorem \ref{theo:Coul} and \ref{827} have to be interpreted as practical criteria to get \emph{a priori} estimates for weak solutions to the Landau equation in an explicit and quantitative way. In this context, most of the computations that we are providing may be seen as formal, especially since they use one or several integrations by parts. Those estimates can however be fully justified under some extra assumption and, in particular, we strongly believe that our formal estimates are actually valid for \emph{any weak solutions} to the Landau equation, which satisfy the Prodi-Serrin criteria that we stated. We define here weak solutions in the following way:
\begin{defi}\label{weak-sol} Given $d\in \N$, $d \ge 3$, $\g \in [-d,0)$ and $T >0$, let  {$f_{\rm in} \ge 0$ lie in  $L^{1}_{2}(\R^{d}) \cap L\log L(\R^{d})$},  i.e. $f_{\rm in}$ admits finite mass, energy and entropy as defined respectively by
$$\varrho_{\rm in}:=\int_{\R^{d}}f_{\rm in}(v)\d v >0, \qquad E_{\rm in}:=\int_{\R^{d}}f_{\rm in}(v)|v|^{2}\d v,$$
and $$H(f_{\rm in}):=\int_{\R^{d}}f_{\rm in}(v)\log f_{\rm in}(v)\d v.$$
 We say that a family  $f=f(t,v) \ge 0$  is a weak solution to \eqref{eq:Landau} with initial condition $f(0,v)=f_{\rm in}(v)$ if the following hold:
 \begin{enumerate}
\item $f \in \mathcal{C}([0, T ); \mathcal{D}'(\R^{d}))$ and 
\begin{equation} \label{eq0}
\int_{\R^{d}}f(t,v)\left(\begin{array}{c}1 \\v \\ |v|^{2}\end{array}\right)\d v=\int_{\R^{d}}f_{\mathrm{\rm in}}(v)\left(\begin{array}{c}1 \\v \\ |v|^{2}\end{array}\right)\d v\,.\end{equation}
\item One has
\begin{equation}\label{eq:fish}
H(f(t)) \leq H(f_{\rm in}) \qquad \forall t \in [0,T]\,, \quad \text{ and } \quad 
{\int_{0}^{T}\d t}  {\int_{\R^{d}}\left|\nabla \sqrt{f(t,v)}\right|^{2}\langle v\rangle^{\g}\d v < \infty\,.}
\end{equation}
\item For any $\varphi=\varphi(t,v) \in \mathscr{C}_{c}^{2}([0,T)\times \R^{d})$,
\begin{multline}\label{weakform}
-\int_{0}^{T}\d t\int_{\R^{d}}f(t,v)\partial_{t}\varphi(t,v) \d v-\int_{\R^{d}}f_{\mathrm{\rm in}}(v)\varphi(0,v)\d v\\
=\frac{1}{2}\sum_{i,j=1}^{d}\int_{0}^{T}\d t \int_{\R^{d} \times \R^{d}} f(t,v)f(t,w)a_{i,j}(v-w)\left[\partial^{2}_{v_{i},v_{j}}\varphi(t,v)+\partial^{2}_{w_{i},w_{j}}\varphi(t,w)\right]\d v\d w\\+
\sum_{i=1}^{d}\int_{0}^{T}\d t\int_{\R^{d}\times \R^{d}}f(t,v)f(t,w)\bm{b}_{i}(v-w)\left[\partial_{v_{i}}\varphi(t,v)-\partial_{w_{i}}\varphi(t,w)\right]\d v\d w.
\end{multline} 
\end{enumerate}
\end{defi}

\begin{rmq} 
 {Notice that  slightly weaker solutions to Landau equation  have been shown to exist in \cite{villH} for initial data considered in Def. \ref{weak-sol}   (and called \emph{H-solutions}) which are not assumed to satisfy estimate \eqref{eq:fish}. It has been shown later  (when $d=3$)  in  \cite{DesvJFA} that such solutions (when they are built thanks to a suitable approximation process) do satisfy \eqref{eq:fish} and are
in fact weak solutions in the sense of Def. \ref{weak-sol}}. Subsequently, it has been shown in  \cite{GZ} that such solutions  satisfy the estimates
$$\mathcal{A}[f] \in L^{\infty}\left([0,T],L^{\frac{d}{d-2}}_{\mathrm{loc}}(\R^{d})\right), \qquad \bm{b}[f] \in L^{\infty}([0,T]\,;\,L^{\frac{d}{2}}_{\mathrm{loc}}(\R^{d}))$$ and, for any test-function $\phi \in L^{\infty}((0,T)\;,\;\mathbb{W}^{1,\infty}_{c}(\R^{d}))$, it holds
\begin{multline}\label{eq:weak-GZ}\int_{0}^{T}\d t\int_{\R^{d}}\partial_{t}f(t,v)\phi(t,v)\d v \\
+\int_{0}^{T}\d t\int_{\R^{d}}\left(\mathcal{A}[f(t)](v)\nabla f(t,v)-f \bm{b}[f(t)](v)\right) \cdot \nabla \phi(t,v)\d v=0.\end{multline}
Here above $\mathbb{W}^{1,\infty}_{c}(\R^{d})$ denotes the collection of all bounded and compactly supported functions $\phi$  such that $\nabla \phi \in L^{\infty}(\R^{d})$, 
 {and all terms can be well defined.}
 \end{rmq}
\begin{rmq}\label{diffusion} We point out that weak solutions $f=f(t)$ as defined in Definition \ref{weak-sol} are such that there exists a  constant $K_{0} > 0,$ depending on {$H(f_{\mathrm{in}})$ and $\|f_{\rm in}\|_{L^{1}_{2}}$}   such that
$$
\sum_{i,j} \, \mathcal{A}_{i,j}[f(t)](v) \, \xi_i \, \xi_j 
\geq K_{0} \langle v \rangle^{\g} \, |\xi|^2 , \qquad \forall\, v,\, \xi \in \R^{d}, \qquad t \geq0.
$$
We refer to Lemma \ref{diffusionA} in Appendix \ref{sec:known} for a full proof of this estimate.
\end{rmq}

To avoid technical complications and keep the presentation as simple as possible, we rather consider in the paper a class of solutions which already enjoy a sufficient regularity to  justify the computations in the next sections. Such local-in-time solutions have been recently constructed in \cite{DesvHe} in dimension $d=3$, as expressed in the following 
\begin{theo}[\cite{DesvHe}]\label{theo:DeHe}  Let $f_{\rm in} \in L\log L(\R^{3}) \cap L^{1}_{55}(\R^{3}) \cap \dot{\H}^{1}(\R^{3})$ be nonnegative. There exists some explicit $T_{\star} >0$ depending on $\lm_{55}(f_{\rm in}),\|f_{\rm in}\|_{\dot{\H}^{1}}$ and $H(f_{\rm in}),\varrho_{\rm in},E_{\rm in}$ such that \eqref{eq:Landau} admits a unique solution $f$ on the time interval $[0,T_{\star}]$ such that
$$f \in \mathscr{C}([0,T_{\star}],\dot{\H}^{1}(\R^{3})) \cap L^{2}\left([0,T_{\star}]\,,\,\dot{\H}^{2}_{ {\g}}(\R^{3})\right).$$
\end{theo}
Motivated by the above existence and uniqueness result, we define the following notion of solutions that we will use in all the sequel.
\begin{defi}\label{defi:weak}  In all the sequel, given $d\in \N$, $d \ge 3$, $\g \in [-d,0)$ and $T >0$, we will say that an initial datum $f_{\rm in} \geq0$ and a family $f=f(t,v) \geq0$ define a {solution} to \eqref{eq:Landau} if  $f=f(t,v)$ is a weak solution to \eqref{eq:Landau} in the sense of Definition \ref{weak-sol} that further satisfies
$$f \in \mathscr{C}([0,T],\dot{\H}^{1}(\R^{d})) \cap L^{2}\left([0,T]\,,\,\dot{\H}^{2}_{ {\g}}\right).$$
\end{defi} \begin{rmq} The computations in the next section will consist in choosing formally $\phi$ {as a power of $f$}
 as test function in the identity \eqref{eq:weak-GZ}. Under the above additional regularity, we point out that this can be made rigorous simply by choosing, as in \cite{chern}, for the test function $\phi$ in \eqref{eq:weak-GZ} a suitable truncation of 
{a power of $f$} using a cutoff function
$\eta_{R}(v)=\eta\left(R^{-1}v\right)$, $R >1$,
where $\eta$ is a smooth cutoff function identically equal to $1$ on the unit ball of $\R^{d}$ and vanishing on the ball centered at the origin and with radius $2$. 
The computations are then valid for such $\phi$ and all the estimates turn out to be independent of $R$ so, letting $R \to \infty$, we justify our computations. 
\end{rmq}

\subsection{Comments about the work \cite{GS}}\label{sec:GS}

After the completion of this work and the publication of a previous version of our manuscript as a preprint in \cite{ABDL}, an important contribution  \cite{GS} appeared as a preprint, presenting a proof of the fact that classical solutions to \eqref{eq:Landau} do not blow up and can be defined globally.

We still believe that our results remain pertinent to the study of the Landau equation, for reasons that we now explain. The analysis of \cite{GS}, in dimension $d=3$, is based upon the decay of the Fisher information
$$I[f]=4\int_{\R^{d}}\left|\nabla \sqrt{f(v)}\right|^{2}\d v$$
along the flow of classical solutions to \eqref{eq:Landau}, extending  previous partial results (see \cite{toscani,DeVi1,villani,fisher}).  Such a decay is proven through a very clever representation of the Landau operator and suitable functional inequality for the Fisher information of this operator. When applied to classical solutions to \eqref{eq:Landau}, the analysis of \cite{GS} applies to solutions associated with an initial datum $f_{\rm in}$ with finite Fisher information and pointwise Gaussian decay:
\begin{equation}\label{eq:AsGs}
f_{\rm in}(v) \leq \exp\left(-\beta_{0}|v|^{2}\right), \qquad I[f_{\rm in}] < \infty, \qquad \beta_{0} >0\,.\end{equation}
The decay of the Fisher information implies then the following integrability estimates for the considered solution $f=f(t)$:
$$f \in L^{\infty}([0,T],L^{3}(\R^{3})) \cap L^{\infty}([0,T],L^{1}_{k}(\R^{3})), \qquad k \geq 0.$$
Of course, this shows that our Prodi-Serrin criterion is satisfied  and our main results apply to such solutions. In this context, our present result provides \emph{quantitative estimates} on the $L^{p}$-norms for such solutions.

 We also believe that our result is relevant for a wider class of solutions than the one considered in \cite{GS}. 
Indeed pointwise Gaussian decay \emph{is not created} for soft potentials $-3 \leq\g < 0$. In our present contribution, we \emph{do not assume} such a Gaussian decay of the initial datum $f_{\rm in}$, but only some  finite $L^{p}(\R^{d}) \cap L^{1}_{k}(\R^{d})$ norm, for suitable $p >1$, $k >2$. Second, our approach applies to initial data $f_{\rm in}$ with
$I[f_{\rm in}]=\infty$. {Note that a (negatively) weighted Fisher information  is immediately created  for soft potentials $-3 \leq\g < 0$ (see \eqref{eq:fish}) but the rate of appearance is not explicit and  such a result is not known for  the Fisher information itself}.
Our result applies therefore to solutions associated with an initial datum $f_{\rm in}$ with a thick tail decay or rough initial datum, which contrasts with the assumption \eqref{eq:AsGs}. For such initial datum, it is not clear how  the results of \cite{GS} may be applied, and an interesting problem is to see for example  if it is possible to provide criteria allowing to \emph{quantify the rate of appearance} of the Fisher information.

\subsection{Organization of the paper} The proof of Theorem  \ref{theo:Coul}  is given in Section \ref{sec:Coulomb} where the two different criteria \eqref{eq:PS-Cin1} and \eqref{eq:PS-Cinr} are treated separately as well as propagation (point $(a)$) and appearance (point $(b)$) of $L^{p}$-norms.
 
Section \ref{sec:prodi} gives the full proof of Theorem \ref{827} with again the various cases $(a)$, $(b)$ under conditions \eqref{eq:PS-Cing1} or \eqref{eq:PS-Cingr}  treated separately. In particular, the full proof of the $\e$-Poincar\'e Prop. \ref{prop:GG} is given in Section \ref{sec:eps}.

 The stability and uniqueness of solutions to the Landau equation is discussed in Section \ref{sec:unique}, which culminates with the proof of Theorem \ref{theo:unique}.

 A final Section \ref{sec:comm} is devoted to additional features of the Landau equation. Namely, we show in subsection \ref{sec:moderate} how the Prodi-Serrin criteria in Theorem \ref{827} applies to the case of moderate soft potentials $-2 \leq \g <0.$ It is well-known that, for such potentials, propagation/appearance of $L^{p}$-norms occurs as well as the uniqueness and regularity  of the solution (see \cite{ALL,Wu,ABL}) but we believe that Theorem \ref{827} sheds a new light on this case, showing that Prodi-Serrin criteria \eqref{eq:PS-Cing} are met in this case. We also discuss the end-point case $r=\infty$, $q=\frac{d}{d+\g+2}$ of the Prodi-Serrin criterion \eqref{eq:PS-Cinr}--\eqref{eq:PS-Cingr} in subsection \ref{sec:endpoint}, showing an analogue of the end-point case $r=\infty,q=3$ in Theorem \ref{theo:PSNS} for Navier-Stokes equations. We point out here that the role of viscosity $\bm{\nu}$ in Theorem \ref{theo:PSNS} is played by the coercivity constant $K_{0}$ in Remark \ref{diffusion} (see Theorem \ref{sec:endpoint} for details). 

 In Appendix \ref{sec:known} we recall some known results about weak solutions to the Landau equation. Some of them are found in the literature only in dimension $d=3$, so that we also give  some sketches of proof to extend them to arbitrary dimension $d \geq 3$. 

Finally,  Appendix \ref{sec:weight} gives some technical results regarding the evolution of weighted $L^{p}$-norms which are useful for the proof of the uniqueness result Theorem \ref{theo:unique}.

\subsection*{Acknowledgements} B.L. gratefully acknowledges the financial support from the Italian Ministry of Education, University and Research (MIUR), Dipartimenti di Eccellenza grant 2022-2027
as well as the support from the de Castro Statistics Initiative, Collegio Carlo Alberto (Torino).

\medskip

\section{Prodi-Serrin like criteria for Coulomb interactions}\label{sec:Coulomb}

 In this section, we focus on the Landau equation for charged particles associated to Coulomb interactions. In this situation, we recall from \eqref{eq:cgCoul}--\eqref{LFD} that the equation writes
\begin{equation}\label{LFD_c}
\left\{
\begin{array}{ccl}
\;\partial_{t} f &= &\grad \cdot \big(\,\mathcal{A}[f]\, \grad f
- \bm{b}[f]\, f\big)=\mathrm{Tr}\left(\mathcal{A}[f] D^{2}f\right) +c_{d}f^{2} , \medskip\\
\;f(t=0)&=&f_{\mathrm{\rm in}}\, ,
\end{array}\right.
\end{equation}
where $D^{2}f$ is the Hessian matrix of $f$ and $c_{d}:=(d-1)(d-2)|\S^{d-1}|.$

We start with a simple proposition which shows the propagation of $L^p$ norms (including the case when $p=\infty$) under one of the end point of Prodi-Serrin condition:
\begin{prop}\label{without_dissip} 
 Let $f_{\rm in}$ and $f=f(t,v)$ define a  solution to Eq. \eqref{LFD_c} in the sense of Definition \ref{defi:weak}. 
Assume that
$$f \in L^{1}([0,T];L^{\infty}(\R^{d}))$$
for some $T >0$, and $f_{\rm in} \in L^{p}(\R^{d})$
 for some $p\in [1,\infty]$. 
Then $f$ also lies in $L^{\infty}([0,T]; L^p(\R^d))$.
More precisely, the following estimate holds:
$$ \|f\|_{L^{\infty}([0,T], L^p(\R^d))} \le \|f_{\rm in}\|_{p } \,\, \exp \bigg( c_{d} \,  \int_{0}^{T}\|f(t)\|_{\infty}\d t\bigg)\,.$$
\end{prop}

\begin{proof}
Using the
multiplicator $f^{p-1}$ for $p \in [1, \infty[$ in eq. (\ref{LFD_c}), we see that, using the nonnegativity of the matrix $\mathcal{A}[f]$, and the identity  {$-\nabla\cdot \bm{b}[f] = \bm{c}[f] =  {c_{d}}\, f$}, 
\begin{equation*}\begin{split}
\frac1p \frac{\d}{\d t} \int_{\R^d}  f^p \d v&= - (p-1)  \int_{\R^d}  f^{p-2} \nabla f \cdot  \mathcal{A}[f] \, \nabla f \d v+ (p-1) \int_{\R^d}  f^{p-1}  \, \bm{b}[f]\cdot \nabla f \d v\\ 
&\le c_{d}\frac{ (p - 1)}{p} \int_{\R^d}  f^{p+1}  \d v \le c_{d}\frac{(p - 1)}{p} \|f\|_{ \infty } \int_{\R^d} f^p \d v\,.
\end{split}\end{equation*}
Thanks to Gr\"onwall's lemma, we end up with, for $t\in [0,T]$
$$  \int_{\R^d}   f^p (t,v)\,\d v \le \int_{\R^d}   f^p_{\rm in}(v)\,\d v \,\, \exp \bigg( c_{d} (p - 1)\,  \int_0^t \|f(s)\|_{ \infty }\,\d s  \bigg), \qquad \forall t \in [0,T] $$
and consequently
$$ \|f\|_{L^{\infty}([0,T], L^p(\R^d))} \le \|f_{\rm in}\|_{ p } \,\, \exp \bigg(c_{d} \,  \int_0^T \|f(s)\|_{ \infty }\,\d s  \bigg). $$
This concludes the proof when $p \in [1, \infty[$. Passing to the limit when $p \to \infty$ in the above estimate shows that the proposition also holds when $p=\infty$.
\end{proof}

We then show the appearance when $t>0$ of $L^p$ norms  for finite $1 < p < \infty$ (notice that the result  now excludes the case when $p=\infty$) under  one of the end point of Prodi-Serrin condition, assuming that the initial datum has some moments in $L^1$:

\begin{prop}\label{eq_with_dissip} Let $f_{\rm in}$ and $f=f(t,v)$ define a  solution to Eq. \eqref{LFD_c} in the sense of Definition \ref{defi:weak}.  Given $p \in (1,\infty)$, assume that  
$$f_{\rm in} \in L^{1}_{\nu_{p}}(\R^{d}), \qquad \nu_{p}= \frac{d^{2}}{2}\frac{p-1}{p}, $$
and  that 
 $$f \in L^{1}([0,T];L^{\infty}(\R^{d}))$$
for some $T >0$. Then, the solution $f=f(t,v)$ satisfies  the following estimate, for any $t \in (0,T]$:
\begin{equation}\label{eq:LpCp} \|f(t)\|_{p} \le  C_{p}  \bigg[ \sup_{\tau\in [0,T]} {\lm}_{\nu_{p}}(\tau)\bigg]\, t^{- \frac{d}{2} (1 - \frac1{p})}\,, \end{equation}
where 
$$C_{p}=C_p(T) := \bigg[ \frac{4\,K_0}{dpC_{\mathrm{Sob}}} \bigg]^{- \frac{d}{2} (1 - \frac1p)}  \,\exp\left(\frac{d^{2}}{p^{2}}K_0(p-1) \, T +  c_{d}\left(1-\frac{1}{p}\right)\int_{0}^{T}\left\|f(\tau)\right\|_{\infty}\d\tau  \right) ,  $$
$K_0$ is the constant (depending on $\varrho_{\rm in}$, $E_{\rm in}$ and $H(f_{\mathrm{\rm in}})$) appearing in Remark \ref{diffusion},
 and $C_{\mathrm{Sob}}$ is the constant appearing in the 
Sobolev embedding  $\dot{\H}^1(\R^d) \hookrightarrow L^{\frac{2d}{d-2}}(\R^d)$ as defined in \eqref{eq:sobo-p}.
\end{prop} 

\begin{rmq} Recall that, under the assumption $f_{\rm in} \in L^{1}_{\nu_{p}}(\R^{d})$, one has
$$\sup_{\tau\in [0,T]} {\lm}_{\nu_{p}}(\tau) \leq C_{d,p}(1+T), $$
where $C_{d,p}$ is a positive constant depending only on $d,p$ and $\lm_{\nu_{p}}(0)=\|f_{\rm in}\|_{L^{1}_{\nu_{p}}}$ (see Prop. \ref{cor:moments}). 
\end{rmq}

 \begin{proof}
Proceeding as in the proof of Prop. \ref{without_dissip}, we write down
\begin{equation}\label{nbiz}\begin{split} 
\frac1p \dfrac{\d}{\d t} \int_{\R^d}  f^p(t,v)\d v &= - (p-1)  \int_{\R^d}  f^{p-2} \nabla f \cdot  \mathcal{A}[f] \, \nabla f \d v+ (p-1) \int_{\R^d}  f^{p-1}  \, \bm{b}[f]\cdot \nabla f \d v\\
&=  - \frac{4\, (p-1)}{p^2}  \int_{\R^d}   \nabla f^{\frac{p}2}\cdot  \mathcal{A}[f] \,  \nabla f^{\frac{p}2}\d v
 + c_{d}\, \frac{p-1}{p}  \int_{\R^d}  f^{p+1} (t,v)\d v\,,
\end{split}\end{equation}
where we noticed that
$$ \nabla f^{\frac{p}2}\cdot  \mathcal{A}[f] \,  \nabla f^{\frac{p}2} =\frac{p^{2}}{4}f^{p-2}\mathcal{A}[f]\nabla f\cdot \nabla f\,.$$
Using the uniform ellipticity of the diffusion matrix $\mathcal{A}[f]$ (recall Proposition \ref{diffusion}), we deduce that $$ \frac1p \dfrac{\d}{\d t} \int_{\R^d}  f^p(t,v)\d v% - (p-1)  \int_{\R^3}  f^{p-2} \nabla f \cdot  \mathcal{A}[f] \, \nabla f + (p-1) \int_{\R^3}  f^{p-1}  \, \bm{b}[f]\cdot \nabla f $$
  + \frac{4\,K_0\, (p-1)}{p^2}  \int_{\R^d} \langle \cdot \rangle^{-d} \left | \nabla f^{\frac{p}2}(t,v)\right|^2 
% \cdot  \mathcal{A}[f] \, \nabla \left[(f^{\frac{p}2})\right] 
\le c_{d}\, \frac{p-1}p  \int_{\R^d}  f^{p+1}(t,v)\d v\,. $$
Multiplying this inequality by $p$ and  using the elementary inequality 
$$ \left | \nabla\left[ \langle \cdot \rangle^{-\frac{d}{2}}  f^{\frac{p}2} \right]\right|^2 \le 2   \langle \cdot \rangle^{-d} \,\left|\nabla  f^{\frac{p}2}\right|^2 + \frac{d^{2}}{2}\, \langle \cdot \rangle^{-d-2}  \, f^p , $$ this turns into
\begin{multline}\label{bb}
 \dfrac{\d}{\d t} \int_{\R^d}  f^p(t,v)\d v
% - (p-1)  \int_{\R^3}  f^{p-2} \nabla f \cdot  \mathcal{A}[f] \, \nabla f + (p-1) \int_{\R^3}  f^{p-1}  \, \bm{b}[f]\cdot \nabla f $$
  + \frac{2\,K_0\, (p-1)}{p}  \int_{\R^d}  \left| \nabla \left[  \langle v \rangle^{-\frac{d}{2}} \,f^{\frac{p}{2}}(t,v)\right]\right|^2\d v\\ 
% \cdot  \mathcal{A}[f] \, \nabla \left[(f^{\frac{p}2})\right] 
\le c_{d}\, (p-1)  \int_{\R^d}  f^{p+1}(t,v)\d v +    \frac{d^{2}\,K_0\, (p-1)}{p}  \int_{\R^d} \langle  v \rangle^{-d-2}  \, f^p(t,v)\d v . 
\end{multline}
At this point, we combine Sobolev's inequality with a suitable interpolation inequality with weights to conclude. Namely, recall  that Sobolev's inequality reads (with $C_{\mathrm{Sob}}$ the best constant in the Sobolev embedding $\dot{\H}^1(\R^d) \hookrightarrow L^{\frac{2d}{d-2}}(\R^d)$)
\begin{equation}\label{eq:sobo-p}
\left\|\langle \cdot\rangle^{-\frac{d}{p}}f\right\|_{\frac{pd}{d-2}}^{p}=\left\|\langle \cdot\rangle^{-\frac{d}{2}}f^{\frac{p}{2}}\right\|_{{\frac{2d}{d-2}}}^{2} \leq C_{\mathrm{Sob}}\, \int_{\R^{d}}\left|\nabla \left[ \langle v\rangle^{-\frac{d}{2}}f^{\frac{p}{2}}\right]\right|^{2}\d v\,. \end{equation}
Now, we use the standard H\"older interpolation inequality (with weights), which holds for all functions $g$ for which the norms below are defined:
\begin{equation}\label{int-ineqp_bis}
\|\langle \cdot \rangle^{a}g\|_{{r}} \leq \|\langle \cdot \rangle^{a_{1}}g\|_{{r_{1}}}^{\theta}\,\|\langle \cdot \rangle^{a_{2}}g\|_{{r_{2}}}^{1-\theta}\,,
\end{equation}
with $r,r_1,r_2 \ge1$, $a,a_1, a_2 \in \R$,
$$\frac{1}{r}=\frac{\theta}{r_{1}}+\frac{1-\theta}{r_{2}}, \quad a=\theta\,a_{1}+(1-\theta)a_{2},  \quad \theta \in (0,1).$$
Using inequality (\ref{int-ineqp_bis}) for $g := f(t,\cdot)$ with
$$ a :=0, \:\:a_1 :=\nu_{p},\:\: a_2 := -\frac{d}{p}, \qquad r :=p,\:\: r_1 :=1, \:\:r_2 :=\frac{dp}{d-2}, \qquad \theta := \frac2{dp+2-d}, $$ 
we get
\begin{equation*}\begin{split}
\int_{\R^{d}}f^{p}(t,v)\d v &\leq \lm_{a_{1}}(t)^{p\theta}\,\left\|\langle \cdot \rangle^{-\frac{d}{p}}f(t,\cdot)\right\|_{\frac{dp}{d-2}}^{p(1-\theta)}\\
&\leq \lm_{a_{1}}(t)^{p\theta}C_{\mathrm{Sob}}^{1-\theta}\,\bigg(\int_{\R^{d}}\left|\nabla \left[\langle v\rangle^{-\frac{d}{2}}f^{\frac{p}{2}}(t,v)\right]\right|^{2}\d v\bigg)^{1-\theta},
\end{split}\end{equation*}
where we used \eqref{eq:sobo-p} in the last inequality. Recalling now the expression of $a_{1}$ and $\theta$, we use this estimate under the form
\begin{equation}\label{usedt}
   \int_{\R^{d}} \left|\nabla \left[\langle v\rangle^{-\frac{d}{2}}f^{\frac{p}{2}}(t,v)\right]\right|^{2}\d v  \ge   C_{\mathrm{Sob}}^{-1}\, \lm_{ \nu_{p} }(t)^{-\frac{2p}{d(p-1)}}\, \bigg( \int_{\R^d}  f^p(t,v)\d v \bigg)^{1+\frac{2}{d(p-1)}}\,. 
\end{equation}
Plugging this in \eqref{bb}, we end up with
\begin{multline*}
 \dfrac{\d}{\d t} \int_{\R^d}  f^p(t,v)\d v  + \frac{2\,K_0\, (p-1)}{p\,C_{\mathrm{Sob}} }\, \lm_{ \nu_{p} }(t)^{-\frac{2p}{d(p-1)}}\, \bigg( \int_{\R^d}  f^p(t,v)\d v \bigg)^{1+\frac{2}{d(p-1)}} \\
 \leq  \bigg[ c_{d}\, (p-1)  \|f(t)\|_{L^{\infty}(\R^d)} +    \frac{d^{2}\,K_0\, (p-1)}{p} \bigg]\,  \int_{\R^d}  \, f^p(t,v)\d v\,,
 \end{multline*}
after having dropped the weight $ \langle \cdot \rangle^{-d-2} $. Defining 
$$ y(t) :=  \left(\int_{\R^d}  \, f^p(t,v)\, \d v \right)\, \exp\bigg( -(p-1) \int_0^t \bigg[c_{d}  \|f(s)\|_{L^{\infty}(\R^d)} +    \frac{d^{2}\,K_0}{p} \bigg] \, \d s  \bigg) , $$
we see that $y(t)$ satisfies the differential inequality 
$$ y'(t) \le - \frac{2\,K_0\, (p-1)}{p\,C_{\mathrm{Sob}} }\, \lm_{ \nu_{p} }(t)^{-\frac{2p}{d(p-1)}}\,  \, y(t)^{1 + \frac{2}{dp-d}} , $$
which can be solved (for $t \in [0,T]$) as 
$$ y(t) \le \bigg[ \frac{4\,K_0}{dp\,C_{\mathrm{Sob}}}   \bigg(\sup_{\tau \in [0,T]} {\lm}_{ \nu_{p} }(\tau) \bigg)^{-\frac{2p}{d(p-1)}}  \, t \bigg]^{- \frac{d(p-1)}2} . $$
Recalling the definition of $y$, we get the desired estimate.
\end{proof}

\begin{rmq} 
Following the dependency with respect to $p$ in the proof, we can prove that, if $f_{\rm in} \in L^{1}_{\frac{d^{2}}{2}}(\R^{d})$ then, for all $t_{*}, T>0$ and $\kappa>0$ small enough, the following estimate holds:
$$ \sup_{t \in [t_*, T]}\,  \int_{\R^d} \exp\left(\kappa f(t,v)^{\frac2d} \right)\,\d v \le C, $$
 where $C>0$ is a constant
% is bounded above on any interval $[t_*, T]$ by a constant
depending only on $t_*, T$, $\varrho_{\rm in}$, $E_{\rm in}$, $H(f_{\mathrm{\rm in}})$, $d$, $\|f_{in}\|_{L^1_{\frac{d^2}{2}}(\R^d)}$, and 
$\left\|f\right\|_{L^{1}([0,T]\,,\,L^{\infty}(\R^{d}))}$.
\end{rmq}

 We now turn to the general case, in which the Prodi-Serrin condition appears as a weighted $L^r_{t}(L^q_{v})$ norm assumed to be finite. More precisely, we write down the

\begin{prop}\label{mps_coul}
 Let $f_{\rm in}$ and $f=f(t,v)$ define a  solution to Eq. \eqref{LFD_c} in the sense of Definition \ref{defi:weak}. 
Assume that 
\begin{equation}\label{eq:PS-Coul}\langle \cdot \rangle^{d}\,f \in L^{r}([0,T]\,;\,L^{q}(\R^{d})) \:\: \text{ with }\:\: \frac{2}{r}+\frac{d}{q}=2\end{equation}
for some $T >0$, where $r \in (1,\infty)$, $q \in (\frac{d}{2},\infty)$. Given $p \in (1,\infty)$, assume that
$$f_{\rm in} \in L^{1}_{\nu_{p}}(\R^{d}), \qquad \nu_{p} :=\frac{d^{2}}{2}\left(1-\frac{1}{p}\right), $$
then the solution $f=f(t,v)$ satisfies  the following estimate, for any $t \in (0,T]$:

\begin{equation}\label{Lp-Kpq0}
\|f(t)\|_{p} \le K_{p,q} \,\left[\sup_{\tau\in [0,T]} {\lm}_{ \nu_{p}}(\tau)\right]\, t^{- \frac{d}{2} (1 - \frac1{p})} ,\end{equation}
where
$$K_{p,q}= K_{p,q}(T) := \bigg[ \frac{2K_0}{dp\,C_{\mathrm{Sob}} }\ \bigg]^{- \frac{d}{2}\,(1 - \frac1p)} \,
\exp\bigg(\frac{d^{2}}{p^{2}}K_0(p-1) \,T+ \frac1p\, C_{p,q}\,\int_{0}^{T}\left\|\langle \cdot\rangle^{|\g|}f(\tau)\right\|_{q}^{r}\d\tau \bigg)\,,$$
\begin{equation}\label{eq:Cpq} C_{p,q} := \left(1 - \frac{d}{2q}\right)\, \left[c_{d}\, (p-1) \, C_{\mathrm{Sob},q}\right]^{\frac{2q}{2q-d}}\, \left(\frac{K_0\, (p-1)}{p} \frac{2q}d \right)^{\frac{-d}{2q-d}},
\end{equation}
  $K_0$ is the constant (depending on $\varrho_{\rm in}$, $E_{\rm in}$ and $H(f_{\mathrm{\rm in}})$) appearing in Remark \ref{diffusion}, 
%$C_{p,q}$ is defined in \eqref{eq:Cpq},
 $C_{Sob}$ is the Sobolev constant defined in (\ref{eq:sobo-p}),
and
$ C_{\mathrm{Sob},q}$ is the Sobolev constant relative to the embedding 
 $\dot{\H}^{\frac{d}{2q} }(\R^d) \hookrightarrow L^{\frac{2q}{q-1}}(\R^d)$ .
\end{prop}

\begin{proof}
We start from inequality \eqref{bb} obtained in the proof  of Prop. \ref{eq_with_dissip} and we estimate the term
$$\int_{\R^{d}}f^{p+1}(t,v)\d v=\int_{\R^{d}}\left[\langle v\rangle^{d}f(t,v)\right]\,\left[\langle v \rangle^{-\frac{d}{2}}f^{\frac{p}{2}}(t,v)\right]^{2}\d v$$
thanks to H\"older's inequality. Denoting  $q'=\frac{q}{q-1}$, where $q$ is the exponent in the Prodi-Serrin condition, we obtain
$$\int_{\R^{d}}f^{p+1}(t,v)\d v \leq \,\left\|\langle \cdot\rangle^{d}f(t)\right\|_{q}\,\left\|\langle \cdot \rangle^{-\frac{d}{2}} \,f^{\frac{p}{2}}(t)\right\|_{2q'}^{2}\,.$$
Using now the inequality $\|\cdot\|_{2q'}^2 \leq C_{\mathrm{Sob},q}\|\cdot\|_{\dot{\H}^{\frac{d}{2q}}}^2$ resulting from the Sobolev embedding  $\dot{\H}^{\frac{d}{2q} }(\R^d) \hookrightarrow L^{\frac{2q}{q-1}}(\R^d)$ ) (see \cite{chemin}, Theorem 1.38),
 we first deduce that
$$\int_{\R^{d}}f^{p+1}(t,v)\d v \leq C_{\mathrm{Sob},q}\left\|\langle \cdot\rangle^{d}f(t)\right\|_{q}\,\left\|\langle \cdot \rangle^{-\frac{d}{2}} \,f^{\frac{p}{2}}(t)\right\|_{\dot{\H}^{\frac{d}{2q}}}^{2}\,.$$
Observing now that, since $q > \frac{d}{2}$, $0 < \frac{d}{2q} < 1$, we can invoke   the  interpolation inequality $\dot{\H}^1(\R^d) \cap  L^2(\R^d) \subset \dot{\H}^{\frac{d}{2q} }(\R^d)$ (see \cite{chemin}, Proposition 1.32) to deduce now that
$$\left\|\langle \cdot \rangle^{-\frac{d}{2}} \,f^{\frac{p}{2}}(t)\right\|_{\dot{\H}^{\frac{d}{2q}}}^{2}\leq \left\|\langle \cdot \rangle^{-\frac{d}{2}} \,f^{\frac{p}{2}}(t)\right\|_{2}^{2-\frac{d}{q}}\,\left\|\nabla \left[\langle \cdot \rangle^{-\frac{d}{2}} \,f^{\frac{p}{2}}(t)\right]\right\|_{2}^{\frac{d}{q}}\,.$$
Plugging this in inequality \eqref{bb}, we obtain
\begin{multline}\label{bb_bis}
 \dfrac{\d}{\d t} \int_{\R^d}  f^p(t,v)\d v
% - (p-1)  \int_{\R^3}  f^{p-2} \nabla f \cdot  \mathcal{A}[f] \, \nabla f + (p-1) \int_{\R^3}  f^{p-1}  \, \bm{b}[f]\cdot \nabla f $$
  + \frac{2\,K_0\, (p-1)}{p}  \int_{\R^d}  \left| \nabla \left[  \langle v \rangle^{-\frac{d}{2}} \,f^{\frac{p}{2}}(t,v)\right]\right|^2\d v\\ 
% \cdot  \mathcal{A}[f] \, \nabla \left[(f^{\frac{p}2})\right] 
%\le c_{d}\, (p-1)  \int_{\R^d}  f^{p+1}(t,v)\d v +    \frac{d^{2}\,K_0\, (p-1)}{p}  \int_{\R^d} \langle  v \rangle^{-d-2}  \, f^p(t,v)\d v \\
\le c_{d}\,(p-1)\,C_{\mathrm{Sob},q}\left\|\langle \cdot\rangle^{d}f(t)\right\|_{q}\,\left\|\langle \cdot \rangle^{-\frac{d}{2}} \,f^{\frac{p}{2}}(t)\right\|_{2}^{2-\frac{d}{q}}\,\left\|\nabla \left[\langle \cdot \rangle^{-\frac{d}{2}} \,f^{\frac{p}{2}}(t)\right]\right\|_{2}^{\frac{d}{q}}\\
+    \frac{d^{2}\,K_0\, (p-1)}{p}  \int_{\R^d}\langle v\rangle^{-d-2} f^p(t,v)\d v \,.
\end{multline}
We now resort to Young's inequality (for $x,y \ge 0, \, \e>0$) in the following form:
$$ x\,y \le  \left(\frac{2q}d \right)^{-1}\, (\e x)^{\frac{2q}d} + \left(\frac{2q}{2q-d} \right)^{-1}\, \left(\frac{y}{\e}\right)^{\frac{2q}{2q-d}}, $$
with the choice
$$ x:=\left\|\nabla \left[\langle \cdot \rangle^{-\frac{d}{2}} \,f^{\frac{p}{2}}(t)\right]\right\|_{2}^{\frac{d}{q}}, \qquad y :=  c_{d}\,(p-1)\,C_{\mathrm{Sob},q}\left\|\langle \cdot\rangle^{d}f(t)\right\|_{q}\,\left\|\langle \cdot \rangle^{-\frac{d}{2}} \,f^{\frac{p}{2}}(t)\right\|_{2}^{2-\frac{d}{q}}, $$
and 
$$ \e := \bigg[\frac{K_0\, (p-1)}{p} \frac{2q}d \bigg]^{\frac{d}{2q}}\,.$$
Inserting this into \eqref{bb_bis}, we obtain
\begin{multline*} 
 \dfrac{\d}{\d t} \int_{\R^d}  f^p(t,v)\d v
% - (p-1)  \int_{\R^3}  f^{p-2} \nabla f \cdot  \mathcal{A}[f] \, \nabla f + (p-1) \int_{\R^3}  f^{p-1}  \, \bm{b}[f]\cdot \nabla f $$
  + \frac{K_0\, (p-1)}{p}  \int_{\R^d}  \left| \nabla \left[  \langle v \rangle^{-\frac{d}{2}} \,f^{\frac{p}{2}}(t,v)\right]\right|^2\d v\\ 
% \cdot  \mathcal{A}[f] \, \nabla \left[(f^{\frac{p}2})\right] 
%\le c_{d}\, (p-1)  \int_{\R^d}  f^{p+1}(t,v)\d v +    \frac{d^{2}\,K_0\, (p-1)}{p}  \int_{\R^d} \langle  v \rangle^{-d-2}  \, f^p(t,v)\d v \\
\le \left(1-\frac{d}{2q}\right)\left[c_{d}\,(p-1)\,C_{\mathrm{Sob},q}\e^{-1}\right]^{\frac{2q}{2q-d}}\left\|\langle \cdot\rangle^{d}f(t)\right\|_{q}^{\frac{2q}{2q-d}}\,\left\|\langle \cdot \rangle^{-\frac{d}{2}} \,f^{\frac{p}{2}}(t)\right\|_{2}^{2}\\+    \frac{d^{2}\,K_0\, (p-1)}{p}  \int_{\R^d}\langle v\rangle^{-d-2} f^p(t,v)\d v \,.
\end{multline*}
Observing that $\left\|\langle \cdot \rangle^{-\frac{d}{2}} \,f^{\frac{p}{2}}(t)\right\|_{2}^{2}=\displaystyle\int_{\R^{d}}\langle v\rangle^{-d}f^{p}(t,v)\d v$, we can deduce from  the above that
\begin{multline}\label{eq:LpCoul}
 \dfrac{\d}{\d t} \int_{\R^d}  f^p(t,v)\d v
% - (p-1)  \int_{\R^3}  f^{p-2} \nabla f \cdot  \mathcal{A}[f] \, \nabla f + (p-1) \int_{\R^3}  f^{p-1}  \, \bm{b}[f]\cdot \nabla f $$
  + \frac{K_0\, (p-1)}{p}  \int_{\R^d}  \left| \nabla \left[  \langle v \rangle^{-\frac{d}{2}} \,f^{\frac{p}{2}}(t,v)\right]\right|^2\d v\\ 
% \cdot  \mathcal{A}[f] \, \nabla \left[(f^{\frac{p}2})\right] 
%\le c_{d}\, (p-1)  \int_{\R^d}  f^{p+1}(t,v)\d v +    \frac{d^{2}\,K_0\, (p-1)}{p}  \int_{\R^d} \langle  v \rangle^{-d-2}  \, f^p(t,v)\d v \\
\le \bigg[C_{p,q}\, \|\langle \cdot\rangle^{d}\, f(t)\|_{q}^{\frac{2q}{2q-d}}+\frac{d^{2}\,K_0\, (p-1)}{p}\bigg]\int_{\R^{d}}\langle v\rangle^{-d}f^{p}(t,v)\d v,\end{multline}
where
\begin{equation} C_{p,q} := \left(1 - \frac{d}{2q}\right)\, \left[c_{d}\, (p-1) \, C_{\mathrm{Sob},q}\right]^{\frac{2q}{2q-d}}\, \left(\frac{K_0\, (p-1)}{p} \frac{2q}d \right)^{\frac{-d}{2q-d}}. \end{equation}
We now plug estimate \eqref{usedt} in this inequality to deduce (neglecting the negative weight in the last integral):
\begin{multline*}
 \dfrac{\d}{\d t} \int_{\R^d}  f^p(t,v)\d v  + \frac{K_0\, (p-1)}{p\,C_{\mathrm{Sob}} }\, \lm_{ \nu_{p} }(t)^{-\frac{2p}{d(p-1)}}\, \bigg( \int_{\R^d}  f^p(t,v)\d v \bigg)^{1+\frac{2}{d(p-1)}} \\
 \leq \bigg[C_{p,q}\, \|\langle \cdot\rangle^{d}\, f(t)\|_{q}^{\frac{2q}{2q-d}}+\frac{d^{2}\,K_0\, (p-1)}{p}\bigg]\int_{\R^{d}} f^{p}(t,v)\d v.\end{multline*}
 We now proceed like at the end of the proof of Prop. \ref{eq_with_dissip}, introducing now
$$ y(t) := \int_{\R^d} \, f^p(t,v)\,\d v \, \exp\bigg( - \int_0^t  \bigg[C_{p,q}\, \|\langle \cdot\rangle^{d}\, f(s,\cdot) \|_{q}^{\frac{2q}{2q-d}} + \frac{d^{2}\,K_0\, (p-1)}{p} \bigg]\, \d s \bigg), $$
and check that
$$ y'(t) \le - \frac{K_0\, (p-1)}{p\,C_{\mathrm{Sob}} }\, \lm_{ \nu_{p} }(t)^{-\frac{2p}{d(p-1)}} \,y(t)^{1+ \frac2{d(p-1)}} , $$
so that (for $t \in [0,T]$) 
$$ y(t) \le \bigg[  \frac{2K_0}{dp\,C_{\mathrm{Sob}} }\, \bigg(\sup_{\tau\in [0,T]}{\lm}_{\nu_{p}}(\tau) \bigg)^{-\frac{2p}{d(p-1)}}  \, t \bigg]^{- \frac{d(p-1)}2} . $$
Recalling the definition of $y$, we get estimate \eqref{Lp-Kpq0} after observing that  $r = \frac{2q}{2q-d}$.
\end{proof}

%  \begin{rmq}
%If the condition on the exponent $q,r$  is $\frac2r + \frac3q <2$ instead of $\frac2r + \frac3q =2$ , we still expect to get $L^{\infty}_{t}(L^p)$ estimates for the solution to Eq. \eqref{LFD_c}, but possibly with different assumptions on moments. \textcolor{blue}{Bert: I suggest to remove this remark if we include the section about Prodi-Serrin \emph{inequality}}.
%\end{rmq}

\begin{rmq} In both Proposition \ref{eq_with_dissip} and Proposition \ref{mps_coul}, whenever $p \le \frac{d^{2}}{d^{2}-4}$, we observe that $\frac{d^{2}}{2}\left(1 - \frac{1}p\right)\le 2$, so that no extra moment (beyond the kinetic energy) for the initial datum is required in this case. \end{rmq}

One can also deduce from the proof above the following propagation result:

\begin{cor} \label{cor:LinftyLP}
 Let $f_{\rm in}$ and $f=f(t,v)$ define a  solution to Eq. \eqref{LFD_c} in the sense of Definition \ref{defi:weak}. 
Assume that  \eqref{eq:PS-Coul} holds for some $T >0$, where $r \in (1,\infty)$, $q \in (\frac{d}{2},\infty)$. Given $p \in (1,\infty)$, assume that
$f_{\rm in} \in  L^{p}(\R^{d})$.
Then, the solution $f=f(t,v)$
% \in L^{\infty}([0,T]\,;\,L^{p}(\R^{d}))$ with
satisfies the estimate 
$$\|f\|_{L^{\infty}([0,T]\,;\,L^{p}(\R^{d}))} \leq \|f_{\rm in}\|_{p}\,
\exp\left(\frac{C_{p,q}}{p}\int_{0}^{T}  {\|\langle \cdot \rangle^d\, f(t)\|_{q}^{r} } \, \d t+ \frac{d^{2}\,K_0\, (p-1)}{p^{2}}T\right), $$
where $C_{p,q}$ is defined in \eqref{eq:Cpq}.
\end{cor}

\begin{proof}  {We start from inequality \eqref{eq:LpCoul} for the evolution of the $L^{p}(\R^{d})$-norm of $f$ (note that this inequality does not use the assumption that $L^1$-moments are initially finite).} Then, dropping 
the weight $\langle \cdot \rangle^{-d}$ and the nonnegative term involving the integral of the gradient on the right-hand-side, we deduce that
$$ \dfrac{\d}{\d t} \int_{\R^d}  f^p(t,v)\d v \leq  \bigg[C_{p,q}\, \|\langle \cdot\rangle^{d}\, f(t, \cdot)\|_{q}^{r}+\frac{d^{2}\,K_0\, (p-1)}{p}\bigg]\int_{\R^{d}}f^{p}(t,v)\d v, $$
where $C_{p,q}$ is given by \eqref{eq:Cpq}. According to Gronwall's Lemma, one deduces that, for all $ t \in [0,T]$,
$$
 \int_{\R^d}  f^p(t,v)\d v \leq \left( \int_{\R^{d}}f_{\rm in}^{p}(v)\d v \right)\exp\left(\int_{0}^{t} \bigg[C_{p,q}\, \|\langle \cdot\rangle^{d}\, f(\tau)\|_{q}^{r}+\frac{d^{2}\,K_0\, (p-1)}{p}\bigg]\d\tau\right),$$
 from which the result easily follows. 
 \end{proof}

\begin{proof} of Theorem \ref{theo:Coul}:  It is a direct consequence of  Proposition \ref{without_dissip} and Proposition \ref{eq_with_dissip}  when  the $L^{1}_{t}(L^{\infty}_{v})$ condition is considered, and of  Proposition \ref{mps_coul} and Corollary \ref{cor:LinftyLP} when the $L^{r}_{t}(L^{q}_{v})$ condition is considered.
\end{proof}

\section{Prodi-Serrin like criteria for soft potentials $-d < \g <0$} \label{sec:prodi}
We present here an analysis, similar to the one of the previous section, for the Landau equation with soft potentials $-d < \g <0.$ The main difference between the Coulomb case treated in Section \ref{sec:Coulomb} and the present one is the nature of the lower order term 
$\bm{c}_{\g}[f]$ which is now a convolution operator. The analysis of this term requires then the specific use of the Hardy-Littlewood-Sobolev inequality which we recalls here for the sake of completeness:

\begin{prop}[\textit{\textbf{Hardy-Littlewood-Sobolev inequality}}] \label{prop:HLS}
Let $d\in \N$, $d\ge 1$, $1 < r ,p < \infty$ and $0 < \lambda < d$ with
$$\frac{1}{p}+\frac{\lambda}{d}+\frac{1}{r}=2.$$
Then there exists $C_{\rm{HLS}}>0$ (depending on $d,p,\lambda$) such that  the estimate
\begin{equation}\label{eq:HLS}
\int_{\R^{2d}}g(x)|x-y|^{-\lambda}h(y)\d x \d y \leq C_{\rm{HLS}}\, \|g\|_{ {p}}\, \|h\|_{ {r}} 
\end{equation}
holds for any smooth $g,h\::\:\R^{d}\to \R$.

 \end{prop}
 
 \subsection{Appearance of $L^{p}$-norms - the $L^{1}_{t}(L^{\frac{d}{d+\g}})$ case}
 
As in the Coulomb case, we are already in position to prove the propagation of $L^p$ norms  for $p < \frac{d}{d+\g}$ under one of the end point of Prodi-Serrin condition:
\begin{prop}\label{without_dissip-g}
{We consider $d\in \N$, $d \ge 3$ and 
 $\g \in (-d, 0)$.}
 Let $f_{\rm in}$ and $f=f(t,v)$ define a  solution to Eq. \eqref{LFD_c} in the sense of Definition \ref{defi:weak}. 
Assume that
\begin{equation}\label{eq:L1-criterion}
f  \in L^{1}\left([0,T]\;;\;L^{\frac{d}{d+\g}}(\R^{d})\right),  
\end{equation}
for some $T >0$, and $f_{\rm in} \in L^{p}(\R^{d})$
 for some $p\in (1,\frac{d}{d+\g})$. 
Then
 $f$ also lies in $L^{\infty}([0,T]; L^p(\R^d))$.
More precisely, the following estimate holds:
$$ \|f\|_{L^{\infty}([0,T], L^p(\R^d))} \le \|f_{\rm in}\|_{p } \,\, \exp \bigg(C_{d,\g,p}\,  \int_{0}^{T}\|f(t)\|_{\frac{d}{d+\g}}\d t\bigg)\,,$$
where $C_{d,\g,p} >0$ is 
an explicit constant depending only on $d,\g,p$.
\end{prop}

\begin{proof} Since only propagation of $L^{p}$-norms are involved here, we simply notice that
\begin{equation}\label{eq:Mp}\begin{split}
\frac1p \frac{\d}{\d t} \int_{\R^d}  f^p \d v&= - (p-1)  \int_{\R^d}  f^{p-2} \nabla f \cdot  \mathcal{A}[f] \, \nabla f \d v+ (p-1) \int_{\R^d}  f^{p-1}  \, \bm{b}[f]\cdot \nabla f \d v\\
&=-(p-1)\int_{\R^d}  f^{p-2} \nabla f \cdot  \mathcal{A}[f] \, \nabla f \d v+\frac{p-1}{p}\int_{\R^{d}}  f^p \bm{c}_{\g}[f]\, \d v\, ,
\end{split}\end{equation}
so
%from \eqref{nbiz} 
that
$$\dfrac{\d}{\d t}\int_{\R^{d}}f^{p}(t,v)\d v \leq  { (p-1)\int_{\R^{d}}\bm{c}_{\g}[f(t)](v)f^{p}(t,v)\d v} , $$
and we only need to estimate this last integral. Using Hardy-Littlewood-Sobolev inequality \eqref{eq:HLS} with $g= f(t,\cdot) \in L^{p}$, $h=f^{p}(t,\cdot)$ and $\lambda = - \gamma$,  we observe that $p < \frac{d}{d+\g}$ implies 
$$\frac{1}{r}=2+\frac{\g}{d}-\frac{1}{p} < 1,$$
and conclude that
\begin{equation*}\begin{split}
 { \int_{\R^{d}} {\bm{c}}_{\g}[f(t)](v)f^{p}(t,v)\d v } &\leq C_{d,\g,p}\|f(t)\|_{{p}}
\|f^{p}(t)\|_{{r}}=C_{d,\g,p}\|f(t)\|_{{p}}\|f(t)\|_{{pr}}^{p},
\end{split}\end{equation*}
where $C_{d,\g,p} :=(d-1)(d+\g)\,C_{\rm{HLS}}$ is depending only on 
$d,\g,p$. Writing $q_{0}=\frac{d}{d+\g}$ and recalling that $p < \frac{d}{d+\g}$, we use an interpolation based on H\"older's inequality, namely,
$$
\|f(t)\|_{{pr}}^{p} \leq \| f(t)\|_{{p}}^{p\theta}\| f(t)\|_{{q_{0}}}^{p(1-\theta)},
$$
with 
$$\frac{1}{pr}=\frac{\theta}{p}+\frac{1-\theta}{q_{0}}, \qquad \theta=\frac{q_{0}-pr}{r(q_{0}-p)}\,.$$
Observing that
$\frac{1}{r}=\frac{1}{q_{0}}+1-\frac{1}{p},$
one actually has 
$\theta=1-\frac{1}{p},$ so that $p\,(1-\theta)=1$ and
\begin{equation}\label{eq:CgLambda0}
 {\int_{\R^{d}}{\bm{c}}_{\g}[f(t)](v)f^{p}(t,v)\d v } \leq C_{d,\g,p}\, \| f(t)\|_{{p}}^{p}\| f(t)\|_{{q_{0}}}={\Lambda}_{0}(t)\int_{\R^{d}}f^{p}(t,v)\d v,\end{equation}
with 
$$ {\Lambda}_{0}(t) := C_{d,\g,p}\,\|f(t)\|_{{q_{0}}} \in L^{1}([0,T])$$
by assumption. Thanks to Gr\"onwall's lemma, we obtain that, for $t\in [0,T]$,
$$  \int_{\R^d}   f^p (t,v)\,\d v \le \left(\int_{\R^d}   f^p_{\rm in}(v)\,\d v \right)\,\, \exp \bigg( C_{d,\g,p} (p - 1)\,  \int_0^t \|f(s)\|_{q_{0}}\,\d s  \bigg), \qquad \forall t \in [0,T] $$
which gives the desired result.
\end{proof}
In order to show the \emph{appearance} of $L^{p}$-norms, we need to investigate more carefully the evolution of  $L^{p}$-norms for solutions to the Landau equation \eqref{LFD}.
We use in the sequel the notation, for $k\in \R$ and $p \in (1,\infty)$,
$$
\lM_{k,p}(t) :=\int_{\R^{d}}f(t,v)^{p}\langle v\rangle^{k}\d v, \qquad \lD_{{k}, p}(t) :=\int_{\R^{d}}\left|\nabla \left(\langle v\rangle^{\frac{k}{2}}f^{\frac{p}{2}}(t,v)\right)\right|^{2}\d v\,,
$$
together with the shorthand notation $\lM_{p}(t) :=\lM_{0,p}(t)$.

 \begin{lem}\label{prop:estMs-g}
We consider $d\in \N$, $d \ge 3$ and 
 $\g \in [-d, 0)$, and assume that $f_{\rm in}$ and $f=f(t,v)$ define a  solution to Eq. \eqref{eq:Landau} in the sense of Definition \ref{defi:weak}.  Then for all $p \in (1,\infty)$, if \begin{equation}\label{eq:nuk}
f_{\rm in} \in L^{1}_{\nu_{\g,p}}(\R^{d}), \qquad \nu_{\g,p}:= \frac{|\g| d}{2}\left(1-\frac{1}{p}\right)\,,
\end{equation}
it holds
\begin{multline}\label{eq:LMp}
\dfrac{\d}{\d t}\lM_{p}(t) + \frac{2K(p)}{C_{\mathrm{Sob}}}\lm_{\nu_{\g,p}}(t)^{-\frac{2p}{d(p-1)}}\lM_{p}(t)^{1+\frac{2}{d(p-1)}} \leq K(p)\g^{2}  {\lM_{p}(t)}\\
+(p-1) \int_{\R^{d}} \bm{c}_{\g}[f(t)](v)\,f^p(t,v) \d v,
\end{multline}
where $K(p) :=\frac{p-1}{p}K_{0}$, $K_0$ being the constant appearing in Remark \ref{diffusion} and $C_{\mathrm{Sob}}$ is the Sobolev constant appearing in the 
Sobolev embedding  $\dot{\H}^1(\R^d) \hookrightarrow L^{\frac{2d}{d-2}}(\R^d)$ (see \eqref{eq:sobo-p}). 
\end{lem}

\begin{proof} As in the proof of the previous proposition, we start with
\begin{equation}\begin{split}\label{ts37}
\frac1p \frac{\d}{\d t} \int_{\R^d}  f^p \d v&= - (p-1)  \int_{\R^d}  f^{p-2} \nabla f \cdot  \mathcal{A}[f] \, \nabla f \d v+ (p-1) \int_{\R^d}  f^{p-1}  \, \bm{b}[f]\cdot \nabla f \d v\\
&=-(p-1)\int_{\R^d}  f^{p-2} \nabla f \cdot  \mathcal{A}[f] \, \nabla f \d v+\frac{p-1}{p}\int_{\R^{d}}  f^p \bm{c}_{\g}[f]\, \d v\, ,
\end{split}\end{equation}
where we recall that $\nabla \cdot \bm{b}[f]=-\bm{c}_{\g}[f].$ Now, as previously observed,
$$(p-1)\int_{\R^{d}} f^{p-2}\mathcal{A}[f]\nabla f\cdot   \nabla f\d v 
 \geq \frac{4K_{0}(p-1)}{p^{2}} \int_{\R^{d}}\langle v\rangle^{\g}\,\left|\grad (f^{\frac{p}{2}})\right|^{2}\d v\,.$$
Moreover, writing
\begin{equation*}
\nabla \left(\langle v\rangle^{\frac{\g}{2}}\,f^{\frac{p}{2}}\right)=\langle v\rangle^{\frac{\g}{2}}\nabla (f^{\frac{p}{2}}) + \frac{\g}{2}v\,\langle v\rangle^{\frac{\g}{2}-2}f^{\frac{p}{2}}\,,
\end{equation*}
which implies
\begin{equation}\label{eq:Gradient}
\langle v\rangle^{\g}\left|\nabla (f^{\frac{p}{2}})\right|^{2} \geq \frac{1}{2}\left|\nabla \left(\langle v\rangle^{\frac{\g}{2}}f^{\frac{p}{2}}\right)\right|^{2} - \frac{\g^{2}}{4} \langle v\rangle^{\g-2}f^{p}, \end{equation}
we observe that
\begin{multline*}
(p-1)\int_{\R^{d}} f^{p-2}\mathcal{A}[f]\nabla f\cdot   \nabla f\d v 
 \geq \frac{2K_{0}(p-1)}{p^{2}} \int_{\R^{d}}\left|\nabla \left(\langle v\rangle^{\frac{\g}{2}}f^{\frac{p}{2}}\right)\right|^{2}\d v\\
 -\frac{K_{0}(p-1)\g^{2}}{p^{2}}\int_{\R^{d}} \langle v\rangle^{\g-2}f^{p}\d v\,.\end{multline*}
Inserting this  {into \eqref{ts37}}, we obtain
\begin{equation}\label{eq:dtMsp-g}
\dfrac{\d}{\d t}\lM_{p}(t) + 2K(p) {\lD_{\gamma, p}(t)} \leq K(p)\g^{2}\lM_{p}(t)\\
+ (p-1) \int_{\R^{d}} \bm{c}_{\g}[f(t)](v)\,f^p(t,v) \d v ,
\end{equation}
where we simply observe that $\langle v\rangle^{\g-2} \leq 1.$ 

Now, as in the proof of Proposition \ref{eq_with_dissip}, we combine a Sobolev inequality with a suitable interpolation inequality with weights to estimate ${\lD_{\gamma, p}(t)}$ in terms of $\lm_{\nu_{\g,p}}(t)$ and $\lM_{p}(t)$. Namely, one can reformulate the Sobolev embedding inequality \eqref{eq:sobo-p} as
\begin{equation}\label{eq:sobo-p-g}
\left\|\langle \cdot\rangle^{\frac{\g}{p}}f\right\|_{\frac{pd}{d-2}}^{p}=\left\|\langle \cdot\rangle^{\frac{\g}{2}}f^{\frac{p}{2}}\right\|_{{\frac{2d}{d-2}}}^{2} \leq C_{\mathrm{Sob}}\, \int_{\R^{d}}\left|\nabla \left[ \langle v\rangle^{\frac{\g}{2}}f^{\frac{p}{2}}\right]\right|^{2}\d v=C_{\mathrm{Sob}}{\lD_{\gamma, p}(t)}\,. \end{equation}
We now use the interpolation inequality \eqref{int-ineqp_bis} with the choice
$$a=0, \quad a_{1}=\nu_{\g,p}, \quad a_{2}=\frac{\g}{p}, \qquad r :=p,\:\: r_1 :=1, \:\:r_2 :=\frac{dp}{d-2}, \qquad \theta := \frac2{dp+2-d} . $$ 
We can reformulate \eqref{usedt} as
\begin{equation}\label{usedt-g}
   \int_{\R^{d}} \left|\nabla \left[\langle v\rangle^{\frac{\g}{2}}f^{\frac{p}{2}}(t,v)\right]\right|^{2}\d v  \ge   C_{\mathrm{Sob}}^{-1}\, \lm_{ \nu_{\g,p} }(t)^{-\frac{2p}{d(p-1)}}\, \bigg( \int_{\R^d}  f^p(t,v)\d v \bigg)^{1+\frac{2}{d(p-1)}}\,. 
\end{equation}
Plugging this into \eqref{eq:dtMsp-g}, we get the result.
\end{proof}

Thanks to the above evolution, we can now show the analogue of Prop. \ref{eq_with_dissip} and prove the \emph{appearance} of $L^{p}$-norms, for $p < \frac{d}{d+\g}$ under the endpoint Prodi-Serrin criterion \eqref{eq:L1-criterion}:

\begin{prop}\label{prop:L1-criterion} 
We consider $d\in \N$, $d \ge 3$ and 
 $\g \in (-d, 0)$, and assume that $f_{\rm in}$ and $f=f(t,v)$ define a  solution to Eq. \eqref{eq:Landau} in the sense of Definition \ref{defi:weak}. Assume that
\begin{equation*}
f  \in L^{1}\left([0,T]\;;\;L^{\frac{d}{d+\g}}(\R^{d})\right),  
\end{equation*}
for some $T >0$. Assume moreover that 
$$f_{\mathrm{\rm in}} \in L^{1}_{\nu_{\g,p}}(\R^{d}), \qquad p \in \left(1,\frac{d}{d+\g}\right), $$
with $\nu_{\g,p}$ defined by (\ref{eq:nuk}). 
{Then, for any $t \in (0,T]$}, 
\begin{equation}\label{eq:LpCp-g} \|f(t)\|_{p} \le  C_{\g,p}\, t^{- \frac{d}{2} (1 - \frac1{p})} \,\sup_{\tau\in [0,T]} {\lm}_{\nu_{\g,p}}(\tau)\,, \end{equation}
where 
$$ {C_{\gamma,p}}=C_{\g,p}(T)= \bigg[ \frac{4\,K_0}{dpC_{\mathrm{Sob}}} \bigg]^{- \frac{d}{2} (1 - \frac1p)} \,\, \exp \bigg(\frac{\g^{2}}{p^{2}}K_0(p-1)T+ C_{d,\g,p}\left(1-\frac{1}{p}\right)\int_{0}^{T}\left\|f(\tau)\right\|_{\frac{d}{d+\g}}\d\tau  \bigg),  $$
 $K_0$ is the constant (depending on $\varrho_{\rm in}$, $E_{\rm in}$ and $H(f_{\mathrm{\rm in}})$) appearing in Remark \ref{diffusion}, $C_{\mathrm{Sob}}$ is the constant appearing in the 
Sobolev embedding  $\dot{\H}^1(\R^d) \hookrightarrow L^{\frac{2d}{d-2}}(\R^d)$ (see \eqref{eq:sobo-p}),
 and $C_{d,\g.p}$ is the positive constant (depending on $d, \g$ and $p$) appearing in Proposition \ref{without_dissip-g}.
\end{prop}

 \begin{proof} We already observed (see \eqref{eq:CgLambda0}) that, under assumption \eqref{eq:L1-criterion},
 $$\int_{\R^{d}}\bm{c}_{\g}[f(t)](v)f^{p}(t,v)\, \d v \leq \Lambda_{0}(t)\,\|f(t)\|_{p}^{p}=\Lambda_{0}(t)\lM_{p}(t),$$
 where $\Lambda_{0}(t) :=C_{d,\g,p}\|f(t)\|_{\frac{d}{d+\g}}$ and $1 < p< \frac{d}{d+\g}.$ Inserting this in the evolution of $\lM_{p}(t)$ given in \eqref{eq:LMp}, we deduce that
 \begin{multline}\label{eq:LMp0}
\dfrac{\d}{\d t}\lM_{p}(t) + \frac{2K(p)}{C_{\mathrm{Sob}}}\lm_{\nu_{\g,p}}(t)^{-\frac{2p}{d(p-1)}}\lM_{p}(t)^{1+\frac{2}{d(p-1)}} \leq \left(K(p)\g^{2}+(p-1)\Lambda_{0} (t)\right)\lM_{p}(t),
\end{multline}
with $\Lambda_{0} \in L^{1}([0,T]).$ As in the proof of Prop. \ref{eq_with_dissip}, we can now define
$$ y(t) := \lM_{p}(t) \, \exp\bigg( -  \int_0^t \bigg[K(p)\g^{2}+(p-1)\Lambda_{0} (s)\bigg] \, \d s  \bigg) , $$
and we see that $y(t)$ satisfies the differential inequality 
$$ y'(t) \le - \frac{2\,K(p)}{C_{\mathrm{Sob}} }\, \lm_{ \nu_{\g,p} }(t)^{-\frac{2p}{d(p-1)}}\,  \, y(t)^{1 + \frac{2}{dp-d}} , $$
where we used that $y(t) \leq \lM_{p}(t)$. This inequality can be solved (for $t \in [0,T]$) as 
$$ y(t) \le \bigg[ \frac{4\,K(p)}{d(p-1)\,C_{\mathrm{Sob}}}   \bigg(\sup_{\tau \in [0,T]} {\lm}_{ \nu_{\g,p} }(\tau) \bigg)^{-\frac{2p}{d(p-1)}}  \, t \bigg]^{- \frac{d(p-1)}2} . $$
This proves the result.
\end{proof}

\begin{rmq} The aforementioned criterion \eqref{eq:L1-criterion} is always met in the case of moderately soft potentials $\g \in [-2,0)$, see Section \ref{sec:moderate}.
\end{rmq}

 \subsection{Proof of the $\varepsilon$-Poincar\'e inequality}\label{sec:eps} For the Prodi-Serrin criteria with $r >1$, by virtue of the proof of the above Proposition, the crucial point for the appearance of $L^{p}$-moments lies in a suitable estimate for 
$$
\int_{\R^{d}} \bm{c}_{\g}[f(t)](v)f^{p}(t,v)\d v.
$$
To deal with such a term, one resorts to the $\e$-Poincar\'e inequality stated in Proposition \ref{prop:GG}. We first give the proof of this fundamental result:
\begin{proof}[Proof of Proposition \ref{prop:GG}]    We start with the case $ -d < \g < 0$. For a given  $g, \phi\ge 0$, we define
$$
I[\phi] :=\int_{\R^{d}}\phi^{2}\bm{c}_{\g}[g]\d v=(d-1)\,(\g+d)\int_{\R^{d}\times \R^{d}}|v-\vet|^{\g}\phi^{2}(v)g(\vet)\d v\d\vet\,.
$$
For any $v,\vet \in \R^{d}$, if $|v-\vet| < \frac{1}{2}\langle v\rangle$, then $\langle v\rangle\leq 2\langle \vet\rangle$.  We deduce from this, see \cite[Eq. (2.5)]{amuxy}, that
$$|v-\vet|^{\g}\leq 2^{-\g}\langle v\rangle^{\g}\left(\ind_{\left\{|v-\vet| \geq\frac{\langle v\rangle}{2}\right\}}+
\langle \vet\rangle^{-\g}|v-\vet|^{\g}\ind_{\left\{|v-\vet|< \frac{\langle v\rangle}{2}\right\}}\right).$$
Therefore,
\begin{equation}\label{eq:ineq}
I[\phi] \leq  2^{-\g}(d-1)\,(\g+d)\left(I_{1}+I_{2}\right),
\end{equation}
with
\begin{multline*}
I_{1}:=\int_{\R^{d}}\langle v\rangle^{\g}\phi^{2}(v)\d v \int_{|v-\vet| \geq \frac{\langle v\rangle}{2}} g(\vet)\d\vet\,,\\
\text{and}\qquad I_{2}:=\int_{\R^{d}}\langle \vet\rangle^{-\g}g(\vet)\d \vet\int_{|v-\vet| < \frac{1}{2}\langle v\rangle }|v-\vet|^{\g}\langle v\rangle^{\g}\phi^{2}(v)\d v.\end{multline*}
Introducing
$$F=\langle \cdot \rangle^{-\g}g, \qquad \psi=\langle \cdot \rangle^{\frac{\g}{2}}\phi , $$
one checks that
$$I_{1}  \leq \, \|g\|_{{1}} \|\psi^{2}\|_{{1}}\,,$$
while
$$I_{2} \leq \int_{\R^{d}\times \R^{d}}|v-\vet|^{\g}F(\vet)\psi^{2}(v)\d v\d\vet.$$
We estimate then $I_{2}$  thanks to  {Hardy-Littlewood-Sobolev} inequality \eqref{eq:HLS} to get
$$I_{2}  \leq C\,\|F\|_{q}\,\|\psi^{2}\|_{{r}}=C\,\|F\|_{{q}}\,\|\psi\|_{{2r}}^{2}, \qquad \frac{1}{r} =2-\frac{1}{q}+\frac{\g}{d}\quad  (1< q,r < \infty), $$
where $C$ depends only on $\g, d, q$.
\par
We recall that we apply this with $\frac{d}{d+2+\g} < q< \frac{d}{d+\g}.$ We write   for some $s \in (0,1)$
\begin{equation}\label{eq:prHLS}
q=\frac{d}{d+\g+2s}, \qquad r=\frac{d}{d-2s}.\end{equation}
Notice that $2r=\frac{2d}{d-2s}$. According for example to Theorem 1.38 of \cite{chemin}, $\dot{\H}^{s}$ is continuously embedded in $L^{2r}$, that is (for $C>0$ depending on $d,s$)
$$I_{2} \leq C \|F\|_{{q}}\|\psi\|_{\dot{\H}^{s}}^{2}.$$
%for a constant $C_{0}>0$ depending on $d$ and $s$ only. 
Moreover, since 
$$\|\psi\|_{\dot{\H}^{s}} \leq \|\psi\|_{\dot{\H}^{1}}^{s}\,\|\psi\|_{{2}}^{1-s}\,, $$
see for example \cite[Proposition 1.32]{chemin}, one has that
$$I_{2} \leq C \|F\|_{{q}}\|\psi\|_{\dot{\H}^{1}}^{2s}\,\|\psi\|_{{2}}^{2-2s}.$$
Thanks to Young's inequality, there is $\tilde{C} >0$ depending only on $d,\g,s$ such that, for any $\delta >0$,
$$2^{-\g}(d-1)\,(\g+d)I_{2} \leq \delta\,\|\psi\|_{\dot{\H}^{1}}^{2}+\tilde{C}\|F\|_{{q}}^{\frac{1}{1-s}}\,\delta^{-\frac{s}{1-s}}\|\psi\|_{{2}}^{2}\, .$$
Plugging this inequality into the estimate for $I_{1}$, we see that
%conclude that there exists $C >0$ depending on $s$ and $\varrho_{\rm in}$ such that
$$I[\phi] \leq \delta\,\|\psi\|_{\dot{\H}^{1}}^{2} + C\left(\|g\|_{{1}} +\delta^{-\frac{s}{1-s}}\|F\|_{q}^{\frac{1}{1-s}}\right)\|\psi\|_{{2}}^{2},$$
which is exactly the desired estimate since $F=\langle\cdot\rangle^{|\g|} g$.

{We now turn to the case  when $\g = -d$.} {One notices that
$$
\int_{\R^{d}}\phi^{2}\bm{c}_{-d}[g]\d v=c_{d}\int_{\R^{d}}\phi^{2}(v)g(v)\d v=c_{d}\int_{\R^{d}}F \psi^{2}\d v\,.
$$
Thus, a simple use of H\"older's inequality yields
%, for any $s \in (0,1)$,
$$\int_{\R^{d}}\phi^{2}\bm{c}_{-d}[g]\d v \leq c_{d}\|F\|_{{q}}\|\psi\|_{{2r}}^{2}\,,\qquad \frac{1}{r}+\frac{1}{q}=1\,.$$
Writing now for $s\in (0,1)$ that $q=\frac{d}{2s}$, we see that $r=\frac{d}{d-2s}$, and we proceed identically as in
the case $-d < \g < 0$. }
% above we deduce that, for any $q >\frac{d}{2}$, there is $C_{0} >0$ depending  on $d,q$ such that
%\begin{multline}\label{eq:estimatcLP-coul}
%\int_{\R^{d}}\phi^{2}\bm{c}_{-d}[g]\d v \leq \e\,\int_{\R^{d}}\left|\nabla \left(\langle v\rangle^{-\frac{d}{2}}\phi(v)\right)\right|^{2}\d v \\
%+C_{0}\,\e^{-\frac{d}{2q-d}}\|\langle \cdot \rangle^{d}\,g\|_{q}^{\frac{2q}{2q-d}}\int_{\R^{d}}\phi^{2}\langle v\rangle^{-d}\d v\,.
%\end{multline}
\end{proof}

\subsection{Appearance of $L^{p}$-norms - general $L^{r}_{t}(L^{q})$ case}

 We now turn to the general Prodi-Serrin criterion in the $L^{r}_{t}(L^{q}_{v})$ case, with $r >1$. In such a situation, $L^{p}$-norms appear for \emph{any} choice of $p >1$. This contrasts with the previous case in which appearance of $L^{p}$ norms occurred only for $p < \frac{d}{d+\g}$. The following is the analogue of Proposition \ref{mps_coul} for $-d < \g <0:$

\begin{prop}\label{mps_g}
 Let $f_{\rm in}$ and $f=f(t,v)$ define a  solution to Eq. \eqref{LFD} in the sense of Definition \ref{defi:weak} with $-d < \g < 0.$ 
Assume that 
\begin{equation}\label{eq:PS-ggl}\langle \cdot \rangle^{|\g|}\,f \in L^{r}([0,T]\,;\,L^{q}(\R^{d})) \:\: \text{ with }\:\: \frac{2}{r}+\frac{d}{q}=2+d+\g\end{equation}
for some $T >0$, where $r \in (1,\infty)$,   {$q \in \left(\max\left(1 , \frac{d}{2+d+\g}\right),\frac{d}{d+\g}\right)$.} 
Given $p \in (1,\infty)$, assume that
$$f_{\rm in} \in L^{1}_{\nu_{\g,p}}(\R^{d}), \qquad \nu_{\g,p} :=\frac{|\g|d}{2}\left(1-\frac{1}{p}\right) .$$
Then, the solution $f=f(t,v)$ satisfies  the following estimate, for any $t \in (0,T]$:

\begin{equation}\label{Lp-Kpq}
\|f(t)\|_{p} \le K_{p,\g,q} \, t^{- \frac{d}{2} (1 - \frac1{p})} \,\sup_{\tau\in [0,T]} {\lm}_{ \nu_{\g,p}}(\tau)\,,\end{equation}
where
\begin{multline*} K_{p,\g,q} := \bigg[ \frac{2K_0}{dp\,C_{\mathrm{Sob}} }\ \bigg]^{- \frac{d}{2}\,(1 - \frac1p)}\exp\bigg(\frac{p-1}{p}\left[\frac{\g^{2}}{p}K_0+C_{0}\|f_{\rm in}\|_{1}\right]   T\bigg)\\
\exp\bigg( \frac{p-1}{p}\, C_{0}\left(\frac{K_{0}}{p}\right)^{r-1}\, \int_{0}^{T}\left\|\langle \cdot \rangle^{|\g|}\, f(\tau)\right\|_{q}^{r}\d\tau
\bigg)\,,\end{multline*}
  $K_0$ is the constant (depending on $\varrho_{\rm in}$, $E_{\rm in}$ and $H(f_{\mathrm{\rm in}})$) appearing in Remark \ref{diffusion}, $C_{\mathrm{Sob}}$ is the constant appearing in the 
Sobolev embedding  $\dot{\H}^1(\R^d) \hookrightarrow L^{\frac{2d}{d-2}}(\R^d)$ (see \eqref{eq:sobo-p}), and  $C_{0}$ is the positive constant (depending on $d, \g, q$) appearing in the $\e$-Poincar\'e inequality \eqref{eq:estimatcLP}.
\end{prop} 
 
% \begin{rmq} We recall that the restriction $r >1$ and the constraint $\frac{2}{r}+\frac{d}{q}=2+d+\g$ implies that $\max\left(1,\frac{d}{d+\g+2}\right) < q < \frac{d}{d+\g}$.\end{rmq}
 \begin{proof} The Prodi-Serrin condition \eqref{eq:PS-ggl} amounts to
$$\int_{0}^{T}\left\|f(t)\right\|_{L^{q}_{q|\g|}}^{r}\d t < \infty, \qquad \frac{d}{q}+\frac{2}{r}=d+2+\g, \qquad 1 < r < \infty.$$
Notice that the assumption on $q$ implies that $\frac{d}{d+2+\g} < q < \frac{d}{d+\g}$, and therefore,
 for some $s \in (0,1)$,
$$q=\frac{d}{d+2s+\g}, $$ 
and the relation between $r$ and $s$ is $r=\frac{1}{1-s}$.
\medskip
We recall the evolution of $L^{p}$-norms given by \eqref{eq:dtMsp-g}:
$$\dfrac{\d}{\d t}\lM_{p}(t) + 2K(p)\lD_{\g,p}(t) \leq K(p)\g^{2}\lM_{p}(t)
+ (p-1) \int_{\R^{d}} \bm{c}_{\g}[f(t)](v)\,f^p(t,v) \d v\,,$$
and estimate the last integral thanks to an application of
 the $\e$-Poincar\'e inequality \eqref{eq:estimatcLP} with $g=f(t, \cdot)$, $\phi=f(t, \cdot)^{\frac{p}{2}}$. This yields  
\begin{multline*}
 { \int_{\R^{d}} {\bm{c}}_{\g}[f(t)](v)\,f^{p}(t,v)\d v }\, \leq \e\,\lD_{\g,p}(t) \\
+C_{0}\left( \|f(t)\|_{{1}}  
 +\e^{-\frac{s}{1-s}}\|\langle \cdot \rangle^{|\g|}f(t)\|_{q}^{\frac{1}{1-s}}\right)\int_{\R^{d}}f(t,v)^{p}\langle v\rangle^{\g}\d v.\end{multline*}
 %Setting 
 %$$\overline{\lm}:=\sup_{\tau \in [0,T]}\lm_{\nu_{\g,p}}(\tau)$$
Choosing $\e$ in such a way that {$(p-1)\, \e =K(p)$}, we see that
\begin{equation*}
\frac{\d}{\d t}\lM_{p}(t) + K(p)\lD_{\g,p}(t) \leq K(p)\g^{2}\lM_{p}(t)
+ (p-1)\left(C_{0}\|f_{\rm in}\|_{1}+\bm{\Lambda}(t)\right)\int_{\R^{d}}f(t,v)^{p}\langle v\rangle^{\g}\d v,\end{equation*}
with 
$$\bm{\Lambda}(t):=C_{0}\e^{-\frac{s}{s-1}}\|\langle \cdot \rangle^{|\g|}f(t)\|_{q}^{\frac{1}{1-s}}=C_{0}\left(\frac{K_{0}}{p}\right)^{r-1}\|\langle \cdot \rangle^{|\g|}f(t)\|_{q}^{r} \in L^{1}([0,T]), $$
(the last point is precisely the Prodi-Serrin assumption \eqref{eq:PS-ggl}). Dropping the weight $\langle v\rangle^{\g}$ in the last integral, we end up with 
\begin{equation} \label{agg-g}
\frac{\d}{\d t}\lM_{p}(t) + K(p)\lD_{\g,p}(t) \leq \left[K(p)\g^{2} + (p-1)C_{0}\|f_{\rm in}\|_{1} + (p-1)\bm{\Lambda}(t)\right]\lM_{p}(t) .
\end{equation}
At this point, using again \eqref{usedt-g} as in the proof of Prop. \ref{prop:estMs-g},
 we deduce that
\begin{multline*}
\dfrac{\d}{\d t}\lM_{p}(t) + \frac{K(p)}{C_{\mathrm{Sob}}}\lm_{\nu_{\g,p}}(t)^{-\frac{2p}{d(p-1)}}\lM_{p}(t)^{1+\frac{2}{d(p-1)}}\\
\leq  \left[K(p)\g^{2} + (p-1)C_{0}\|f_{\rm in}\|_{1} + (p-1)\bm{\Lambda}(t)\right]\lM_{p}(t) .
\end{multline*}
The conclusion follows exactly as in the proof of Proposition \ref{prop:L1-criterion} (which corresponds to the case $s=0$).\end{proof}
 As far as propagation of $L^{p}$-norms is concerned, we can deduce also the following:

 \begin{cor}\label{cor:propagLP}  Let $f_{\rm in}$ and $f=f(t,v)$ define a  solution to Eq. \eqref{LFD} in the sense of Definition \ref{defi:weak} with $-d < \g < 0.$ 
Assume that \eqref{eq:PS-ggl} holds true for some $T >0$, where $r \in (1,\infty)$, $q \in \left(\max\left(1,\frac{d}{2+d+\g}\right), {\frac{d}{d+\g}}\right)$. Given $p \in (1,\infty)$, assume that
 {$f_{\rm in} \in L^{p}(\R^{d})$. }
%(I removed the assumption $f_{\rm in} \in L^{1}_{\nu_{\g,p}}(\R^{d})$ since, in the proof of Theorem 3.8, this assumption was used after the obtaining of (3.20). Here, we resume from (3.20). V.)}
Then, the solution $f=f(t,v)$ belongs to $L^{\infty}([0,T]\,,\,L^{p}(\R^{d}))$. More precisely,
\begin{multline}\label{Lp-prop}
\|f\|_{L^{\infty}([0,T];L^{p}(\R^{d}))} \le \, \|f_{\rm in}\|_{p}\exp\left(\frac{p-1}{p}\left[\frac{\g^{2}}{p}K_0+C_{0}\|f_{\rm in}\|_{1}\right]   \, T \right)\\
\exp\bigg(\frac{p-1}{p}\,C_{0}\left(\frac{K_{0}}{p}\right)^{r-1} \int_{0}^{T}\left\|\langle \cdot \rangle^{|\g|}f(\tau)\right\|_{q}^{r}\d\tau \bigg)\,,\end{multline}
 where $K_0$ is the constant (depending on $\varrho_{\rm in}$, $E_{\rm in}$ and $H(f_{\mathrm{\rm in}})$) appearing in Remark \ref{diffusion}, 
 {and  $C_{0}$ is the positive constant (depending on $d, \g, q$) appearing in the $\e$-Poincar\'e inequality \eqref{eq:estimatcLP}.}
 \end{cor}

 \begin{proof} Coming back to  \eqref{agg-g} {(notice that no assumption on the finiteness of $L^1$-moment is needed to get this inequality)}, dropping the nonnegative term $K(p)\lD_{p}(t)$ on the right-hand-side, we deduce from Gr\"onwall's Lemma that
 $$\lM_{p}(t) \leq \exp\left\{\left(K(p)\g^{2} + (p-1)C_{0}\|f_{\rm in}\|_{1} \right)t + (p-1)\int_{0}^{t}\bm{\Lambda}(\tau)\d\tau\right\}\lM_{p}(0),$$
from which the result follows. 
 \end{proof}

\begin{proof} of Theorem \ref{827}:  It is a direct consequence of Proposition \ref{without_dissip-g} and Proposition \ref{prop:L1-criterion}  when  the $L^{1}_{t}(L^{\infty}_{v})$ condition is considered, and of Proposition \ref{mps_g} and Corollary \ref{cor:propagLP} when the $L^{r}_{t}(L^{q}_{v})$ condition is considered.
\end{proof}

\section{Stability and uniqueness of solution}\label{sec:unique}

%\textcolor{blue}{Still to be written using the result of \cite{chern}. I think we can unify here both the Coulomb and $-d < \g <0$ cases.} 
 
 We adapt here the strategy proposed in \cite{chern} to deduce uniqueness of solutions. Notice that our proof covers all cases $\g \in [-d,0)$ and $d\geq 3$ whereas the uniqueness result in \cite{chern} is given in dimension $d=3$ for the Coulomb case $\g=-3$. 
 %Second, we provide a general stability result in $L^{p}$-space $1 < p \leq 2$ in which we consider two solutions $f,g$ to \eqref{LFD} associated with respective initial data $g_{\rm in} \in L^{2}(\R^{d})$ and $f_{\rm in} \in L^{p}(\R^{d})$. 

 %Then, a simple consequence of our stability result is then the uniqueness of solutions satisfying Prodi-Serrin condition and associated to an initial datum in (weighted) $L^{2}(\R^{d})$ space. 
 
 Our strategy is inspired by the work \cite{chern} in the sense that we deduce the stability from suitable $L^{\infty}([0,T]\,;\,L^{p}_{k}(\R^{d}))$ estimates on the solution to \eqref{LFD}. We show first that such uniform in time estimates, which are variants of the estimates obtained in Sections \ref{sec:Coulomb} and \ref{sec:prodi}, hold true under our Prodi-Serrin conditions. The result is a consequence  of the study of the evolution of weighted $L^{p}$-norms provided in Appendix \ref{sec:weight}. More precisely, we show the following 

\begin{prop}\label{prop:ChernINT}
Let $f_{\rm in}$ and $f=f(t,v)$ define a  solution to Eq. \eqref{LFD} in the sense of Definition \ref{defi:weak} on $[0,T]$ (for some $T>0$) with $d\in \N$, $d\ge 3$, $-d \leq \g < 0.$ Let $k \geq0$. We consider the following two alternative assumptions (with the convention $\frac{d}{d+\g} = \infty$ if $\g=-d$)
\begin{enumerate}
\item[\textbf{Hyp. 1.}]  $f_{\rm in} \in L^{p}_{k}(\R^{d})$ 
for some $1 < p < \frac{d}{d+\g}$  and
 $$f \in L^{1}([0,T]\,;\;L^{\frac{d}{d+\g}}(\R^{d}))\qquad   .$$  
 %and $f_{\rm in} \in L^{p}_{k}(\R^{d})$ 
%for some $1 < p < \frac{d}{d+\g}$;
\end{enumerate}
\begin{enumerate}
\item[\textbf{Hyp. 2.}] $f_{\rm in} \in L^{p}_{k}(\R^{d})$ for some  $p >1$ and $$\langle \cdot \rangle^{|\g|}\,f \in L^{r}([0,T]\,;\,L^{q}(\R^{d})) \quad \text{ with } \quad  \frac{2}{r}+\frac{d}{q}=2+d+\g\,,$$
 {where $r \in (1,\infty)\,, q \in \left(\max\left(1, \frac{d}{2+d+\g}\right),\frac{d}{d+\g}\right)\,.$} Moreover, in the Coulomb case in dimension $d=3=-\g$, we assume $f_{\rm in} \in L^{1}_{s}(\R^{d})$ for some $s >2.$
 \end{enumerate} Then, under \textbf{Hyp. 1} or \textbf{Hyp. 2}, and for any $t \in [0,T]$, one has
\begin{equation}\label{eq:LinftyLp}
\int_{\R^{d}} f^p(t,v) \langle v\rangle^{k} \d v +K_{0}\left(1-\frac{1}{p}\right)\int_{0}^{t}\d \tau\int_{\R^{d}} \left|\nabla \left(\langle v\rangle^{\frac{k+\g}{2}} f^{\frac{p}{2}}(\tau,v) \right)\right|^{2}\d v  \leq \bm{C}(T,f_{\rm in})\,,\,\end{equation}
where $\,\bm{C}(T,f_{\rm in})$ is an explicit positive constant depending on $d$, $\g,p,k,T, K_0, \|f_{\rm in}\|_{L^{1}_{2}}, H(f_{\rm in}), \|f_{\rm in}\|_{L^{p}_{k}}$ and the Prodi-Serrin norms $\|f\|_{L^{1}_{t}(L^{\frac{d}{d+\g}})}$ or $\left\|\langle \cdot\rangle^{|\g|}f\right\|_{L^{r}_{t}(L^{q})}$ (and $q$),
according to the considered case (and on $\|f_{\rm in}\|_{L^{1}_{s}}$ $s >2$ in the Coulomb case in dimension $d=3$). 

If one moreover assumes in Assumptions \textbf{Hyp. 1} or \textbf{Hyp. 2} that $p > \frac{d}{d+\g+2}$ and 
$$f_{\rm in} \in L^{1}_{\nu}(\R^{d}), \qquad \nu > \overline{\nu}$$
with 
$$\overline{\nu}:=\dfrac{d(p-1)|\g|-k(\g+2)^{-}}{d(p-1)-p(\g+2)^{-}}\ind_{k \leq p|\g|}+\dfrac{(d(p-1)-k)|\g|}{d(p-1)+p\g}\ind_{k >p|\g|},$$
where $a^{-}=-a\ind_{a <0}$,
% \begin{equation}\label{eq:nuq}
% \overline{\nu} :=\begin{cases}\dfrac{d(p-1)|\g|-k(\g+2)^{-}}{d(p-1)-p(\g+2)^{-}} \qquad &\text{ if } k \leq p|\g|\\
%\\
%\dfrac{(d(p-1)-k)|\g|}{d(p-1)+p\g} \qquad &\text{ if } k \geq p|\g|,\end{cases}\end{equation}
%where $a^{-}=-a\ind_{a <0}$, 
then, for any smooth $\phi$,  the following version of the $\e$-Poincar\'e inequality is valid for any $t \in [0,T]$ and any $\e >0$:
\begin{multline}\label{eq:estimatcLP-ft}
\int_{\R^{d}}\phi^{2}(v)\bm{c}_{\g}[f(t)] \d v \leq \e\,\int_{\R^{d}}\left|\nabla \left(\langle v\rangle^{\frac{\g}{2}}\phi(v)\right)\right|^{2}\d v 
+C_{k}(T,f_{\rm in},\e)\int_{\R^{d}}\phi^{2}(v)\langle v\rangle^{\g}\d v,
\end{multline}
where $C_{k}(T, f_{\rm in},\e)$ is an explicit positive constant depending on $ \|f_{\rm in}\|_{L^{p}_{k}}$, $\|f_{\rm in}\|_{L^{1}_{\max(2,\nu)} }$, $\e$, $d,\g,p$, $k,T, K_0,$ and the Prodi-Serrin norms $\|f\|_{L^{1}_{t}(L^{\frac{d}{d+\g}})}$ or $\left\|\langle \cdot\rangle^{|\g|}f\right\|_{L^{r}_{t}(L^{q})}$ (and $q$), according to the considered case. 
\end{prop}

%\begin{rmq} We emphasize here that both estimates \eqref{eq:LinftyLp}  and \eqref{eq:estimatcLP-ft} are valid for a range of parameters $p$ prescribed by the initial datum: they hold true for any $p >1$ under Assumptions \textbf{Hyp. 2} and for $1 < p < \frac{d}{d+\g}$ under \textbf{Hyp. 1.} On the contrary, \eqref{eq:estimatcLP-ft} holds true only under the additional assumption that $p >\frac{d}{d+\g+2}.$  We wish to point out however that, assuming $f_{\rm in} \in L^{p_{0}}(\R^{d})$ with $p_{0} > \frac{d}{d+\g+2}$, then estimates \eqref{eq:LinftyLp} and \eqref{eq:estimatcLP-ft} are then valid or any $p \leq p_{0}$ (with the additional restriction $p < \frac{d}{d+\g}$ under \textbf{Hyp. 1}). \textcolor{blue}{I have not understood the last sentence of this remark, are there problems with weights?} \textcolor{purple}{Bert: Since we are now  dealing with $p=2$ only, I guess we can remove the remark.} 
%\end{rmq} 

\begin{proof} The proof follows the line of the corresponding results in  Proposition \ref{prop:L1-criterion} and Theorem \ref{mps_g}, whose proofs are modified in order to handle weights. Given $k \geq0$, we use the notations of Appendix \ref{sec:weight} and deduce the evolution of the weighted $L^{p}$-norms from \eqref{eq:dtMsp},
 which holds for any $p > 1$:
\begin{multline}\label{eq:lMkpt}
\dfrac{\d}{\d t}\lM_{k,p}(t) + 2K(p)\lD_{k+\g,p}(t) \leq K(p)(k+\g)^{2}\lM_{k+\g,p}(t)\\
+ C_{k,\g,p}\sum_{i=0}^{2}\int_{\R^{d}}\langle v\rangle^{k-i} \bm{c}_{\g+i}[f(t)](v)\,f^p(t,v) \d v\,,
\end{multline}
where $C_{k,\g,p}$ depends on $d, k,\g,p$ (the explicit value of the constant is given in Appendix \ref{sec:weight}).
\medskip

We divide the proof according to the two cases $\g=-d$ or $-d < \g <0.$\\
\par 

\textit{1) The Coulomb case.} Let us begin with the Coulomb case $\g=-d$. Thanks to \eqref{eq:lemcompar}, we deduce now from  \eqref{eq:lMkpt} that 
\begin{multline}\label{eq:lmS1-coul}
\dfrac{\d}{\d t}\lM_{k,p}(t) + 2K(p)\lD_{k-d,p}(t) \leq \bm{c}_{k,p}\lM_{k,p}(t)  +\bm{C}_{k,p}\int_{\R^{d}}\langle v\rangle^{k-1}\overline{\bm{c}}_{1-d}[f(t)](v)f^{p}(t,v)\d v\,\\
+C_{k,p}\int_{\R^{d}}\langle v\rangle^{k}f^{p+1}(t,v)\d v\,,\end{multline}
where $\bm{C}_{k,p},\bm{c}_{k,p},C_{k,p}$ are explicit positive constants depending only on $p,d,k,K_{0},\|f_{\rm in}\|_{L^{1}}$ and  where, for all $\alpha \in (-d,0)$
$$\overline{\bm{c}}_{\alpha}[f(t)](v) := (d-1)(d+\alpha) \, \int_{|v-\vet|\leq1}|v-\vet|^{\alpha}f(t,\vet)\d\vet.$$
We estimate the last two integrals in \eqref{eq:lmS1-coul} in different ways according to the Prodi-Serrin conditions that we consider.

\noindent \textit{$\bullet$ \textbf{Hyp. 1.}} First, one has $\ds\int_{\R^{d}}\langle v\rangle^{k}f^{p+1}(t,v)\d v \leq \|f(t)\|_{\infty}\lM_{k,p}(t).$
Second, 
\begin{equation}\label{c1-d}
  \overline{\bm{c}}_{1-d}[f(t)](v)=(d-1)\int_{|v-\vet| \leq1}|v-\vet|^{1-d}f(t,\vet)\d \vet \leq (d-1)\|f(t)\|_{\infty}|\S^{d-1}|
  \end{equation}
so that
$$\int_{\R^{d}}\langle v\rangle^{k-1}\overline{\bm{c}}_{1-d}[f(t)](v)f^{p}(t,v)\d v\leq (d-1)|\S^{d-1}|\|f(t)\|_{\infty}\lM_{k-1,p}(t).$$
Therefore, \eqref{eq:lmS1-coul} becomes
$$\dfrac{\d}{\d t}\lM_{k,p}(t) + 2K(p)\lD_{k-d,p}(t) \leq \bm{C}(k,p,d)\|f(t)\|_{\infty}\lM_{k,p}(t)$$
and, since $\|f(\cdot)\|_{\infty} \in L^{1}([0,T])$ by assumption, we deduce from  Gr\"onwall's Lemma the bound
\begin{multline*}
\lM_{k,p}(t) + 2K(p)\int_{0}^{t}\lD_{k-d,p}(s)\exp\left(\bm{C}(k,p,d)\int_{s}^{t}\|f(\tau)\|_{\infty}\d\tau\right)\, ds \\
\leq \lM_{k,p}(0)\exp\left(\bm{C}(k,p,d)\int_{0}^{t}\|f(\tau)\|_{\infty}\d\tau\right)\end{multline*}
which gives the result.

\noindent \textit{$\bullet$ \textbf{Hyp. 2.}} In this second case,  we observe that the term $\int_{\R^{d}}\langle v\rangle^{k}f^{p+1}(t,v)\d v$ is identical (up to a constant) to $\int_{\R^{d}}\langle v\rangle^{k}\bm{c}_{-d}[f(t)](v)f^{p}(t,v)\d v$ and is estimated using the Coulomb version of the $\e$-Poincar\'e inequality \eqref{eq:estimatcLP}, with $\phi=\langle \cdot \rangle^{\frac{k}{2}}f^{\frac{p}{2}}$ and $g=f$, leading to
\begin{equation}\label{eq:CkpCou}
C_{k,p}\int_{\R^{d}}\langle v\rangle^{k}f^{p+1}(t,v)\d v \leq \varepsilon \lD_{k-d,p}(t) + C_{\e} \|\langle \cdot \rangle^{d}f(t)\|_{q}^{\frac{2q}{2q-d}} \lM_{k-d,p}(t) ,\end{equation}
where $q >\frac{d}{2}$ is given by $\frac{2}{r}+\frac{d}{q}=2$ (recall we consider here $\g=-d$) and $r=\frac{2q}{2q-d}$. 

To estimate the term $$\ds\int_{\R^{d}}\langle v\rangle^{k-1}\overline{\bm{c}}_{1-d}[f(t)](v)f^{p}(t,v)\d v,$$ we have to additionally distinguish the two cases $d=3$, $d \geq 4$. We give here only the proof in the case $d=3$ and refer to \cite[Proof of Prop. 1.8, $d \geq 4$]{ABDL} for the other case. We point out already that the use that we will make of Proposition \ref{prop:ChernINT} in the subsequent stability result is actually restricted to the case $d=3.$  

For $d=3$, one has $\overline{\bm{c}}_{1-d}[f(t)]=\overline{\bm{c}}_{-2}[f(t)]\leq \bm{c}_{-2}[f(t)]$ and, resorting to the critical $\e$-Poincar\'e inequality in \cite{ABL} (see Prop. \ref{propo:GG2} in Appendix \ref{sec:known}), one has the following:
 for any $\alpha >2$ and any $\e >0$,
\begin{multline*} {\int_{\R^{3}}\langle v\rangle^{k-1}\overline{\bm{c}}_{-2}[f(t)](v)\,f^{p}(t,v)\d v\, } \\
 { \leq \e\;\lD_{k-3,p}(t) +C_{0}\exp\left(\e^{-\frac{\alpha}{\alpha-2}}\,\lm_{\alpha}(f(t))^{\frac{\alpha}{\alpha-2}}\right)\int_{\R^{3}}f(t,v)^{p}\langle v\rangle^{k-3}\d v \,,
}
\end{multline*}
for some $C_{0} >0$ depending only on $\varrho_{\rm in}$, $E_{\rm in }$ and $H(f_{\rm in})$. In particular, assuming $f_{\rm in} \in L^{1}_{\alpha}(\R^{3})$ and $\alpha >2$, one deduces that there exists $C_{T} >0$ depending only on $\|f_{\rm in}\|_{L^{1}_{\alpha}}$,  $H(f_{\rm in})$ such that
\begin{multline}\label{eq:c1-d}
\bm{C}_{k,p}\int_{\R^{d}}\langle v\rangle^{k-1}\overline{\bm{c}}_{1-d}[f(t)](v)\,f^{p}(t,v)\d v\,
 \leq \e\;\lD_{k-d,p}(t) \\ +\exp\left(C_{T}(1+\e^{-\frac{\alpha}{\alpha-2}})\right)\lM_{k-d,p}(t)\,\:\;\ (d=3).
\end{multline}
Plugging \eqref{eq:c1-d}, \eqref{eq:CkpCou} into \eqref{eq:lmS1-coul} and choosing $\e=\e_{0}>0$ small enough, we obtain
$$\dfrac{\d}{\d t}\lM_{k,p}(t) + K(p)\lD_{k-d,p}(t) \leq \Lambda(t)\,\lM_{k,p}(t), \qquad (d=3)$$
where we introduced, for the specific choice of $\e=\e_{0}$,
$$
\Lambda(t)=\exp\left(C_{T}(1+\e_{0}^{-\frac{\alpha}{\alpha-2}})\right) + C_{\e_{0}} \|\langle \cdot \rangle^{d}f(t)\|_{q}^{\frac{2q}{2q-d}} +\bm{c}_{k,p}.$$
By assumption, $\Lambda \in L^{1}([0,T])$ and one concludes as before by a Gr\"onwall argument.
\\

\textit{2) The case $-d<\g <0.$} 
%We consider now the case $-d < \g <0.$
  Combining Eq. \eqref{eq:lMkpt}  with Lemma \ref{lem:compa} (and modifying the name of the constant appearing in this lemma), we deduce then that \begin{equation}\label{eq:lmS1}
\dfrac{\d}{\d t}\lM_{k,p}(t) + 2K(p)\lD_{k+\g,p}(t) \leq \bm{c}_{k,p}\lM_{k,p}(t)  +\bm{C}_{k,p}\int_{\R^{d}}\langle v\rangle^{k}\overline{\bm{c}}_{\g}[f(t)](v)f^{p}(t,v)\d v\,,\end{equation}
where $\bm{C}_{k,p},\bm{c}_{k,p}$ are explicit positive constants depending only on $k,p,d,\g,K_{0}$.
%\|f_{\rm in}\|_{L^{1}_{2}}$. 
As before, we estimate the last integral in different ways according to the Prodi-Serrin conditions that we consider.

\noindent \textit{\textbf{Hyp. 1.}} 
According to Peetre's inequality,
 for all $v,\vet\in \R^{d}$ and $s \in \R$, $\langle v\rangle^{s} \leq 2^{\frac{|s|}{2}}\langle v-\vet\rangle^{|s|}\langle \vet\rangle^{s}$,
 so that (for some $C_s>0$ depending on $s$)
$$\int_{\R^{d}}\langle v\rangle^{k}\overline{\bm{c}}_{\g}[f(t)](v)f^{p}(t,v)\d v \leq C_{s}\int_{|v-\vet|\leq1}|v-\vet|^{\g}\langle \vet\rangle^{s}f(t,\vet)\langle v\rangle^{k-s}f^{p}(t,v)\d v\d\vet,$$
for any $0 \leq s \leq k$. Using now Hardy-Littlewood-Sobolev inequality \eqref{eq:HLS} with $g=\langle \cdot\rangle^{s}f(t,\cdot)$, $h(\cdot)=\langle \cdot \rangle^{k-s}f^{p}(t,\cdot)$, so that
$$\frac{1}{r}=2+\frac{\g}{d}-\frac{1}{p} < 1,$$
we conclude that (for some $C_{d,\g,p,s}>0$ depending on $d,\g,p,s$)
\begin{equation*}\begin{split}
\int_{\R^{d}}\langle v\rangle^{k}\overline{\bm{c}}_{\g}[f(t)](v)f^{p}(t,v)\d v &\leq C_{d,\g,p,s}\|\langle \cdot \rangle^{s}f(t)\|_{{p}}
\|\langle \cdot\rangle^{k-s}f^{p}(t)\|_{{r}}\\
&=C_{d,\g,p,s}\|\langle \cdot \rangle^{s}f(t)\|_{{p}}\|\langle \cdot\rangle^{\frac{k-s}{p}}f(t)\|_{{pr}}^{p}.\end{split}\end{equation*}
Writing $q_{0}=\frac{d}{d+\g}$ and recalling that $p < \frac{d}{d+\g}$, we use an interpolation based on H\"older's inequality, namely
$$
\|\langle \cdot\rangle^{\frac{k-s}{p}}f(t)\|_{{pr}}^{p} \leq \|\langle \cdot\rangle^{a}f(t)\|_{{p}}^{p\theta}\|\langle \cdot\rangle^{b}f(t)\|_{{q_{0}}}^{p(1-\theta)},
$$
with 
$$\frac{1}{pr}=\frac{\theta}{p}+\frac{1-\theta}{q_{0}}, \qquad \theta=\frac{q_{0}-pr}{r(q_{0}-p)}, \qquad \frac{k-s}{p}=a\theta+b(1-\theta).$$
Observing that
$\frac{1}{r}=\frac{1}{q_{0}}+1-\frac{1}{p},$
one actually has 
$\theta=1-\frac{1}{p},$ so that $p\,(1-\theta)=1$ and
$$\int_{\R^{d}}\langle v\rangle^{k}\overline{\bm{c}}_{\g}[f(t)](v)f^{p}(t,v)\d v \leq C_{d,\g,p,s}\, \|\langle \cdot \rangle^{s}f(t)\|_{{p}}\|\langle \cdot\rangle^{a}f(t)\|_{{p}}^{p-1}\|\langle \cdot \rangle^{b}f(t)\|_{{q_{0}}}.$$
Choosing $s=a=\frac{k}{p}$ and observing that $b=0$ and $\|\langle \cdot\rangle^{s}f(t)\|_{{p}}^{p}=\lM_{k,p}(t)$, we see that
\begin{equation}\label{Eq:bmCg}
\int_{\R^{d}}\langle v\rangle^{k}\overline{\bm{c}}_{\g}[f(t)](v)f^{p}(t,v)\d v  \leq   {\Lambda}_{0}(t)\lM_{k,p}(t),\end{equation}
with 
$$ {\Lambda}_{0}(t)=C_{d,\g,p,s}\,\|f(t)\|_{{q_{0}}} \in L^{1}([0,T]).$$
Coming back to \eqref{eq:lmS1}, we end up with the differential inequality
$$\dfrac{\d}{\d t}\lM_{k,p}(t) + 2K(p)\lD_{k+\g,p}(t) \leq \left(\bm{c}_{k,p}+\bm{C}_{k,p}\Lambda_{0}(t)\right)\lM_{k,p}(t), $$
so that assuming that $f_{\rm in} \in L^{p}_{k}(\R^{d})$ and remembering that $\bm{c}_{p}+\bm{C}_{p}\Lambda_{0}(t) \in L^{1}([0,T])$, we deduce from  Gr\"onwall's Lemma the bound
\begin{multline*}
\lM_{k,p}(t) + 2K(p)\int_{0}^{t}\lD_{k+\g,p}(s)\exp\left(\int_{s}^{t}\left(\bm{c}_{k,p}+\bm{C}_{k,p}\Lambda_{0}(\tau)\right)\d\tau\right)\d s\\
 \leq \lM_{k,p}(0)\exp\left(\int_{0}^{t}\left(\bm{c}_{k,p}+\bm{C}_{k,p}\Lambda_{0}(\tau)\right)\d \tau\right)\,,\end{multline*}
which gives the result.

\noindent \textit{\textbf{Hyp. 2.}} In this second case, observing that $\overline{\bm{c}}_{\g}[f(t)] \leq \bm{c}_{\g}[f(t)]$, we apply the $\e$-Poincar\'e inequality \eqref{eq:estimatcLP} with $g=f(t, \cdot)$, $\phi=\langle  \cdot \rangle^{\frac{k}{2}}f^{\frac{p}{2}}(t, \cdot)$, and we  deduce that, when $\e >0$,
 \begin{multline*}
 {  \int_{\R^{d}}\langle v\rangle^{k}\overline{\bm{c}}_{\g}[f(t)](v)\,f^{p}(t,v)\d v }\, \leq \e\,\lD_{k+\g,p}(t) \\
+C\left( \|f(t)\|_{{1}}  
 +\e^{-\frac{s}{1-s}}\|f(t)\|_{L^{q}_{{q|\g|}}}^{\frac{1}{1-s}}\right)\int_{\R^{d}}f(t,v)^{p}\langle v\rangle^{\g+k}\d v, \end{multline*}
 for some positive constant $C >0$, and where $s=\frac{d-q(d+\g)}{2q} \in (0,1)$. Proceeding now as in the proof of Theorem \ref{mps_g},
choosing $\e$ in such a way that {$ \bm{C}_{k,p} \e =K(p)$}, we deduce from \eqref{eq:lmS1} that
\begin{equation} \label{agg}
\frac{\d}{\d t}\lM_{k,p}(t) + K(p)\lD_{k+\g,p}(t) \leq 
\left(\bm{c}_{k,p}+\bm{C}_{k,p} \Lambda(t)\right)\lM_{k,p}(t)\,,
\end{equation}
with (using $C_\e >0$ for emphasizing the dependence w.r.t $\e$)
 {$$\Lambda(t):=C\|f_{\rm in}\|_{1}+C_{\e}\|\langle \cdot \rangle^{|\g|}f(t)\|_{q}^{\frac{1}{1-s}}=C\|f_{\rm in}\|_{1}+C_{\e}\|\langle \cdot \rangle^{|\g|}f(t)\|_{q}^{r} \in L^{1}([0,T]), $$}
(the last point is precisely the assumption of the Proposition). As before, we conclude thanks to Gr\"onwall's Lemma and integration of \eqref{agg}.  This proves \eqref{eq:LinftyLp} in the case $-d < \g <0$ under the two possible assumptions \textit{\textbf{Hyp. 1}} or \textit{\textbf{Hyp. 2}}. 
\medskip
 
Let us now show how the $L^{\infty}([0,T]\,,\,L^{p}_{k}(\R^{d}))$ estimate \eqref{eq:LinftyLp} is enough to deduce \eqref{eq:estimatcLP-ft}. Let $p >\frac{d}{d+\g+2}$  be fixed -- with the restriction $p < \frac{d}{d+\g}$
when we work under \textit{\textbf{Hyp. 1}.} 

Applying the $\e$-Poincar\'e inequality \eqref{eq:estimatcLP} with $g=f(t, \cdot)$, 
%$\phi$
 we  deduce as above that (for any suitable $\phi$) 
\begin{multline*}
 { \int_{\R^{d}} {\bm{c}}_{\g}[f(t)](v)\,\phi^{2}(v)\d v }\, \leq \e\,\int_{\R^{d}}\left|\nabla \left(\langle v\rangle^{\frac{\g}{2}}\phi(v)\right)\right|^{2}\d v \\
+C\left( \|f(t )\|_{{1}}  
 +\e^{-\frac{s}{1-s}}\|\langle \cdot \rangle^{|\g|}f(t)\|_{\bar{q}}^{\frac{1}{1-s}}\right)\int_{\R^{d}}\phi^{2}(v)\langle v\rangle^{\g}\d v, \end{multline*}
 is valid for any {$\max\left(1,\frac{d}{d+\g+2}\right) < \bar{q} < \frac{d}{d+\g}$} with $s=\frac{d-\bar{q}(d+\g)}{2\bar{q}}.$ Using simple interpolation with  {$\max\left(1,\frac{d}{d+\g+2}\right) < \bar{q} <p$} one has
 $$\|\langle \cdot \rangle^{|\g|}f(t)\|_{\bar{q}} \leq \|\langle \cdot \rangle^{m}f(t)\|_{1}^{1-\theta_{0}}\|\langle \cdot \rangle^{\frac{k}{p}}f(t)\|_{p}^{\theta_{0}},$$
 with 
 $$\theta_{0}=\frac{p}{\bar{q}}\frac{\bar{q}-1}{p-1} \qquad \text{ and } \quad m=m(\bar{q})=\frac{p|\g|-\theta_{0}k}{p(1-\theta_{0})}.$$
% =\frac{\bar{q}(p-1)|\g|-k(\bar{q}-1)}{p-\bar{q}} \leq \frac{q(p-1)|\g|}{p-q}.$$ 
Since $m=m(\bar{q})=\frac{\bar{q}(p-1)|\g|-k(\bar{q}-1)}{p-\bar{q}}$, the mapping $\bar{q} \mapsto m(\bar{q})$ is nondecreasing if $p|\g| \geq k$ and nonincreasing if $k\geq p|\g|.$ Thus, recalling that $\max\left(1,\frac{d}{d+\g+2}\right) < \bar{q} < \frac{d}{d+\g}$, one has
 \begin{equation}\label{eq:nuq}
 \overline{\nu}=\inf_{\bar{q}}m(\bar{q})=\begin{cases}\dfrac{d(p-1)|\g|-k(\g+2)^{-}}{d(p-1)-p(\g+2)^{-}} \qquad &\text{ if } k \leq p|\g|,\\
\\
\dfrac{(d(p-1)-k)|\g|}{d(p-1)+p\g} \qquad &\text{ if } k \geq p|\g|,\end{cases}\end{equation}
 {where $a^- := - a \, 1_{a<0}$.}
 
Thus, assuming $f_{\rm in} \in L^{1}_{\nu}(\R^{d})$ with $\nu > \overline{\nu}$, one can choose $\max\left(1,\frac{d}{d+\g+2}\right) < \bar{q} < \frac{d}{d+\g}$ with $m=m(\bar{q}) \leq \nu$ so that
 $$\sup_{t \in [0,T]}\|\langle \cdot \rangle^{m}f(t)\|_{1} < \infty.$$
 Then, thanks to \eqref{eq:LinftyLp}, 
 $$\sup_{t \in [0,T]}\|\langle \cdot\rangle^{|\g|}f(t)\|_{\bar{q}}^{\frac{1}{1-s}} \leq C_{T}(f_{\rm in}), $$
 where $C_{T}(f_{\rm in})$ depends on $\bm{C}_{T}(f_{\rm in})$ appearing in \eqref{eq:LinftyLp} and $\|f_{\rm in}\|_{L^{1}_{\nu}}$. This proves the result.
\end{proof}
{With this, we can deduce the following stability result for solutions to Landau equation. Note that the assumptions on the two considered initial data are not identical.}

 \begin{prop}\label{theo:chern2} Let $d$ be an integer, $d\ge3$,  and $\g$ be such that $-d \leq \g < -2$ and $\frac{d}{d+\g+2} < 2$. Let $f_{\rm in},g_{\rm in}$ and $f=f(t,v), g=g(t,v)$ define two  solutions to Eq. \eqref{LFD} in the sense of Definition \ref{defi:weak} with $f(0,\cdot) =f_{\rm in}$, $g(0,\cdot) =g_{\rm in}$. 
We consider one of the following two assumptions (with the convention $\frac{d}{d+\g} = \infty$ when $\g = -d$):  
\begin{enumerate}
\item[\textbf{Hyp. 1:}] there exists $T >0$ such that $f,g \in L^{1}([0,T]\,;\;L^{\frac{d}{d+\g}}(\R^{d}))$  {and $\frac{d}{d+\g} >2$};
%for some $T >0;$ 
\end{enumerate}
\begin{enumerate}
\item[\textbf{Hyp. 2:}] there exists $T >0$ such that 
$$\langle \cdot \rangle^{|\g|}\,f ; \,\langle \cdot \rangle^{|\g|}\,g \in L^{r}([0,T]\,;\,L^{q}(\R^{d})), \qquad \text{ with } \quad \frac{2}{r}+\frac{d}{q}=2+d+\g,$$
   where $r \in (1,\infty)$, $q \in \left(\frac{d}{2+d+\g},\frac{d}{d+\g}\right)$.\end{enumerate} 
Given $k > d$, assume moreover that
$$f_{\rm in} \in L^{2}_{k}(\R^{d}) \cap L^{1}_{N}(\R^{d}), \quad g_{\rm in} \in {L^{2}_{k+2|\g|}(\R^{d})}$$
with {$N > \max\left(\frac{d}{2},2,\frac{(d-k)|\g|}{d+2\g},\frac{(d-k)|\g|+2k}{d+2(\g+2)}\right).$}
Then,  there exists $\bm{C}_{T}(f_{\rm in},g_{\rm in}) >0$ depending on $d, \g, T, k,  K_0$, $\|f_{\rm in}\|_{L^{1}_{N}}, \|g_{\rm in}\|_{L^{1}_{2}}$,  $H(f_{\rm in}), H(g_{\rm in})$, $\|f_{\rm in}\|_{L^{2}_{k}}, \|g_{\rm in}\|_{L^{2}_{k+ 2|\g|}}$, and the Prodi-Serrin norms $\|f,g\|_{L^{1}_{t}(L^{\frac{d}{d+\g}})}$ or $\left\|\langle \cdot\rangle^{|\g|}f, \langle \cdot\rangle^{|\g|}g\right\|_{L^{r}_{t}(L^{q})}$ (and $q$),
%f_{\rm in},g_{\rm in},$ 
such that
\begin{equation}\label{eq:Stability}
\|f(t)-g(t)\|_{L^{2}_{k}}^{2} \leq \bm{C}_{T}\|f_{\rm in}-g_{\rm in}\|_{L^{2}_{k}}^{2} \qquad \forall t \in [0,T].
\end{equation}
\end{prop}

\begin{rmq} Notice that the result still applies if $f$ and $g$ do not satisfy the \emph{same} hypothesis: it applies if $f$ satisfies \textbf{Hyp. 1} and $g$ satisfies \textbf{Hyp. 2} or the opposite. We stated the result in the above way for simplicity.\end{rmq}
 
\begin{rmq}\label{rmk:constraint} We point out that, in the Coulomb case $\g=-d$, the restriction $\frac{d}{d+\g+2} <2$ enforces $d=3.$ \end{rmq}

  \begin{proof}
 We consider two solutions $f,g$ with   initial data $f(0)=f_{\rm in}$ and $g(0)=g_{\rm in}$,
 and write $h=f-g$. {We notice first that, since $f_{\rm in}, g_{\rm in} \in L^{2}_{k}(\R^{d})$, estimate \eqref{eq:LinftyLp} holds true for both $f,g$ with $p=2$, where we recall that, under Assumption \textit{\textbf{Hyp. 1}}, we assume $2 < \frac{d}{d+\g}$. It ensures that 
$$f,g \in L^{\infty}([0,T]\;;\;L^{2}_{k}(\R^{d})).$$
This means that  for all $t \in [0,T]$, $h(t,\cdot)\in L^{\infty}([0,T]\,;\,L^{2}_{k}(\R^{d}))$. Furthermore, since $g_{\rm in} \in L^{2}_{k+2|\g|}(\R^{d})$, we deduce from Proposition \ref{prop:ChernINT} that  }
 \begin{equation}\label{eq:LinftyLpg}
   \sup_{t \in [0,T]} \int_{\R^{d}} g^{2}(t,v) \langle v\rangle^{k+ {2|\g|}} \d v +\frac{K_{0}}{2}\int_{0}^{T}\d \tau\int_{\R^{d}} \left|\nabla \left(\langle v\rangle^{\frac{ {k+|\g|}}{2}} g(\tau,v) \right)\right|^{2}\d v  < \infty\,.\end{equation}
We then investigate the evolution of  
 $$\mathscr{J}(t):=\int_{\R^{d}}\langle v\rangle^{k}h^{2}(t,v)\d v.$$ One has then
 \begin{equation*}\begin{split}
\dfrac{\d}{\d t}\mathscr{J}(t)&=2\int_{\R^{d}}\langle v\rangle^{k} h(t,v)\,\partial_{t}h(t,v)\d v\\
&=-2\int_{\R^{d}} \mathcal{A}[f] \nabla h \cdot \nabla \left(h\langle v\rangle^{k}\right)\d v -2\int_{\R^{d}}\mathcal{A}[h]\nabla g \cdot \nabla  \left( h\langle v\rangle^{k}\right)\d v \\
&\phantom{+} +2\int_{\R^{d}}h\, \bm{b}[f] \cdot \nabla  \left(h\langle v\rangle^{k}\right)\d v + 2\int_{\R^{d}}g\bm{b}[h] \cdot \nabla  \left(h\langle v\rangle^{k}\right)\d v\\
&=\mathscr{I}_{1}+\mathscr{I}_{2}+\mathscr{I}_{3}+\mathscr{I}_{4}. \end{split}\end{equation*}
 As in the proof of Lemma \ref{prop:Lp-esti} one has 
 $\mathscr{I}_{1}=\mathscr{I}_{1,1}+\mathscr{I}_{1,2}$ with
 $$\mathscr{I}_{1,1}=-2\int_{\R^{d}}\langle v\rangle^{k}\mathcal{A}[f]\nabla h\cdot \nabla h\d v, \qquad \mathscr{I}_{1,2}=-2k\int_{\R^{d}}\langle v\rangle^{k-2} h\mathcal{A}[f]\nabla h\cdot v\d v.$$
 From the diffusive properties of $\mathcal{A}[f]$,
 $$\int_{\R^{d}}\langle v\rangle^{k} \mathcal{A}[f] \nabla h\cdot \nabla h \d v \geq K_{0}\int_{\R^{d}}\langle v\rangle^{k+\g}\,|\nabla h|^{2} \d v.$$ 
  Now, 
$$\mathscr{I}_{1,2}=k\int_{\R^{d}}h^{2}(t,v)\nabla \cdot \left[\mathcal{A}[f]\langle v\rangle^{k-2}v\right]\d v.$$
Therefore, using the reasoning and notations of Lemma \ref{prop:Lp-esti} 
$$\mathscr{I}_{1,2}=k\int_{\R^{d}}h^{2}\langle v\rangle^{k-2}\left(\bm{b}[f]\cdot v\right)\d v +k \int_{\R^{d}}\langle v\rangle^{k-4}h^{2}\mathrm{Trace}\left(\mathcal{A}[f]\cdot \bm{A}(v)\right)\d v.$$
Using several integration by parts and using that $\nabla \cdot \bm{b}[f]=-\bm{c}_{\g}[f]$, one also shows that
$$
\mathscr{I}_{3}= {\int_{\R^{d}}\langle v\rangle^{k}\,h^{2}\bm{c}_{\g}[f]\d v } +k\int_{\R^{d}}\langle v\rangle^{k-2}h^{2}\left(\bm{b}[f]\cdot v\right)\d v$$
therefore
\begin{multline*}
\mathscr{I}_{1,2}+\mathscr{I}_{3}=k \int_{\R^{d}}\langle v\rangle^{k-4}h^{2}\mathrm{Trace}\left(\mathcal{A}[f]\cdot \bm{A}(v)\right)\d v
+  {  \int_{\R^{d}}\langle v\rangle^{k}h^{2}\bm{c}_{\g}[f]\d v} \\
+2k\int_{\R^{d}}\langle v\rangle^{k-2}h^{2}\left(\bm{b}[f]\cdot v\right)\d v.\end{multline*}
Using the estimate of $\mathrm{Trace}\left(\mathcal{A}[f]\cdot \bm{A}(v)\right)$ in Lemma \ref{prop:Lp-esti} and estimate \eqref{eq:I2B} for estimating the integral involving $\bm{b}[f]$, we deduce that
$$
\mathscr{I}_{1,2}+\mathscr{I}_{3} \leq  {C_{d,k,\g}}\left(\int_{\R^{d}}\langle v\rangle^{k-2}h^{2}\bm{c}_{\g+2}[f]\d v+\int_{\R^{d}}\langle v\rangle^{k-1}h^{2}\bm{c}_{\g+1}[f]\d v\right) 
+  {\int_{\R^{d}}\langle v\rangle^{k} h^{2}\bm{c}_{\g}[f]\d v.}
$$
Using Lemma \ref{lem:compa}, we deduce  {in the case $-d<\g<-2$} that there exist $C=C(d,k,\g), c=c(d,k,\g) >0$ (depending on $\|f_{\rm in}\|_{L^1_2(\R^d)}$) such that 
$$\mathscr{I}_{1,2}+\mathscr{I}_{3} \leq C \int_{\R^{d}}\langle v\rangle^{k}h^{2}\overline{\bm{c}}_{\g}[f]\d v + c\,\int_{\R^{d}}\langle v\rangle^{k}\,h^{2}\d v.$$
Thus, observing that $\overline{\bm{c}}_{\g}[f] \leq \bm{c}_{\g}[f]$, we get
\begin{equation}\label{eq:I1I30A}
\mathscr{I}_{1}+\mathscr{I}_{3} \leq -2K_{0}\int_{\R^{d}}\langle v\rangle^{k+\g}\,|\nabla h|^{2} \d v 
+C \int_{\R^{d}}\langle v\rangle^{k} h^{2}{\bm{c}}_{\g}[f](v)\d v + c \int_{\R^{d}}\langle v\rangle^{k}h^{2}\d v.\end{equation}
In the case $\g=-d$, Lemma \ref{lem:compa} implies that there exist $C=C(d,k), \tilde{C}=\tilde{C}(d,k), c=c(d,k) >0$ such that 
  $$\mathscr{I}_{1,2}+\mathscr{I}_{3} \leq C \int_{\R^{d}}\langle v\rangle^{k}h^{2}{\bm{c}}_{-d}[f]\d v + \tilde{C} \int_{\R^{d}}\langle v\rangle^{k-1}h^{2}\overline{\bm{c}}_{1-d}[f]\d v + c\,\int_{\R^{d}}\langle v\rangle^{k-1}\,h^{2}\d v.$$
  Note that \eqref{c1-d} holds when Assumption \textit{\textbf{Hyp. 1}} is in force and implies that 
  $$ \int_{\R^{d}}\langle v\rangle^{k-1}h^{2}\overline{\bm{c}}_{1-d}[f]\d v
  \le (d-1) |\S^{d-1}| \|f(t)\|_{{\infty}} \int_{\R^{d}}\langle v\rangle^{k}h^{2}\d v.$$
  Now, under Assumption \textit{\textbf{Hyp. 2}}, since $d=3$ (see Remark \ref{rmk:constraint}), we have $\overline{\bm{c}}_{1-d}[f(t)]\leq \bm{c}_{-2}[f(t)]$ and we deduce from Proposition \ref{propo:GG2} that there exist $C_0>0$ such that, for any $\e>0$ ($N>2$ being defined in the statement of the theorem),
 \begin{multline*} {\int_{\R^{d}}\langle v\rangle^{k-1}\overline{\bm{c}}_{1-d}[f(t)](v)\,h^{2}(t,v)\d v\, }  \leq \e\;\int_{\R^{d}}\left|\nabla \left(\langle v\rangle^{\frac{k-d}{2}} h(t,v)\right)\right|^{2}\d v \\
  +C_{0}\exp\left(\e^{-\frac{N}{N-2}}\,\lm_{N}(f(t))^{\frac{N}{N-2}}\right)\int_{\R^{d}}h(t,v)^{2}\langle v\rangle^{k-d}\d v \,,
   \end{multline*}
  Thus, in the case $\g=-d$, for any $\e>0$, 
\begin{multline}\label{eq:I1I30B}
\mathscr{I}_{1}+\mathscr{I}_{3} \leq -2K_{0}\int_{\R^{d}}\langle v\rangle^{k+\g}\,|\nabla h|^{2} \d v +\e\tilde{C}\;\int_{\R^{d}}\left|\nabla \left(\langle v\rangle^{\frac{k+\g}{2}} h(t,v)\right)\right|^{2}\d v \\
+C \int_{\R^{d}}\langle v\rangle^{k} h^{2}{\bm{c}}_{\g}[f](v)\d v + \Gamma_\e(t) \int_{\R^{d}}\langle v\rangle^{k}h^{2}\d v,\end{multline}
where $\Gamma_{\e}\in L^{1}([0,T])$. 

Finally, gathering \eqref{eq:I1I30A} and \eqref{eq:I1I30B}, we obtain for $-d\le \g<-2$ and for any $\e>0$,
\begin{multline}\label{eq:I1I30}
\mathscr{I}_{1}+\mathscr{I}_{3} \leq -2K_{0}\int_{\R^{d}}\langle v\rangle^{k+\g}\,|\nabla h|^{2} \d v +\e\tilde{C}\;\int_{\R^{d}}\left|\nabla \left(\langle v\rangle^{\frac{k+\g}{2}} h(t,v)\right)\right|^{2}\d v \\
+C \int_{\R^{d}}\langle v\rangle^{k} h^{2}{\bm{c}}_{\g}[f](v)\d v + \Gamma_\e(t) \int_{\R^{d}}\langle v\rangle^{k}h^{2}\d v,\end{multline}
with $\Gamma_{\e} \in L^{1}([0,T])$.

One now looks at $\mathscr{I}_{2}$, which is more delicate. One has again $\mathscr{I}_{2}=\mathscr{I}_{2,1}+\mathscr{I}_{2,2}$, with
$$\mathscr{I}_{2,1}=-2\int_{\R^{d}}\langle v\rangle^{k} \mathcal{A}[h]\nabla g \cdot \nabla h\,\d v, \qquad \qquad \mathscr{I}_{2,2}=-2k\int_{\R^{d}}\langle v\rangle^{k-2}\,h\,\mathcal{A}[h]\nabla g\cdot v\d v\,.$$
Writing $\langle v\rangle^{k-2} h \mathcal{A}[h]\nabla g \cdot v=\left(\langle v\rangle^{\frac{k-2}{2}} h \,v\right)\cdot \left(\langle v\rangle^{\frac{k-2}{2}}\mathcal{A}[h]\nabla g\right)$ and using Young's inequality, we deduce that
$$|\mathscr{I}_{2,2}| \leq k \int_{\R^{d}}\langle v\rangle^{k}h^{2}(t,v)\d v + k\int_{\R^{d}}\langle v\rangle^{k}\left|\mathcal{A}[h]\right|^{2}|\nabla g|^{2}\d v\,.$$
Similarly, with $\langle v\rangle^{k} \mathcal{A}[h]\nabla g \cdot \nabla h=\left(\langle v\rangle^{\frac{k+\g}{2}} \nabla h\right)\cdot \left(\langle v\rangle^{\frac{k-\g}{2}} \mathcal{A}[h]\nabla g\right)$, we deduce from Young's inequality that, for any $\delta >0$,
$$
|\mathscr{I}_{2,1}| \leq \delta \int_{\R^{d}}\langle v\rangle^{k+\g} \,|\nabla h|^{2}\d v + \frac{1}{\delta} \int_{\R^{d}}\langle v\rangle^{k-\g} |\mathcal{A}[h]|^{2}\,|\nabla g|^{2}\d v.$$
We deduce that there exists $C_{k} > 0$ such that, for any $\delta \in ]0,1]$, 
\begin{multline*}
|\mathscr{I}_{2}| \leq \delta \int_{\R^{d}}\langle v\rangle^{k+\g}\,|\nabla h|^{2}\d v + \frac{C_{k}}{\delta}\bigg( \int_{\R^{d}}\langle v\rangle^{k} h^{2}(t,v)\d v
+  \|\mathcal{A}[h]\|_{\infty}^{2}\int_{\R^{d}}\langle v\rangle^{k-\g}|\nabla g|^{2}\d v\bigg).\end{multline*}
Now, from Lemma \ref{lem:AbF}, Eq. \eqref{eq:AF0},
 there is $C >0$ such that, when $q > \frac{d}{d+2+\g}$ and $m > d\left(1-\frac{1}{q}\right)$,
$$\|\mathcal{A}[h]\|_{\infty} \leq C \|\langle \cdot \rangle^{m}h\|_{q}. $$
%, \qquad  q > \frac{d}{d+2+\g}, \qquad m > d\left(1-\frac{1}{q}\right).$$
Applying this with  {$q=2$}, $m=\frac{k}{2}$, and this shows that, if  {$k>d$},
$$\|\mathcal{A}[h]\|_{\infty}^{2} \leq C^{2}\int_{\R^{d}}\langle v\rangle^{k}h^{2}(t,v)\d v , $$
and we deduce therefore that, for any $\delta \in ]0,1]$,
\begin{equation}\label{eq:I2}
|\mathscr{I}_{2}| \leq \delta   \int_{\R^{d}}\langle v\rangle^{k+\g} \,|\nabla h|^{2}\d v + {\Xi}_{\delta}(t)\int_{\R^{d}}\langle v\rangle^{k}h^{2}\d v ,
\end{equation}
where 
$${\Xi}_{\delta}(t)=\delta^{-1}C_{k}\left[1+C^{2}\int_{\R^{d}}\langle v\rangle^{k-\g}|\nabla g(t,v)|^{2}\d v\right].$$
Notice that \eqref{eq:LinftyLpg} implies that for any $\delta \in ]0,1]$,
$${\Xi}_{\delta} \in L^{1}([0,T]). $$
 One proceeds in the exact same way  for $\mathscr{I}_{4}=\mathscr{I}_{4,1}+\mathscr{I}_{4,2}$, with
 $$\mathscr{I}_{4,1}=2\int_{\R^{d}}\langle v\rangle^{k} g(t,v)\bm{b}[h] \cdot \nabla h \d v, \qquad \mathscr{I}_{4,2}=2k\int_{\R^{d}}\langle v\rangle^{k-2}h(t,v)g(t,v)\bm{b}[h] \cdot v\d v.$$
{To estimate $|\bm{b}[h]|$, we need to distinguish between the  case $2 > \frac{d}{d+\g+1}$ for which can apply \eqref{eq:BF0} to estimate $\|\bm{b}[h]\|_{\infty}$ in terms of $\|\langle \cdot \rangle^{m}h\|_{2}$ $(m > \frac{d}{2})$, {and the case $2 \leq \frac{d}{d+\g+1}$ for which we resort to \eqref{eq:BF1} or \eqref{eq:BF12}.} We focus here on this second case, the case $\bm{b}[h] \in L^{\infty}(\R^{d})$ being simpler (one can follow the line of the previous estimate for $\|\mathcal{A}[h]\|_{\infty}$, we leave it to the reader).} We treat in details first the case $2 < \frac{d}{d+\g+1}$. Clearly, 
$$\left|\mathscr{I}_{4,1}\right| \leq {2}\int_{\R^{d}}\left(\langle v\rangle^{\frac{k+\g}{2}} |\nabla h|\right)\,\left(g\langle v\rangle^{\frac{k-\g}{2}} \right)\,|\bm{b}[h]|\d v.$$
 {Let $p_{\g}^{*}$ be defined as in \eqref{eq:BF1} with the choice of $q=2$, i.e. $p^{*}_{\g}=\frac{2d}{d-2(1+\g+d)} > 2$. Let $p$ be such that $\frac{1}{p_{\g}^{*}}+\frac{1}{2}+\frac{1}{p}=1$, that is $p=\frac{d}{1+\g+d}$. One deduces from H\"older's inequality that}
$$\left|\mathscr{I}_{4,1}\right| \leq {2} \|\bm{b}[h]\|_{p^{*}_{\g}}\,\left\|\langle \cdot \rangle^{\frac{k+\g}{2}} |\nabla h|\,\right\|_{2}
\left\|\langle \cdot \rangle^{\frac{k-\g}{2}} g\right\|_{p}.$$
Now, a simple use of Young's inequality yields, for any $\delta>0$,
$$\left|\mathscr{I}_{4,1}\right| \leq \delta\int_{\R^{d}}\langle v\rangle^{ k+\g } \left|\nabla h(t,v)\right|^{2}\d v +\frac{1}{\delta}\left\|\bm{b}[h]\right\|_{p^{*}_{\g}}^{2}\left\|\langle \cdot\rangle^{\frac{k-\g}{2}}\,g(t) \right\|_{p}^{2}.$$
We now use \eqref{eq:BF1} to deduce that $\|\bm{b}[h]\|_{p^{*}_{\g}}^{2} \leq C(f_{\rm in},g_{\rm in})\|h\|_{2}^{2}$,
 where $C(f_{\rm in},g_{\rm in})$ depends on $d,\g$ and $\|f_{\rm in}\|_{1},\|g_{\rm in}\|_{1}.$ It remains only to estimate 
$$\left\|\langle \cdot\rangle^{\frac{k-\g}{2}}\,g(t) \right\|_{p}^{2}, \qquad  { p=\frac{d}{1+\g+d}}.$$
One checks easily that $p < \frac{2d}{d-2}$ {under the assumption $2>\frac{d}{d+\g+2}$, while  $p >2$ clearly holds by assumption. Then, by interpolation and Sobolev embedding, there exist $C >0$ and $\alpha \in (0,1)$ such that 
$$\left\|\langle \cdot \rangle^{\frac{k-\g}{2}}g(t)\right\|_{p} \leq C\,\left\|\langle \cdot \rangle^{\frac{k-\g}{2}}g(t)\right\|_{2}^{1-\alpha}\,\left\|\nabla \left(\langle \cdot \rangle^{\frac{k-\g }{2}}g(t)\right)\right\|_{2}^{\alpha}.$$}
Thanks to Young's inequality, we deduce that
$$\left\|\langle \cdot \rangle^{\frac{k-\g}{2}}g(t)\right\|_{p}^{2} \leq C_{0}\left( {\left\|\langle \cdot \rangle^{\frac{k-\g}{2}}g(t)\right\|_{2}^{2}} + \int_{\R^{d}}\left|\nabla \left(\langle v\rangle^{\frac{k-\g}{2}}g(t,v)\right)\right|^{2}\d v\right).$$
We showed that, for any $\delta >0$, 
$$
\left|\mathscr{I}_{4,1}\right| \leq  \delta\int_{\R^{d}}\langle v\rangle^{ k+\g }\left|\nabla h(t,v)\right|^{2}\d v + \Theta_{\delta}(t)\int_{\R^{d}}\langle v\rangle^{k+\g}h^{2}(t,v)\d v ,
$$ where
$$\Theta_{\delta}(t)=\delta^{-1}C_{d,\g,k}\left(  {\,\left\|\langle \cdot \rangle^{\frac{k-\g}{2}}g(t)\right\|_{2}^{2}}+\int_{\R^{d}}\left|\nabla \left(\langle v\rangle^{\frac{k-\g}{2}}g(t,v)\right)\right|^{2}\d v\right) ,$$
for some positive constant $C_{d,\g,k}$ depending on $d,\g,k$ and $\|f_{\rm in}\|_{1},\|g_{\rm in}\|_{1}$.  Notice that, according to   \eqref{eq:LinftyLpg}, one has $\Theta_{\delta} \in L^{1}([0,T])$. In the same way, 
\begin{equation*}\begin{split}
\left|\mathscr{I}_{4,2}\right| &\leq  2k\|\bm{b}[h]\|_{p^{*}_{\g}}\,\left\|\langle \cdot \rangle^{\frac{k-1}{2}}h\,\right\|_{2}
\left\|\langle \cdot \rangle^{\frac{k-1}{2}}g\right\|_{p} \\
&\leq k \int_{\R^{d}} {\langle v\rangle^{k-1} }h^{2}(t,v)\,\d v + k\|\bm{b}[h]\|_{p^{*}_{\g}}^{2}
\left\|\langle \cdot \rangle^{\frac{k-1}{2}}g\right\|_{p}^{2}\,.\end{split}\end{equation*}
As before, this yields then the estimate
$$\left|\mathscr{I}_{4,2}\right| \leq \Theta(t)\,\int_{\R^{d}}\langle v\rangle^{k}h^{2}(t,v)\,\d v$$
where
$$\Theta(t)=C_{d,\g,k}\left(1+ {\,\left\|\langle \cdot \rangle^{\frac{k-1}{2}}g(t)\right\|_{2}^{2}}+\int_{\R^{d}}\left|\nabla \left(\langle v\rangle^{\frac{k-1}{2}}g(t,v)\right)\right|^{2}\d v\right)$$
for some positive constant $C_{d,\g,k}>0$ depending on $d,\g,k$ and $\|f_{\rm in}\|_{L^{1}_{2}}$. In particular, $\Theta \in L^{1}([0,T])$.
Combining these two estimates,  for any $\delta >0$, there is $\Lambda_{\delta} \in L^{1}([0,T])$ such that
\begin{equation}\label{eq:I4}
|\mathscr{I}_{4}| \leq \delta  \int_{\R^{d}}\langle v\rangle^{ k+\g } \left|\nabla h(t,v)\right|^{2}\d v+{\Lambda}_{\delta}(t)\int_{\R^{d}}\langle v\rangle^{k}h^{2}\d v\,.\end{equation}
Choosing  {$\delta=\delta_{0}= \min(1,\frac{K_{0}}{2})$}
 in \eqref{eq:I2}-\eqref{eq:I4} and combining this with \eqref{eq:I1I30}, one sees that 
\begin{multline*}
\dfrac{\d}{\d t}\int_{\R^{d}}\langle v\rangle^{k} h^{2}(t,v)\d v +K_{0}\int_{\R^{d}}\langle v\rangle^{k+\g} \left|\nabla h(t,v)\right|^{2}\d v \leq  {\e\tilde{C}\;\int_{\R^{d}}\left|\nabla \left(\langle v\rangle^{\frac{k+\g}{2}} h(t,v)\right)\right|^{2}\d v}\\
+ {\bm{\Xi}_{\e}(t) \int_{\R^{d}}\langle v\rangle^{k} h^2(t,v)\d v } +C \int_{\R^{d}}\langle v\rangle^{k}h^{2}(t,v)\overline{\bm{c}}_{\g}[f(t)](v)\d v \end{multline*}
with 
$$ {\bm{\Xi}_{\e}(t)={\Lambda}_{\delta_{0}}(t)+{\Xi}_{\delta_{0}}(t)+\Gamma_{\e}(t), \qquad \bm{\Xi}_{\e} \in L^{1}([0,T])}.$$
 {Since $f_{\rm in} \in L^1_{N}(\R^d)$, we may now use} the $\e$-Poincar\'e inequality \eqref{eq:estimatcLP-ft} with $\phi^{2}=\langle \cdot\rangle^{k}h^{2}$.   {Notice that the definition of $\overline{\nu}$ in \eqref{eq:nuq} with $p=2$ yields $\bar{\nu}=\frac{d|\g|+k\g+2k}{d+2(\g+2)}$ if $k \leq 2|\g|$ while $\bar{\nu}=\frac{(d-k)|\g|}{d+2\g}$ if $k \geq 2|\g|$.} We deduce that, for any $\e >0$, 
\begin{multline*}
\dfrac{\d}{\d t}\int_{\R^{d}}\langle v\rangle^{k}h^{2}(t,v)\d v + K_{0}\int_{\R^{d}}\langle v\rangle^{k+\g}\left|\nabla h(t,v)\right|^{2}\d v\\
\leq C\e \int_{\R^{d}}\left|\nabla \left(\langle v\rangle^{\frac{k+\g}{2}} h(t,v)\right)\right|^{2}\d v
+C_{T}(f_{\rm in},\e)\int_{\R^{d}}\langle v\rangle^{k+\g} h^{2}(t,v)\d v\\
+\bm{\Xi}_{\e}(t) \int_{\R^{d}}\langle v\rangle^{k}h^{2}(t,v)\d v ,
\end{multline*}
where $C_{T}(f_{\rm in},\e) >0$ is a positive constant depending only on $\e,T,\|f_{\rm in}\|_{L^{2}_{k}},\g$. Recalling that there are $c_{0},c_{1} >0$ such that 
% \begin{multline}\label{aut}
$$\int_{\R^{d}}\langle v\rangle^{k+\g}\,|\nabla h|^{2}\d v \geq  c_{0}\int_{\R^{d}}\left|\nabla \left(\langle v\rangle^{\frac{k+\g}{2}}h\right)\right|^{2}\d v -c_{1}\int_{\R^{d}}\langle v\rangle^{k+\g-2}h^{2}\d v , $$
%\end{multline}
one chooses then $\e=\e_{0} > 0$ such that $C\e_{0}=\frac{1}{2}K_{0}c_{0}$ to deduce that
\begin{multline*}
\dfrac{\d}{\d t}\int_{\R^{d}}\langle v\rangle^{k}h^{2}(t,v)\d v +\frac{1}{2}K_{0}\textcolor{olive}{c_0}\int_{\R^{d}}\left|\nabla \left(\langle v\rangle^{\frac{k+\g}{2}}h(t,v)\right)\right|^{2}\d v \leq \bm{\Lambda}(t)\int_{\R^{d}}\langle v\rangle^{k}h^{2}(t,v)\d v ,\end{multline*}
with 
$$\bm{\Lambda}(t)=\bm{\Xi}_{\e_{0}}(t)+ {\frac{1}{2}K_{0}\,c_{1} } +C_{T}(f_{\rm in},\e_{0}), \qquad \bm{\Lambda} \in L^{1}([0,T]).$$
We conclude by a Gronwall argument that, for any $t \in [0,T]$,
$$\int_{\R^{d}}\langle v\rangle^{k}h^{2}(t,v)\d v \leq \bm{C}_{T}\int_{\R^{d}}\langle v\rangle^{k}h^{2}(0,v)\d v, \qquad \bm{C}_{T}:=\exp\left(\int_{0}^{T}\bm{\Lambda}(t)\d t\right),$$
which gives the result. 
\medskip

The case $\frac{d}{d+\g+1}=2$ is treated in the same way using now the weighted estimate \eqref{eq:BF12} for $q=2$. Given $\frac{2d}{d+2} <  \ell < 2$ one defines $\bar{p}^{*}_{\g}=\frac{2\ell}{2-\ell}$ and for  $p$ such that $\frac{1}{\bar{p}_{\g}^{*}}+\frac{1}{2}+\frac{1}{p}=1$, that is $p=\ell'=\frac{\ell}{\ell-1}$, we get, as before,  for any $\delta>0$,
\begin{equation*}\begin{split}
\left|\mathscr{I}_{4,1}\right| &\leq \delta\int_{\R^{d}}\langle v\rangle^{ k+\g } \left|\nabla h(t,v)\right|^{2}\d v +\frac{1}{\delta}\left\|\bm{b}[h]\right\|_{p^{*}_{\g}}^{2}\left\|\langle \cdot\rangle^{\frac{k-\g}{2}}\,g(t) \right\|_{p}^{2}\\ 
&  {\leq  \delta\int_{\R^{d}}\langle v\rangle^{ k+\g } \left|\nabla h(t,v)\right|^{2}\d v +\frac{C}{\delta}\|\langle \cdot \rangle^{m}h\|_{2}^{2} \left\|\langle \cdot\rangle^{\frac{k-\g}{2}}\,g(t) \right\|_{p}^{2} }
\end{split}\end{equation*}
for $m > |1+\g|=\frac{d}{2}$, where we used \eqref{eq:BF12} in the case $q=2$. The estimate for $\left\|\langle \cdot\rangle^{\frac{k-\g}{2}}\,g(t) \right\|_{p}^{2}$ with $p=\frac{\ell}{\ell-1}$ is made as before since $2 < p < \frac{2d}{d-2}$ (because $\frac{2d}{d+2} < \ell < 2$). Details are left to the reader.
\end{proof}
%\begin{rmq} If $\frac{d}{d+\g+1} < 2$, the above proof can be slightly modified using, for the study of $|\mathscr{I}_{4}|$ the estimate of $\|\bm{b}[h]\|_{\infty}^{2}$ given in \eqref{eq:BF0}. Details are left to the reader.\end{rmq}
 
As a consequence, we immediately deduce our main uniqueness result as stated in the Introduction.

\begin{proof}[Proof of Theorem \ref{theo:unique}] The proof is a simple consequence of the above stability inequality  and the fact that $g_{\rm in}=f_{\rm in}$.\end{proof}

\section{Further comments}\label{sec:comm}
  
In this last section we give an informal discussion of some additional features of our contribution. Specifically, we explain how the results of the previous sections allow to recover well-known results in the case of moderately soft potentials $-2 \leq \g <0$, and then we discuss the endpoint case of Prodi-Serrin criterion for which $r=\infty$.

\subsection{About the moderate soft potential case}\label{sec:moderate} We illustrate here how Prodi-Serrin's criterion is satisfied by any weak solution constructed with an approximation procedure, as in \cite{DesvJFA} to the Landau equation, in the moderate soft potential case:
$$-2 \leq \g < 0.$$
It is already a well-known fact that Landau equation is globally well-posed in this case (weak solutions are global and unique under suitable assumptions) and that apperance/propagation of $L^{p}$-norm occurs in this case. Moreover, as shown in the recent contribution \cite{ABL}, such appearance of $L^{p}$-bounds implies the appearance of pointwise bounds for solutions to \eqref{LFD}. The key point which makes the Prodi-Serrin criterion satisfied is the \emph{a priori} estimate related to the weighted Fisher information obtained by the third author in \cite{DesvJFA}, see Corollary \ref{cor:Desv} which, for $-2 \leq \g < 0$ is strong enough to ensure one of the Prodi-Serrin criterion to hold.

\begin{prop} Let $f_{\rm in}$ and $f=f(t,v)$ a weak solution to equation \eqref{eq:Landau} in the sense of Definition \ref{defi:weak} . Then the following holds:
\begin{enumerate}
\item Assume that $-2 < \g <0$ and $f_{\rm in} \in L^{1}_{ {s}}(\R^{d})$ with $s=\frac{\g^{2}}{2(\g+2)}$, then
$$f \in L^{1}([0,T]\,,\,L^{\frac{d}{d+\g}}(\R^{d})), \qquad \forall T >0 . $$
\item For $-2 \leq \g < 0$, assume that there is some $\delta >0$ such that $f_{\rm in} \in L^{1}_{|\g|+\delta}(\R^{d})$. Then, there is a pair $(r,q)$ with $r >1$, $1 < q <\frac{d}{d+\g}$
 and $\frac{2}{r}+\frac{d}{q}=d+\g+2$ such that
$$\langle \cdot \rangle^{|\g|}f \in L^{r}([0,T]\,,\,L^{q}(\R^{d})), \qquad \forall T >0.$$
\end{enumerate}
\end{prop} 

\begin{proof} Recall that  Corollary \ref{cor:Desv} ensures that the weak solution $f$ satisfies
\begin{equation}\label{eq:Fis}
\int_{0}^{T}\d t \int_{\R^{d}}\left|\nabla \left(\langle v\rangle^{\frac{\g}{2}}\sqrt{f(t,v)}\right)\right|^{2}\d v < \infty\,. \end{equation}
We begin with Assumption 1 (and $-2 < \g <0$). 
%Setting 
%$$g(t,v)=\langle v\rangle^{\frac{\g}{2}}\sqrt{f(t,v)}, \qquad \forall t \geq0,\quad v \in \R^{d}$$
%one has $g \in L^{2}([0,T]\,,\,\dot{\H}^{1}(\R^{d}))$. Recalling the 
Thanks to Sobolev inequality
$$\|h\|_{\frac{2d}{d-2}}^2 \leq C_{\mathrm{Sob}}\|\nabla h\|_{2}^2, \qquad \forall h \in \dot{\H}^1(\R^{d}) $$
(and since $\g >-2$), we can use the interpolation (H\"older's) inequality
$$\|f(t)\|_{\frac{d}{d+\g}}=\|\sqrt{f(t)}\|_{\frac{2d}{d+\g}}^{2} \leq \left\|\langle \cdot \rangle^{a}\sqrt{f(t)}\right\|_{2}^{2(1-\alpha)}\,\left\|\langle \cdot\rangle^{\frac{\g}{2}}\sqrt{f(t)}\right\|_{\frac{2d}{d-2}}^{2\alpha}$$
with $\frac{d+\g}{2d}=\frac{1-\alpha}{2}+\frac{d-2}{2d}\alpha,$ $a(1-\alpha)+\frac{\g}{2}\alpha=0,$ which means $\alpha=\frac{|\g|}{2}$ and $a=\frac{\g^{2}}{2(2+\g)}$. This, combined with the above Sobolev inequality, implies that
$$\|f(t)\|_{\frac{d}{d+\g}} \leq C_{\mathrm{Sob}}^{\frac{|\g|}{2}}\left[\sup_{t \in [0,T]}\lm_{\frac{|\g|^{2}}{2+\g}}(f(t))\right]^{2-|\g|}\,\left\|\nabla \left(\langle \cdot \rangle^{\frac{|\g|}{2}}\sqrt{f(t)}\right)\right\|_{2}^{|\g|}.$$
Since $|\g| < 2$, we see from \eqref{eq:Fis} that  $\|f(\cdot)\|_{\frac{d}{d+\g}} \in L^{1}([0,T])$ and this proves the first point.
\medskip

Now, for the second Prodi-Serrin criterion, we still resort to \eqref{eq:Fis} but now with the weighted interpolation inequality
$$\|\langle \cdot \rangle^{|\g|}f(t)\|_{q} \leq \|\langle \cdot \rangle^{\beta}f(t)\|_{1}^{1-\theta}\,\left\|\langle \cdot \rangle^{\g}f(t)\right\|_{\frac{d}{d-2}}^{\theta} \leq \|\langle \cdot \rangle^{\beta}f(t)\|_{1}^{1-\theta}\,\left\|\langle \cdot \rangle^{\frac{\g}{2}}\sqrt{f(t)}\right\|_{ {\frac{2d}{d-2}}}^{2\theta} $$
with 
$$\frac{1}{q}=1-\theta+\frac{d-2}{d}\theta=1-\frac{2}{d}\theta,  \qquad |\g|=\beta(1-\theta)+\g\theta.$$ 
We assume that $1 < q < \frac{d}{d+\g}$ so that $0 < \theta < \frac{|\g|}{2} \leq 1$. We choose then 
$$\frac{1}{r}=1+\frac{\g}{2}+\theta <1,$$
so that
$$\|\langle \cdot \rangle^{|\g|}f(t)\|_{q}^{r} \leq C_{\rm Sob}^{r\theta}\,\lm_{\beta}(f(t))^{r(1-\theta)}\,\left\|\nabla \left(\langle \cdot \rangle^{\frac{|\g|}{2}}\sqrt{f(t)}\right)\right\|_{2}^{2r\theta}$$
while
$$\frac{d}{q}+\frac{2}{r}=\frac{d}{q}+2+\g+2\theta=d+\g+2 , $$
which makes $(r,q)$ an admissible Prodi-Serrin pair. Now, since $2r\theta \leq 2$, one sees that
$$\sup_{t \in [0,T]}\lm_{\beta}(f(t)) < \infty \implies \|\langle \cdot \rangle^{|\g|}f(t)\|_{q}
 \in L^{r}([0,T]).$$ 

 One notices now that  
{$\beta = |\g|\, \frac{1 + \frac{d}2\, (1 - q^{-1})}{1 + \frac{d}2\, (1 - q^{-1})}$. We see that if $\g >-2$,
we can get  $\beta = 2$ by choosing $q>1$ sufficiently close to $1$.  If $\g=-2$, we get (for some given $\delta>0$) 
$\beta = 2 + \delta$ by also choosing $q>1$ sufficiently close to $1$. This proves the result.}
\end{proof}
 
\begin{rmq} For $-2 < \g <0$, the second point here above means that Prodi-Serrin's criterion $L^{r}_{t}(L^{q})$ with $r >1$ is satisfied for some pair $(r,q)$ for \emph{any admissible} initial datum $f_{\rm in} \in L^{1}_{2}(\R^{d})$ whereas, for $\g=-2$, some additional moment assumption $f_{\rm in} \in L^{1}_{2+\delta}(\R^{d})$ is needed. Notice that, arguing as in \cite{ABL} and resorting to the de la Vall\'ee-Poussin criterion, this additional assumption can be removed. Even if the first point of the Proposition seems to yield a non optimal moment assumption on $f_{\rm in}$, we still believe it provides interesting information on the link between the Prodi-Serrin criterion in the endpoint case $r=1,q=\frac{d}{d+\g}$, and the Fisher information estimate of Corollary \ref{cor:Desv}.\end{rmq}

\subsection{About the critical endpoint $r=\infty,$ $q=\frac{d}{d+\g+2}$}\label{sec:endpoint}

We consider $d\in \N$, $d \ge 3$, and $\g \in [-d, -2)$.   In the endpoint case for which $r=\infty$ and $q=\frac{d}{d+\g+2}$, it is actually  possible to obtain a result similar to that of Theorem \ref{theo:PSNS}.  Indeed, under some smallness assumption on 
$$\|\langle \cdot \rangle^{|\g|}f\|_{L^{\infty}([0,T]\,,\,L^{\frac{d}{d+\g+2}}(\R^{d}))},$$
 the appearance and propagation of $L^{p}$-norms which is the cornerstone of our analysis still holds true:

\begin{prop}\label{theo:PSend} 
We consider $d\in \N$, $d \ge 3$ and 
{$\g \in [-d, -2)$}. Let $f_{\rm in}$ and $f=f(t,v)$ be a weak solution to equation \eqref{eq:Landau} in the sense of Definition \ref{defi:weak} with 
$f_{\rm in} \in L^{1}_{s}(\R^{d})$ $(s >2).$ Additionally, assume that there exists  $T>0$ such that
$$
\langle \cdot\rangle^{|\g|}\, f \in L^{\infty}([0,T]\;;\;L^{\frac{d}{d+\g+2}}(\R^{d})).
$$
Then, given $k \geq0$, $p > 1$, there exists some \emph{explicit} $\delta >0$ (depending on both $d, \g, k,p$) such that, if 
\begin{equation}\label{eq:Serrinend}
\sup_{t \in (0,T)}\left\|\langle\cdot\rangle^{|\g|}f(t,\cdot)\right\|_{\frac{d}{d+\g+2}} \leq  {\delta}K_{0}\, ,
\end{equation}
and  $$f_{\rm in} \in L^{p}_{k}(\R^{d}),$$ one has, for any $t \in [0,T]$,
\begin{equation}\label{eq:LinftyLp-end}
\int_{\R^{d}} f^p(t,v) \langle v\rangle^{k} \d v +C_{\delta,p,k}K_{0}\int_{0}^{t}\d \tau\int_{\R^{d}} \left|\nabla \left(\langle v\rangle^{\frac{k+\g}{2}} f(\tau,v)^{\frac{p}{2}}\right)\right|^{2}\d v  \leq \bm{C}(T,f_{\rm in},\delta)\,,\,\end{equation}
where $\,\bm{C}(T,f_{\rm in},\delta)$ is an explicit positive constant depending on $d,\g, T, p, k, K_0, \delta,\|f_{\rm in}\|_{L^{p}_{k}},\|f_{\rm in}\|_{L^{1}_{s}}$ and the Prodi-Serrin norms $\|f\|_{L^{\infty}_{t}(L_v^{\frac{d}{d+\g+2}})}$,
 while $C_{\delta,p,k}$ is an explicit constant depending only on $p,\delta$ and $k$. 
\end{prop}
 
\begin{proof} We set $q_{0}=\frac{d}{d+2+\g}$ and recall the evolution of $L^{p}$ norms given in \eqref{eq:lmS1}:
\begin{multline}\label{eq:lmS2}
\dfrac{\d}{\d t}\lM_{k,p}(t) + 2K(p)\lD_{k+\g,p}(t) \leq \bm{c}_{k,p}\lM_{k,p}(t)  +\bm{C}_{k,p}\int_{\R^{d}}\langle v\rangle^{k}\overline{\bm{c}}_{\g}[f(t)](v)f^{p}(t,v)\d v\,.
\end{multline}
As in the proof of Proposition \ref{mps_g}, the key point is to evaluate the singular term 
$$\int_{\R^{d}} { \langle\cdot\rangle^{k}} \,\overline{\bm{c}}_{\g}[f(t)]f^{p}(t,v)\d v$$ 
in terms of $\|\langle\cdot\rangle^{|\g|}f(t,\cdot)\|_{q_{0}}$.  We argue as in the proof of the  $\e$-Poincar\'e inequality to estimate
$$\int_{\R^{d}}\overline{\bm{c}}_{\g}[F]\psi^{2}\d v$$
with 
$$F=\langle\cdot\rangle^{-\g}f(t,v), \qquad \psi=\langle \cdot\rangle^{\frac{\g+k}{2}}f^{\frac{p}{2}}.$$
 {Indeed, since $\alpha_{1,\g} \leq \langle \vet \rangle^{\g} \leq \alpha_{2,\g}\langle v \rangle^{\g}$ when $|v-v_*| \le 1$ (for some constants $\alpha_{1,\g} , \alpha_{2,\g} >0$) , the estimate
on $\int_{\R^{d}}\overline{\bm{c}}_{\g}[F]\psi^{2}\d v$ yields an estimate on 
{$\int_{\R^{d}}  \langle \cdot \rangle^{k} \overline{\bm{c}}_{\g}[f(t,\cdot)]\, f^{p}(t,v) \d v$}.}

We can use Hardy-Littlewood-Sobolev inequality \eqref{eq:HLS}  to get
$$\int_{\R^{d}}\overline{\bm{c}}_{\g}[F]\psi^{2}\d v \leq C_{\textrm{HLS}}\|F\|_{{q_{0}}}\|\psi^{2}\|_{{r}}, \qquad \frac{1}{r}=2-\frac{1}{q_{0}}+\frac{\g}{d}=\frac{d-2}{d}.$$
This yields
$$
\int_{\R^{d}}\overline{\bm{c}}_{\g}[F]\psi^{2}\d v \leq C_{\textrm{HLS}}\|F\|_{{q_{0}}}\|\psi\|_{{2r}}^{2} \leq C_{d,\g}\|F\|_{{q_{0}}}\|\nabla \psi\|_{{2}}^{2}, $$
since $\frac{2d}{d-2}$ is the critical exponent for the Sobolev embedding and $C_{d,\g}=C_{\textrm{HLS}}C_{\text{Sob}}$.  Moreover, since $\psi= \langle \cdot\rangle^{\frac{\g+k}{2}}f^{\frac{p}{2}}$, we deduce from \eqref{eq:lmS2} that
$$
  \frac{\d}{\d t}\lM_{k,p}(t)+2K(p)\lD_{k+\g,p}(t) \leq 
\bm{C} _{k,p,d,\g}\|F(t)\|_{q_{0}}\lD_{k+\g,p}(t)\,+\bm{c}_{k,p}\,\lM_{k,p}(t)\,,$$
for some positive constant $\bm{C}_{k,p,d,\g}$ depending on $k,p,d,\g$, related to $\bm{C}_{k,p}C_{d,\g}$ and the constants $\alpha_{1,\g},\alpha_{2,\g}.$ Therefore, assuming that
\begin{equation}\label{eq:Kp}\sup_{t \in [0,T]}\|F(t)\|_{{q_{0}}} \leq {\delta\,K_{0}} \qquad \text{ with } \qquad \delta < {\frac{2(p-1)}{p\bm{C}_{k,p,d,\g}} },
\end{equation}
one obtains, recalling that $K(p)=K_{0}\frac{p-1}{p}$,
$$
\frac{\d}{\d t}\lM_{k,p}(t)+ { \left(2\frac{p-1}{p}-\delta\,   \bm{C}_{k,p,d,\g}  \right)K_{0}}\,\lD_{k+\g,p}(t) \leq \bm{c}_{k,p}\,\lM_{k,p}(t), \,
$$
which allows to conclude exactly as in the proof of Proposition \ref{prop:ChernINT}.  
%As for the proof of Proposition \ref{theo:Prodi-Serrin}, it is not difficult to modify the 
\medskip

We now briefly explain how to modify the above proof in the case $\g=-d$, when $d=3$. {We first observe that $\lM_{k,p}$ now satisfies the differential inequality
\begin{multline*}\frac{\d}{\d t}\lM_{k,p}(t)+2K(p)\lD_{k-3,p}(t) \leq  
\bm{c}_{k,p}\,\lM_{k,p}(t) + C_{k,p} \int_{\R^{3}} { \langle v\rangle^{k-1}} \,\overline{\bm{c}}_{-2}[f(t)](v)f^{p}(t,v)\d v\\
+ C_{k,p} \int_{\R^{3}} { \langle v\rangle^{k}} \,f^{p+1}(t,v)\d v \, .\end{multline*}
Then, we recall (see Proposition \ref{propo:GG2}) the estimate
$$  \int_{\R^{3}} { \langle v\rangle^{k-1}} \,\overline{\bm{c}}_{-2}[f(t)](v) f^{p}(t,v)\d v   \le \var\, \lD_{k-3,p}(t) + C(\var, {\bm m}_{\eta}(t))\, \lM_{k,p}(t) , $$
which holds for any $\var>0$ and any $\eta >2$. Finally, we notice
\begin{equation*}\begin{split}  
 \int_{\R^{3}} { \langle v\rangle^{k}} \,f^{p+1}(t,v)\d v&=\int_{\R^{3}} { \langle v\rangle^{3}} \,f(t,v)\left[\langle v\rangle^{\frac{k-3}{2}}f^{\frac{p}{2}}(t,v)\right]^{2}\d v \\
 &\le \|  \langle\cdot\rangle^{3} \,f(t)\|_{\frac{3}{2}} \, \|  \langle\cdot\rangle^{\frac{k-3}2} \,f^{\frac{p}2}(t) \|_{6}^2
 \le C_{\text{Sob}} \, \delta \,K_0\,  \lD_{k-3,p}(t)  . \end{split}\end{equation*}
This proves the result.}
\end{proof}

\begin{rmq} Notice that the above assumption \eqref{eq:Serrinend} is not optimal. Indeed, with the notations of the proof, and decomposing further 
$$F=F\ind_{F>R}+F\ind_{F\leq R}\,,$$
after reproducing the aforementioned computations, it is possible to replace \eqref{eq:Serrinend} with an estimate of the type 
\begin{equation}\label{eq:End}
\limsup_{R\to\infty}\sup_{t\in [0,T)}\|F(t)\ind_{F(t)>R}\|_{\frac{d}{d+2+\g}} \leq {\delta\,K_{0} }\,, \qquad F(t,v)=\langle v\rangle^{|\g|}f(t,v)\,,
%\,q_{0}=\frac{d}{d+2+\g},
\end{equation}
which is more general than some mere \emph{uniform integrability} of the family $(F^{q_{0}}(t))_{t\in [0,T]}$ since we do not assume the above $\limsup$ to vanish. In particular, the conclusion of Proposition \ref{theo:PSend} holds true for instance if
$$\langle \cdot \rangle^{|\g|}f(t,\cdot) \in L^{\infty}((0,T)\;;\;L^{\frac{d}{d+\g+2}}(\log L)^{\alpha}(\R^{d}) ) . $$
%with $\bm{X}=L^{\frac{d}{d+\g+2}}(\log L)^{\alpha}(\R^{d})$, $\alpha >0$. 
We do not elaborate more in this direction. 
\end{rmq}

Finally, for the Navier-Stokes equation, a sharper endpoint estimate than Theorem \ref{theo:PSNS} has been derived in \cite{seregin} by a compactness argument. This result has been made quantitative very recently in \cite{tao1,palasek} and appears interesting to inquire about the Landau equation counterpart.

\subsection*{Data availability}

No data was used for the research described in the article.

\appendix

 \section{Facts about solutions to the Landau equation}\label{sec:known}

We collect several mathematical results about the solutions to the Landau equation in the range of parameters we are dealing with here: $- d\leq \g <0$.  These results are meant to serve as a mathematical tool-box for the core of the paper.

\begin{defi} \label{defi:fin} 
Let $f_{\rm in} \in L^{1}_{2}(\R^{d}) \cap L\log L(\R^{d})$ be given and nonnegative with 
$$\varrho_{\rm in}:=\int_{\R^{d}}f_{\rm in}(v)\d v >0, \qquad E_{\rm in}:=\int_{\R^{d}}f_{\rm in}(v)|v|^{2}\d v,$$
and $$H(f_{\rm in}):=\int_{\R^{d}}f_{\rm in}(v)\log f_{\rm in}(v)\d v.$$
We say that $f \in \mathcal{Y}_{0}(f_{\mathrm{\rm in}})$ if $f\in L^{1}_{2}(\R^{d})$ and 
\begin{equation*}\label{eq0}
\int_{\R^{d}}f(v)\left(\begin{array}{c}1 \\v \\ |v|^{2}\end{array}\right)\d v=\int_{\R^{d}}f_{\mathrm{\rm in}}(v)\left(\begin{array}{c}1 \\v \\ |v|^{2}\end{array}\right)\d v,
\end{equation*}
and {$H(f) \leq H(f_{\mathrm{\rm in}}).$}
\end{defi}

{We observe that at the formal level, a solution $f=f(t,v)$ of the Landau equation with initial datum $f_{\rm in} \in L^{1}_{2}(\R^{d}) \cap L\log L(\R^{d})$ satisfies 
$$f(t) \in \mathcal{Y}_{0}(f_{\rm in}) \qquad \forall t \geq 0.$$}

For general estimates regarding the class of functions $\mathcal{Y}_{0}(f_{\rm in})$  {and the following lemma, we refer to \cite{ALL}:}

\begin{lem}\label{L2unif}
Let $0\leq f_{\mathrm{\rm in}}\in L^{1}_{2}(\R^{d}) \cap L\log L(\R^{d})$ be fixed as in Definition \ref{defi:fin}.   Then, the following hold:
\begin{enumerate}
\item For any  $f \in \mathcal{Y}_{0}(f_{\mathrm{\rm in}})$, it holds that
\begin{equation}\label{e0}
\int_{|v|\leq R} f \, \d v \geq  \varrho_{\rm in}\left(1-\frac{E_{\rm in}}{\varrho_{\rm in}R^{2}}\right)  \qquad \forall R >0.
\end{equation}
\smallskip
\item For any $\delta >0$ there exists $\eta(\delta)>0$ depending only on $\|f_{\rm in}\|_{L^{1}_{2}}$ and  $H(f_{\mathrm{\rm in}})$ such that for any $f \in \mathcal{Y}_{0}(f_{\mathrm{\rm in}})$ and measurable set $A\subset \R^{d}$
\begin{equation}\label{Lem6DV}
|A|\leq \eta(\delta) \Longrightarrow \int_A f \, \d v \leq \delta\,.
\end{equation}
\end{enumerate}
\end{lem}

The aforementioned estimates are the key for the proof of the coercivity estimate for the matrix $\mathcal{A}[f]$ given in Remark  \ref{diffusion}.  Namely, one has the following

\begin{lem}\label{diffusionA}
Let $0\leq f_{\mathrm{\rm in}}\in L^{1}_{2}(\R^{d}) \cap L\log L(\R^{d})$ be fixed.  Then, there exists a  constant $K_{0} > 0,$ depending on $d$, $\g$, $\rho_{\rm in}$, $E_{\rm in}$, $H(f_{\rm in})$  such that
$$ \sum_{i,j} \, \mathcal{A}_{i,j}[f](v) \, \xi_i \, \xi_j 
\geq K_{0} \langle v \rangle^{\g} \, |\xi|^2 , \qquad \forall\, v,\, \xi \in \R^{d},
$$
holds for any  $f \in \mathcal{Y}_{0}(f_{\mathrm{\rm in}})$.
\end{lem}

\begin{proof} We show here how the conclusion of Lemma \ref{L2unif} allows to deduce the result. We follow the proof of \cite[Proposition 2.3]{ALL} given in dimension $d=3$ and explain where some changes have to be made to treat general dimension $d\ge	3.$ First, one observes that it is enough to show the result for $\xi \in\R^{d}$ with $|\xi|=1$.  We fix such $\xi$ and, for $\theta \in  (0,\frac{\pi}{2})$, we introduce
$$D_{\theta,\xi}(v)=\left\{\vet \in \R^{d}\,,\,\left|\xi \cdot (v-\vet)\right| \geq |v-\vet|\cos\theta\right\}.$$
Given $v,\vet \in \R^{d}$, $v \neq \vet$ we set for simplicity $u=v-\vet$, $\widehat{u}=\frac{u}{|u|}$ and notice that
$$\sum_{i,j}a_{i,j}(v-\vet)\xi_{i}\xi_{j}=|u|^{\g+2}\sum_{i,j}\left(\delta_{i,j}-\widehat{u}_{i}\widehat{u}_{j}\right)\xi_{i}\xi_{j}=|u|^{\g+2}\left(1-\left(\widehat{u}\cdot \xi\right)^{2}\right).$$
Therefore, if $v \in \R^{d}$  it holds
$$\sum_{i,j}a_{i,j}(v-\vet)\xi_{i}\xi_{j} \geq |u|^{\g+2}\sin^{2}\theta, \qquad \forall \vet \notin D_{\theta,\xi}(v).$$
For any $f \in \mathcal{Y}_{0}(f_{\rm in})$ and any $v \in \R^{d}$, we have then, for any $\theta \in \left(0,\frac{\pi}{2}\right)$ and any $R >0$
\begin{equation*}\begin{split}
\sum_{i,j} \, \mathcal{A}_{i,j}[f](v) \, \xi_i \, \xi_j  %&\geq  \sum_{i,j} \,\int_{\R^{d}} \mathbf{1}_{|\vet| \leq R}f(\vet)a_{i,j}(v-\vet) \, \xi_i \, \xi_j \d \vet \\
&\geq  \sum_{i,j} \,\int_{B_{R}\setminus D_{\theta,\xi}(v)}   f(\vet)a_{i,j}(v-\vet) \, \xi_i \, \xi_j \d \vet\end{split}\end{equation*}
where $B_{R}=\{\vet \in \R^{d}\,,\,|\vet| \leq R\}.$  Thus
\begin{equation}\label{eq:AfBR}
\sum_{i,j} \, \mathcal{A}_{i,j}[f](v) \, \xi_i \, \xi_j\geq \sin^{2}\theta\int_{B_{R}\setminus D_{\theta,\xi}}f(\vet)|u|^{\g+2} \d \vet.\end{equation}
We now fix
$R > \sqrt{\frac{2E_{\rm in}}{\varrho_{\rm in}}}$
so that \eqref{e0} reads $\int_{B_{R}}f(\vet)\d\vet \geq \frac{\varrho_{\rm in}}{2}$ and
\begin{equation}\label{eq:Af}\int_{B_{R}\setminus D_{\theta,\xi}(v)}f(\vet)\d\vet  \geq \frac{1}{2}\varrho_{\rm in}-\int_{B_{R}\cap D_{\theta,\xi}(v)}f(\vet)\d \vet.\end{equation}
One argues as in \cite{DeVi1} to estimate $\left|B_{R}\cap D_{\theta,\xi}(v)\right|.$ This is the only place in the proof in which dimension $d$ plays a role, the proof in \cite{DeVi1} being given in $d=3.$ Notice that $B_{R}\cap D_{\theta,\xi}(v)$ is the intersection of a cone with a ball. It can be estimated from above by the intersection of the cylinder with the ball $B_{R}$ where the cylinder is obtained choosing $u=v-\vet$ parallel to $\xi$. The radius of the basis of the cylinder is then $r:=(R+|v|)\tan \theta$ and its maximal length is the diameter of the ball. Thus,  $|B_{R}\cap D_{\theta,\xi}(v)| \leq c_{d}R\,r^{d-1},$
with $c_{d}$ depending only on $|\S^{d-2}|$, i.e.
%has also, for any $v \in \R^{d}$,
%\begin{equation*}\begin{split}
%\left|B_{R} \cap D_{\theta,\xi}(v) \right|&=\int_{|u-v| \leq R} \mathbf{1}_{|\widehat{u}\cdot \xi| \geq \cos \theta}\d u \leq \int_{|u| \leq R+|v|}\mathbf{1}_{|\widehat{u}\cdot \xi| \geq \cos\theta}\d u\\
%&\leq \int_{0}^{R+|v|}r^{d-1}\d r\int_{\S^{d-1}}\mathbf{1}_{|\widehat{u}\cdot \xi| \leq \cos\theta}\d \widehat{u}=\frac{1}{d}(R+|v|)^{d}\int_{\S^{d-1}}\mathbf{1}_{|\widehat{u}\cdot \xi| \geq \cos\theta}\d\widehat{u}.\end{split}\end{equation*}
%Notice that the integral over $\S^{d-1}$ is now independent of $\xi$ and is computed using spherical coordinates
%$$\int_{\S^{d-1}}\mathbf{1}_{|\widehat{u}\cdot \xi| \leq \cos\theta}\d\widehat{u}=|\S^{d-2}|\int_{0}^{\pi}\mathbf{1}_{|\cos \alpha| \geq \cos\theta}\left(\sin \alpha\right)^{d-2}\d\alpha=2|\S^{d-2}|\int_{0}^{\theta}\sin^{d-2}\alpha\d\alpha.$$
\begin{equation}\label{eq:Jd}
\left|B_{R} \cap D_{\theta,\xi}(v) \right| \leq c_{d}R\, (R+|v|)^{d-1}\tan^{d-1} \theta.\end{equation}
 We distinguish between two cases:

$\bullet$ \textit{First case: $|v|>2R$.} In such a case, for $\vet \in B_{R}$, one has $\frac{1}{2}|v| \leq |v-\vet| \leq \frac{3}{2}|v|$ and there exists $C_{\g} >0$ such that
$|v-\vet|^{\g+2} \geq C_{\g}|v|^{\g+2}$. According to \eqref{eq:AfBR}--\eqref{eq:Af}, it holds
\begin{equation}\label{eq:AfF}
\sum_{i,j} \, \mathcal{A}_{i,j}[f](v) \, \xi_i \, \xi_j\geq  C_{\g}|v|^{\g+2}\sin^{2}\theta \left(\frac{1}{2}\varrho_{\rm in}-\int_{B_{R}\cap D_{\theta,\xi}(v)}f(\vet)\d\vet\right)\end{equation}
where, according to \eqref{eq:Jd}, $|B_{R} \cap D_{\theta,\xi}(v)| \leq c_{d}R\left(\frac{3}{2}|v|\right)^{d-1}\tan^{d-1}\theta.$ With the notations of \eqref{Lem6DV}, given $\delta=\frac{1}{4}\varrho_{\rm in}$, we choose $\theta$ (depending on $v$) small enough so that 
$$c_{d}R\left(\frac{3}{2}|v|\right)^{d-1}\tan^{d-1}\theta=\eta(\delta), \qquad \text{ which entails } \quad 
{ \int_{B_{R}\cap D_{\theta,\xi}(v)}f(\vet)\d\vet \leq \frac{1}{4}\varrho_{\rm in} } , $$
according to \eqref{Lem6DV} (notice that such a choice of $\theta$ is always possible up to increasing $R$). Then, \eqref{eq:AfF} yields
$$\sum_{i,j} \, \mathcal{A}_{i,j}[f](v) \, \xi_i \, \xi_j\geq \frac{C_{\g}}{4}\varrho_{\rm in}|v|^{\g+2}\sin^{2}\theta.$$
Notice that $\tan\theta=\frac{2}{3|v|}\left(\frac{\eta(\delta)}{c_{d}R}\right)^{\frac{1}{d-1}}$,
 so that $\sin^{2}\theta=C_{R,d}|v|^{-2}\cos^{2}\theta$ for some explicit $C_{R,d}$ depending only on $d,R$ and $\rho_{\rm in}$, $E_{\rm in}$, $H(f_{\rm in})$. Therefore
 {$$\sum_{i,j} \, \mathcal{A}_{i,j}[f](v) \, \xi_i \, \xi_j\geq \frac{C_{\g}}{4}\rho_{\rm in} C_{R,d}|v|^{\g}\cos^{2}\theta, \qquad |v| >2R, $$}
with $\cos^{2}\theta > \frac{1}{2}$ since we choose $\theta$ small enough (up to having picked $R$ larger). Thus, there is $K >0$ (depending on $d,\g$, $\rho_{\rm in}$, $E_{\rm in}$, $H(f_{\rm in})$) such that
$$\sum_{i,j} \, \mathcal{A}_{i,j}[f](v) \, \xi_i \, \xi_j\geq K\langle v\rangle^{\g}, \qquad \forall |v| >2R.$$

$\bullet$ \textit{Second case: $|v| < 2R$.} In such a case, arguing exactly as in \cite[Prop. 1.3]{ALL}, we conclude that
$$\inf_{|v| < 2R} \sum_{i,j} \, \mathcal{A}_{i,j}[f](v) \, \xi_i \, \xi_j=\tilde{K} >0$$
which of course implies the result with $K_{0}=\min(K,\tilde{K})$. The details of this second case are left to the reader and are readily adapted from the corresponding result in \cite{ALL,DeVi1}.\end{proof}
\begin{rmq} From the above construction, the constant $K_{0}$ depends on $R$ (and therefore on $\varrho_{\rm in},E_{\rm in})$  and on the construction of the  mapping $\eta$ in \eqref{Lem6DV} so that $K_{0}$ depends also on $H(f_{\rm in})$.
\end{rmq}

The following property of the dissipation of entropy for Landau operator was observed in \cite{DesvJFA,Desv}:

\begin{prop}\label{cor:Fisherg} Let $0\leq f_{\mathrm{\rm in}}\in L^{1}_{2}(\R^{d}) \cap L\log L(\R^{d})$ be fixed as in Definition \ref{defi:fin}. Then, there is a positive constant $C_{0}(\gamma)$ depending only on $d$, $\g$,  $\rho_{\rm in}$, $E_{\rm in}$, $H(f_{\rm in})$
% $f_{\mathrm{\rm in}}$ through $\|f_{\rm in}\|_{L^{1}_{2}}$ and $H(f_{\mathrm{\rm in}})$, 
such that for all $f \in \mathcal{Y}_{0}(f_{\mathrm{\rm in}})$,
$$\int_{\R^{d}}\left|\nabla \sqrt{f(v)}\right|^{2}\langle v\rangle^{\gamma}\d v \leq C_{0}(\gamma)\big(1+\mathscr{D}(f)\big)\,.$$
where $\mathscr{D}(f)$ denotes the dissipation of entropy functional defined as
$$\mathscr{D}(f)=\frac{1}{2}\int_{\R^{d}\times \R^{d}}|v-\vet|^{\g+2}\,f\,f_{\ast}\left|\Pi(v-\vet)\left(\nabla \log f-\nabla \log f_{\ast}\right)\right|^{2}\d v\d \vet\geq0.$$
\end{prop}

In particular, since formally the entropy of solutions to the Landau equation is decreasing according to
$$\dfrac{\d}{\d t}H(f(t))=-\mathscr{D}(f(t)), \qquad t \geq 0,$$
one can proves that,  for H-solutions  $f=f(t,v)$ to the Landau equation, it holds
$$\int_{0}^{T}\mathscr{D}(f(t))\d t \le H(f_{\rm in}) \qquad T >0$$ and 
one deduces the following result in \cite{DesvJFA,Desv}:

\begin{cor}\label{cor:Desv}
If $f_{\rm in}$ and $f=f(t,v)$ define a  solution to Eq. \eqref{eq:Landau} in the sense of Definition \ref{defi:weak} then
$$\int_{0}^{T}\d t\int_{\R^{d}}\left|\nabla \left(\langle v\rangle^{\frac{\g}{2}}\sqrt{f(t,v)}\right)\right|^{2}\d v \leq \bm{C}_{T}(f_{\rm in})$$
where $\bm{C}_{T}(f_{\rm in})$ is an explicit positive constant depending only on $T,d,\g$, $\varrho_{\rm in},E_{\rm in}$ and $H(f_{\rm in}).$
\end{cor}

\begin{rmq} Notice that, since $f \in L^{\infty}([0,T]\,;\,L^{1}(\R^{d}))$, the above can be reformulated in the equivalent way as 
$$\int_{0}^{T}\d t\int_{\R^{d}}\left|\nabla \sqrt{f(t,v)}\right|^{2} \langle v\rangle^{\g}\d v \leq \bm{C}_{T}(f_{\rm in})$$
for some $\bm{C}_{T}(f_{\rm in})$ different from the previous one. It is in this form that the result is stated in \cite{DesvJFA}.\end{rmq}
Finally, recall the following \emph{a priori} growth of moments of solutions to  the Landau equation, referring to \cite{kleber, CDH} for a proof in dimension $d=3$. The general case is following the same lines (see for instance \cite{ABDL1} for a proof in the case of the Landau-Fermi-Dirac equation which is easily adapted to the Landau setting)

\begin{prop}\label{cor:moments} We consider $d\in \N$, $d \ge 3$ and 
 $\g \in [-d, 0)$, and assume that $f_{\rm in}$ and $f=f(t,v)$ define a  solution to Eq. \eqref{eq:Landau} in the sense of Definition \ref{defi:weak}. 
%Let $f_{\rm in}$ and $f=f(t,v)$ define a  solution to Eq. \eqref{eq:Landau} in the sense of Definition \ref{defi:weak}.  
Assume that
$$f_{\rm in} \in L^{1}_{s}(\R^{d}), \qquad s > 2,$$
then
$$\lm_{s}(f(t)) \leq C_{s}\left(1+t\right) \qquad  t\geq0$$
for some positive constant $C_{s}$ depending on $d$, $\g$, $s$,
 $\rho_{\rm in}$, $E_{\rm in}$, $H(f_{\rm in})$.
 % $\|f_{\rm in}\|_{L^{1}_{s}}$ and $H(f_{\rm in})$.
 In particular
$$\sup_{t \in [0,T]}\|f(t)\|_{L^{1}_{s}} \leq C_{s}(1+T) \qquad T >0.$$
\end{prop}
\medskip

 We complement also the above with the following properties of the matrix $\mathcal{A}[h]$ and vector $\bm{b}[h]$. Notice that the result applies without any sign assumption on the mapping $h\::\R^{d} \to \R.$

\begin{lem} \label{lem:AbF} Let  $d\in\N$, $d \ge 3$, $-d \le \g < -2$.
%and $h \in L^{p}(\R^{d})$ with  $p > \frac{d}{d+2+\g}$ be given. 
For any 
 $$q \in \left(\frac{d}{d+\g+2}, \infty \right), \qquad m > d\left(1-\frac{1}{q}\right),$$
 there is $C_{m,d,\g,q} >0$ depending only on $m,d, \g,q$ such that
\begin{equation}\label{eq:AF0}
 \|\mathcal{A}[h] \|_{\infty} \leq C_{m,d,\g,q}\left\|\langle \cdot \rangle^{m}h\right\|_{q}. 
\end{equation}
In the same way,
% if $p > \frac{d}{d+1+\g}$ and 
if $q \in \left(\frac{d}{d+\g+1}, \infty\right)$, 
\begin{equation}\label{eq:BF0} 
 \|\bm{b}[h]\|_{\infty} \leq \tilde{C}_{m,d,\g,q}\left\|\langle \cdot \rangle^{m}h\right\|_{q}, \qquad m > d\left(1-\frac{1}{q}\right)\,,
 \end{equation} 
for  $\tilde{C}_{m,d,\g,q} >0$ depending only on $m,d,\g,q$. 
\medskip

% As a consequence, given  $k \geq0$, if 
%$$h \in L^{p}(\R^{d}) \cap L^{1}_{N}(\R^{d}), \qquad \frac{d}{d+\g+2} < p < 2, \qquad N > \frac{2d-k(p-1)}{(2-p)(p-1)}$$
%then, \begin{equation}\label{eq:BF2}
%\|\mathcal{A}[h]\|_{\infty}^{2} \leq C_{p,N} \|h\|_{L^{1}_{N}}^{2-p}\,\|\langle \cdot \rangle^{\frac{k}{p}}h\|_{p}^{p},\end{equation}
%for some explicit positive constant $C_{p,N}$ depending only on $p,N$ and $d.$ The same estimate holds for $\bm{b}[h]$ 
%\begin{equation}\label{eq:BF2bis}
%\|\bm{b}[h]\|_{\infty}^{2} \leq C_{p,N} \|h\|_{L^{1}_{N}}^{2-p}\,\|\langle \cdot \rangle^{\frac{k}{p}}h\|_{p}^{p},\end{equation}
%if one assumes $\frac{d}{d+\g+1}  < 2.$ 
Then, for 
$$\frac{d}{d+\g+2} < q < \frac{d}{d+\g+1},$$ there exists $C_{d,\g,q} >0$ depending only on $d$, $\g$, $q$ such that
\begin{equation}\label{eq:BF1}
\|\bm{b}[h]\|_{p^{*}_{\g}} \leq C_{d,\g,q}\|h\|_{q} \qquad \quad  {p}^{*}_{\g}=\frac{qd}{d-q(1+\g+d)}.\end{equation}
Finally, if $q=\frac{d}{d+\g+1}$, for any $p \in (\frac{d}{d+\g+2},q)$ and any $m > |\g+1|$,
 there exists $C_{m,d,\g,p} >0$ such that
\begin{equation}\label{eq:BF12}
{\|\bm{b}[h]\|_{\bar{p}^{*}_{\g}} \leq C_{m,d,\g,p}\|\langle \cdot \rangle^{m}h\|_{q}, \qquad \bar{p}^{*}_{\g}=\frac{pd}{d-p(1+\g+d)}=\frac{pq}{q-p}.  }
\end{equation}
{In the case when $q=\frac{d}{d+\g+1}=2$, the estimate becomes, for any $p\in (\frac{2d}{d+2}, 2)$,} 
\begin{equation}\label{eq:BF12bis}
{\|\bm{b}[h]\|_{\frac{2p}{2-p}} \leq C_{m,d,\g,p}\|\langle \cdot \rangle^{m}h\|_{2} }
\end{equation}
\end{lem} 

  \begin{proof}  Let us write, for a smooth $h$,
  $$I_{s}[h](v)=\int_{\R^{d}}|v-w|^{s}\,|h(w)|\d w, \qquad s \in \R ,$$
  and notice that $|\mathcal{A}[h]| \leq 2I_{\g+2}[h]$ and $|\bm{b}[h]| \leq (d-1)I_{\g+1}[h].$  We prove the result for $I_{\g+2}[h]$. One has, for any $R >0$
 \begin{equation*}\begin{split}
 I_{\g+2}[h](v) &\leq \int_{|v-w| \leq R}|v-w|^{\g+2}|h(w)|\d w + \int_{|v-w|>R}|v-w|^{\g+2}|h(w)|\d w\\
 &\leq \left(\int_{|v-w|\leq R}|v-w|^{q'(\g+2)}\d w\right)^{\frac{1}{q'}}\|h\|_{q} + R^{\g+2}\|h\|_{1} ,
 \end{split}\end{equation*}
 where $\frac{1}{q}+\frac{1}{q'}=1$ and {$q >\frac{d}{2+d+\g}.$} One has then easily that there is $C >0$ (depending on $\g,d$ and $q$) such that
 $$I_{\g+2}[h](v) \leq C \left(R^{\g+2+\frac{d}{q'}}\|h\|_{q}+R^{\g+2}\|h\|_{1}\right), $$
 which, after optimizing $R$ (recall that $\g+2<0$ while $\g+2+\frac{d}{q'} >0$ since {$\frac1q < 1 + \frac{\g+2}d$}), reads simply
 $$|I_{\g+2}[h](v)| \leq {C}\, \|h\|_{q}^{-\frac{\g+2}{d}q'}\,\|h\|_{1}^{\frac{(\g+2)q'+d}{d}}.$$
 Since (for some $C$ depends on $d,m,q$) $\|h\|_{1} \leq C \,\|h\langle \cdot \rangle^{m}\|_{q}$ for any $m >\frac{d}{q'}$, this yields \eqref{eq:AF0}. 
\medskip

Now, the same exact estimate holds for $|\bm{b}[h]| \leq (d-1)I_{\g+1}[h]$ under the assumption  {$q > \frac{d}{d+\g+1}$} and this proves \eqref{eq:BF0}.      
%\par
 %To deduce now \eqref{eq:BF2}, let us fix therefore $\frac{d}{d+\g+2} < p < 2$, and choosing $\frac{d}{d+\g+2} < q < p$,
% we resort to the additional interpolation 
%$$\|\langle \cdot \rangle^{m}h\|_{q} \leq \|\langle \cdot \rangle^{\frac{k}{p}}h\|_{p}^{\theta}\|\langle \cdot\rangle^{N}h\|_{1}^{1-\theta}, \quad m=\theta\,\frac{k}{p}+(1-\theta)N$$
% with $\theta=\frac{p}{q}\frac{q-1}{p-1}.$ %
%Choosing then $q$ such that $\theta=\frac{p}{2}$ (i.e. $q=\frac{2}{3-p}$ \textcolor{blue}{I find $p> 1 - \frac2d (\g +2)$ instead of $p> 1 + \frac2d (\g +2)$ in the footnote, this seems to require an extra condition on $p$. L.}) \footnote{Notice, for $1 < p <2$, $q=\frac{2}{3-p} < p$. Moreover,  $q >\frac{d}{d+\g+2}$ if and only if $p >1+\frac{2}{d}(\g+2)$. Since $2+\g <0$ and $p > \frac{d}{d+\g+2} >1$, this condition is satisfied. }
%so that $N=\frac{2m-k}{2-p}$  we
 %deduce that
%$$\|\mathcal{A}[h]\|_{\infty}^{2} \leq  C_{m,d,q}^{2}\|\langle \cdot \rangle^{\frac{k}{p}}h\|_{p}^{p}\,\|\langle \cdot \rangle^{N}h\|_1^{2-p}$$
%which is the desired result since $m > \frac{d}{q'}$ with $q'=\frac{2}{p-1}$ yields the desired constraint on $N$. The proof of \eqref{eq:BF2bis} follows the same line. 
\medskip

Let us now
prove \eqref{eq:BF1}, recalling that {$\frac{d}{d+\g+2} <q < \frac{d}{d+\g+1}$}.  According to Hardy-Littlewood-Sobolev inequality (Proposition \ref{prop:HLS}), for any $q < \frac{d}{d+\g+1}$, a simple duality argument shows that
\begin{equation}\label{eq:Ig+1}
\|\,I_{\g+1}[h]\|_{r'} \leq C_{\rm HLS}\|h\|_{q}, \qquad \frac{1}{r'}=1-\frac{1}{r}=\frac{1}{q}-\frac{1+\g}{d}-1=\frac{d-q(1+\g+d)}{qd}.\end{equation}
Notice that $r' > 1 \iff q > \frac{d}{1+\g+2d}$ which is satisfied since $\frac{d}{d+\g+2} > \frac{d}{1+\g+2d}$. This proves \eqref{eq:BF1}. 

Strangely, the case $q=\frac{d}{1+\g+d}$ is not covered by the above statements and deserves a separate treatment. Indeed, Hardy-Littlewood-Sobolev inequality \eqref{eq:Ig+1} for $q=\frac{d}{d+\g+1}$ would give $r'=\infty$, which is excluded by \eqref{eq:HLS}. We rather apply it with $p < q=\frac{d}{d+\g+1}$ and, as above, deduce
$$\|\,I_{\g+1}[h]\|_{r'} \leq C_{\rm HLS}\|h\|_{p}, \qquad \frac{1}{r'}=\frac{d-p(1+\g+d)}{pd}=\frac{1}{p}-\frac{1}{q}.$$
Since $1  < p  < q$, we use a further interpolation $\|h\|_{p} \leq \|h\|_{1}^{\theta}\,\|h\|_{q}^{1-\theta}$ with $\theta=\frac{q-p}{p(q-1)}$ and, with Young's inequality
$$\|h\|_{p} \leq \|h\|_{1}+\|h\|_{q} \leq C\|\langle \cdot \rangle^{m}h\|_{q}, \qquad m > d\left(1-\frac{1}{q}\right).$$
 {This concludes the proof of the Lemma.}
%This shows that \eqref{eq:BF1} holds true for $q=\frac{d}{d+\g+1}$ with 
%Then, with now $q=\frac{2p}{2-p},$ $1< p< 2$, one has
%$$\|\bm{b}[h]\|_{q}^{2} \leq C\,\|\langle \cdot \rangle^{m}h\|_{2}^{2}.$$
%This means that, in this case, we need to replace $q=\infty$ with \emph{any} $q=\frac{2p}{2-p}$ and we need to check that the uniqueness still follows.
\end{proof}

  We finally recall a result regarding the $\e$-Poincar\'e inequality in the (critical) case $\g=-2$ in any dimension established in  \cite{ABL}.

\begin{prop}\label{propo:GG2}
 Let $d\in\N, d \geq 3$. 
 Assume   $f_{\rm in}$ be as in  Definition \ref{defi:fin}. Then there exists $C_{0} >0$ depending only on $\|f_{\rm in}\|_{L^{1}_{2}}$ and $H(f_{\rm in})$ such that
% For any $f \in \mathcal{Y}_{0}(f_{\mathrm{in}})$ \textcolor{blue}{I think that the notation is not defined in this version. L.} 
for any $\e >0$,  
\begin{multline*}
\int_{\R^{d}}\phi^{2} {\bm{c}_{-2}}[g]\d v \leq \e\int_{\R^{d}}\left|\nabla \left( {\langle v\rangle^{-1}}\phi(v)\right)\right|^{2}\d v 
+C_{0}\exp\left(\e^{-\frac{s}{s-2}}\,\lm_{s}(g)^{\frac{s}{s-2}}\right)\int_{\R^{d}}\phi^{2} {\langle v\rangle^{-2}}\d v\,
\end{multline*}
holds true for any suitable nonnegative functions $g \in \mathcal{Y}_{0}(f_{\mathrm{in}}) \cap L^{1}_{s}(\R^{d})$ $(s >2)$ and  $\phi$.
\end{prop}  
\begin{proof} We refer to \cite{ABL} for a full proof of the result in general dimension $d \geq 3$ which uses Lorentz spaces. We just explain the main steps of it, referring to \cite{ABL} for details and to \cite{chemin} for definition and properties of Lorentz spaces. 
 As in the proof of Proposition \ref{prop:GG}, eq. \eqref{eq:ineq},  we decompose, according to 
$|v-v_*| \ge \frac{\langle v\rangle}2$ or not, 
%we need to estimate
$$\int_{\R^{d}\times \R^{d}}|v-\vet|^{-2}g(\vet)\phi^{2}(v)\d v=I_{1}+I_{2},$$
with $I_{1} \leq  \|g\|_{{1}} \|\psi^{2}\|_{{1}}\,,$ and $I_{2} \leq  C_{d,\g}\, \int_{\R^{d}\times \R^{d}}|v-\vet|^{-2}\psi(v)^{2}F(\vet)\d \vet\d v\,,$ still using the notations
$ F(v)=\langle v\rangle^{2}g(v)$, $\psi(v)=\langle v\rangle^{-1}\phi(v)$.
%Suitable estimates of the Riesz operator in Lorentz spaces give
%$$ J \leq C_{d}\|F\|_{{1}}\,\|\psi\|_{L^{2p,2}}^{2}, \qquad p=\tfrac{d}{d+\g}=\tfrac{d}{d-2}\,. $$
%Now we are in the situation in which the exponent $2p=\frac{2d}{d-2}$ is precisely the Sobolev exponent and Sobolev embedding with Lorentz spaces yields then
%$$ J \leq C_{d}\|F\|_{1}\,\|\nabla \psi\|_{2}^{2}\,, $$
%which gives
%$$\int_{\R^{d}}\phi^{2}\bm{c}_{-2}[g]\d v \leq C_{d}\left(\|F\|_{1}\,\|\nabla \psi\|_{2}^{2}+\|f\|_{1}\|\psi\|^{2}_{L^{2}}\right)\,.$$
%This would corresponds to $\e$-Poincar\'e inequality for $\e=1$ but, 
%In order to create room and make the term in front of $\|\nabla \psi\|_{2}^{2}$ arbitrarily small, one estimate $\|F\|_{1}$ using moments and entropy. Typically, one splits the integral  defining $J$ according to $|v-\vet| >1$ and $|v-\vet| \leq 1$ to get
Then $I_2 \le J_{1}+J_{2}$, with
$$J_{1}:=\int_{|v-\vet|>1}|v-\vet|^{-2}F(\vet)\psi^{2}(v)\d v\d\vet \leq \|F\|_{L^{1}}\|\psi\|_{2}^{2}\,,$$
and
$$J_{2}:=\int_{|v-\vet|\leq 1}|v-\vet|^{-2}F(\vet)\psi^{2}(v)\d v\d\vet\, .$$
The term $J_2$ is then further split, for $R >0$, according to 
$$F=F^{+}_{R}+F^{-}_{R}, \qquad F^{+}_{R}=F\ind_{F >R}, \qquad F^{-}_{R}=F\ind_{F \leq R}\,,$$
and 
$$ J_{2,R}^{\pm} := \int_{|v-v_*|\le 1} |v-v_*|^{-2}\, F^{\pm}_{R}(v_*)\, \psi^{2}(v)\d v\d\vet . $$
Then, $J_{2,R}^{-} \le C_d\, R\, \|\psi\|_{2}^{2}$, while, thanks to the use of Lorentz spaces, 
$$  J_{2,R}^{+} \leq C_{d}\|F_R^+\|_{{1}}\,\|\psi\|_{L^{\frac{2d}{d-2},2}}^{2} \le  C_{d}\|F_R^+\|_{{1}}\, \|\nabla \psi\|_{2}^{2}. $$
Then,
%Performing the same kind of estimates that lead to the previous estimate of $J$, 
we end up with 
\begin{equation*}
\int_{\R^{d}}\phi^{2}\bm{c}_{-2}[g]\d v \leq  C_{d,\g} \, \bigg(  \|F\|_{1}\|\psi\|_{2}^{2}+ R\|\psi\|_{2}^{2}+ \|F^{+}_{R}\|_{1}\,\|\nabla \psi\|_{2}^{2} \bigg), \qquad \forall R >0\,.
\end{equation*} 
In order to conclude the proposition, one needs to
%Now, using $\lm_{s}(g) < \infty$ and $H(g) < \infty$, one can easily
 prove that $\|F_{R}^{+}\|_{1}$ is small as $R$ is sufficiently large. This comes from the observation that  
for, $s>2$ and $\theta_{s}=1-\frac{2}{s}$ and any suitable function $g \in L^{1}_{s} \cap L\log L(\R^{d})$, 
$$\| \langle\cdot\rangle^{2} g\ind_{f >R} \|_{1}\leq \| \langle\cdot\rangle^{s} g\ind_{g >R} \|^{1-\theta_s}_{1}\| g\ind_{g >R} \|^{\theta_s}_{1}\leq \| \langle\cdot\rangle^{s} g \|^{1-\theta_s}_{1}\bigg(\frac{\|g\log g \|_{1}}{\log(R)}\bigg)^{\theta_s}\,.$$
We refer the reader  to \cite{ABL} for more details.
% in which such an idea is applied to the mapping $F=\langle \cdot\rangle^{2}g.$
\end{proof}

\section{Evolution of weighted $L^{p}$-norms}\label{sec:weight}

We briefly explain here how the analysis of Sections \ref{sec:Coulomb} and \ref{sec:prodi} can be modified to allow the appearance of \emph{weighted} $L^{p}$-norms. We recall the notations
$$
\lM_{k,p}(t) :=\int_{\R^{d}}f(t,v)^{p}\langle v\rangle^{k}\d v, \qquad \lD_{k,p}(t) :=\int_{\R^{d}}\left|\nabla \left(\langle v\rangle^{\frac{k}{2}}f^{\frac{p}{2}}(t,v)\right)\right|^{2}\d v\,.
$$
for $k\in \R$ and $p \in (1,\infty)$. Let us dig directly into the technical lemma
\begin{lem}\label{prop:Lp-esti}
Consider $d\in \N$, $d \ge 3$ and $\g \in [-d, 0)$, and assume that $f_{\rm in}$ and $f=f(t,v)$ define a  solution to Eq. \eqref{eq:Landau} in the sense of Definition \ref{defi:weak}.  Then for all $k\in \R_+$ and $p \in (1,\infty)$,
\begin{multline}\label{eq:dtMsp}
\dfrac{\d}{\d t}\lM_{k,p}(t) + 2K(p)\lD_{k+\g,p}(t) \leq K(p)(k+\g)^{2}\lM_{k+\g,p}(t)\\
+C_{k,\g,p}\sum_{i=0}^{2}\int_{\R^{d}}\langle v\rangle^{k-i} \bm{c}_{\g+i}[f(t)](v)\,f^p(t,v) \d v ,
\end{multline}
with ($K_0$ being the constant appearing in Remark \ref{diffusion})
$$K(p) :=\frac{p-1}{p}K_{0}, \qquad C_{k,\g,p} :=\max\left(p-1,\frac{2k}{\g+1+d},\frac{d^{2}k^{2}}{(d-1)(d+\g+2)}\right).$$
\end{lem}

\begin{proof} We easily check that, for any $k\geq 0$,
 \begin{multline*}
\frac{1}{p}\frac{\d}{\d t} \int_{\R^{d}} f^p(t,v) \langle v\rangle^{k} \d v =\int_{\R^{d}}\langle v\rangle^{k}f^{p-1}(t,v)\nabla \cdot \left(\mathcal{A}[f]\nabla f-\bm{b}[f] f\right)\d v\\
=-(p-1)\int_{\R^{d}}\langle v\rangle^{k}f^{p-2}\mathcal{A}[f]\nabla f\cdot   \nabla f\d v + (p-1)\int_{\R^{d}}\langle v\rangle^{k}f^{p-1} \bm{b}[f]\cdot \nabla f\d v\\
-k\int_{\R^{d}}\langle v\rangle^{k-2}\,f^{p-1}\left(\mathcal{A}[f]\nabla f\right)\cdot v\d v +k \int_{\R^{d}}\langle v\rangle^{k-2}f^{p} \bm{b}[f]\cdot v\,\d v\,. 
\end{multline*}
Arguing as   in the proof of Prop. \ref{prop:estMs-g}, we obtain that
$$(p-1)\int_{\R^{d}}\langle v\rangle^{k}f^{p-2}\mathcal{A}[f]\nabla f\cdot   \nabla f\d v 
 \geq \frac{4K_{0}(p-1)}{p^{2}} \int_{\R^{d}}\langle v\rangle^{k+\g}\,\left|\grad (f^{\frac{p}{2}})\right|^{2}\d v\,.$$
Moreover, writing
\begin{equation*}
\nabla \left(\langle v\rangle^{\frac{k+\g}{2}}\,f^{\frac{p}{2}}\right)=\langle v\rangle^{\frac{k+\g}{2}}\nabla (f^{\frac{p}{2}}) + \frac{k+\g}{2}v\,\langle v\rangle^{\frac{k+\g}{2}-2}f^{\frac{p}{2}}\,,
\end{equation*}
from which
\begin{equation}\label{eq:Gradient}
\langle v\rangle^{k+\g}\left|\nabla (f^{\frac{p}{2}})\right|^{2} \geq \frac{1}{2}\left|\nabla \left(\langle v\rangle^{\frac{k+\g}{2}}f^{\frac{p}{2}}\right)\right|^{2} - \frac{1}{4}(k+\g)^{2} \langle v\rangle^{k+\g-2}f^{p}, \end{equation}
we observe that
\begin{multline*}
(p-1)\int_{\R^{d}}\langle v\rangle^{k}f^{p-2}\mathcal{A}[f]\nabla f\cdot   \nabla f\d v 
 \geq \frac{2K_{0}(p-1)}{p^{2}} \int_{\R^{d}}\left|\nabla \left(\langle v\rangle^{\frac{k+\g}{2}}f^{\frac{p}{2}}\right)\right|^{2}\d v\\
 -\frac{K_{0}(p-1)(k+\g)^{2}}{p^{2}}\int_{\R^{d}} \langle v\rangle^{k+\g-2}f^{p}\d v\,.\end{multline*}
Also, note that
\begin{multline*}
 \int_{\R^{d}}\langle v\rangle^{k} f^{p-1}  \bm{b}[f] \cdot \grad f  \d v
 =  -\frac{1}{p} \int_{\R^{d}}  f^p \grad \cdot \Big(\bm{b}[f] \langle v\rangle^{k}\Big) \d v\\
=-\frac{k}{p}\int_{\R^{d}}\langle v\rangle^{k-2} f^p  \bm{b}[f]\cdot v\,\d v + \frac{1}{p}\int_{\R^{d}}\langle v\rangle^{k} f^p \bm{c}_{\g}[f]\, \d v.
\end{multline*}
Therefore, 
\begin{multline*}
 \dfrac{\d}{\d t}\lM_{k,p}(t) + \frac{2K_{0}(p-1)}{p} \lD_{k+\g,p}(t) \\
 \leq   {(p-1)} \int_{\R^{d}}\langle v\rangle^{k} f^p\,\bm{c}_{\g}[f]\d v
+k \int_{\R^{d}}\langle v\rangle^{k-2} f^p\,\left(\bm{b}[f]\cdot v\right)\d v  \\
+ \frac{K_{0}(p-1)(k+\g)^{2}}{p}\int_{\R^{d}} \langle v\rangle^{k+\g-2}f^{p}\d v
-kp\int_{\R^{d}}\langle v\rangle^{k-2}f^{p-1} \left(\mathcal{A}[f]\nabla f\cdot v \right)\d v.
\end{multline*}
Let us investigate in more detail the latter term. Integration by parts shows that \begin{align*}
-kp\int_{\R^{d}}\langle v\rangle^{k-2}f^{p-1} \left(\mathcal{A}[f]\nabla f\right) \cdot v \d v &= -k\int_{\R^{d}}\nabla    (f^{p}) \cdot \Big(\mathcal{A}[f]
\, \langle v\rangle^{k-2}v \Big) \,\d v\\
&= k\int_{\R^{d}} f^{p} \; \nabla\cdot \Big(\mathcal{A}[f]\,\langle v\rangle^{k-2} v \Big) \,\d v\,.
\end{align*}
Using the product rule 
$$\nabla\cdot \Big(\mathcal{A}[f]\,\langle v\rangle^{k-2} v \Big)=\langle v\rangle^{k-2}\,{\bm{b}}[f]\cdot v\, + \mathrm{Trace}\left(\mathcal{A}[f] \cdot \,
D_{v}\left(\langle v\rangle^{k-2} v\right)\right), $$ 
where $D_{v}\big( \langle v\rangle^{k-2} v\big)$ is the matrix with entries $\partial_{v_{i}}\big( \langle v\rangle^{k-2} v_{j}\big)$, $i,j=1,\ldots,d$, or more compactly, 
$$D_{v}\big( \langle v\rangle^{k-2} v\big)=\langle v\rangle^{k-4}\bm{A}(v)\,,$$
where $\bm{A}(v)=\langle v\rangle^{2}\mathbf{Id}+(k-2)\,v\otimes v$, $v \in \R^{d}$,
we obtain
\begin{multline}\label{eq:dtMs}
 \dfrac{\d}{\d t}\lM_{k,p}(t) + \frac{2K_{0}(p-1)}{p}\lD_{k+\g,p}(t) \leq  { {(p-1)}\int_{\R^{d}}\langle v\rangle^{k} \bm{c}_{\g}[f](v)\,f^p \d v }\\
+2k \int_{\R^{d}}\langle v\rangle^{k-2}\,f^p \left(\bm{b}[f]\cdot v\right)\d v  +  \frac{K_{0}(p-1)(k+\g)^{2}}{p}\int_{\R^{d}}\langle v\rangle^{k+\g-2}f^{p}\d v
\\
+k\int_{\R^{d}}\langle v\rangle^{k-4}f^{p}\,\, \mathrm{Trace}\left(\mathcal{A}[f]\cdot \bm{A}(v)\right)\d v\,.
\end{multline}
We denote by $I_{1},I_{2},I_{3},I_{4}$ the various terms on the right-hand-side of \eqref{eq:dtMs}
and control each of them separately.  We get
\begin{multline}\label{eq:I2B}
|I_{2}| \leq 2k\int_{\R^{d}}\langle v\rangle^{k-1}f^{p}(t,v)\,|\bm{b}[f(t)](v)|\d v \\
\leq 2k(d-1)\,\int_{\R^{2d}}\langle v\rangle^{k-1}f^{p}(t,v)|v-\vet|^{\g+1}f(t,\vet)\d \vet\d v\, ,
\end{multline}
that is,
$$
|I_{2}| \leq \frac{2k}{\g+1+d}\int_{\R^{d}}\langle v\rangle^{k-1}f^{p}(t,v)\bm{c}_{\g+1}[f(t)](v)\d v.
$$
Clearly, we also get
$$|I_{3}| \leq  \frac{K_{0}(p-1)(k+\g)^{2}}{p}\lM_{k+\g,p}(t).$$
For the term $I_{4}$, one checks easily that, for any $i,j \in \{1,\ldots,d\}$,
$$\left|\mathcal{A}_{i,j}[f]\right| \leq |\cdot|^{\g+2}\ast f, \qquad \left|\bm{A}_{i,j}(v)\right| \leq k\langle v\rangle^{2}, $$
and 
$$|I_{4}| \leq d^{2}k^{2}\int_{\R^{2d}}\langle v\rangle^{k-2}f^{p}(t,v)|v-\vet|^{\g+2}f(t,\vet)\d v\d\vet.$$
Therefore, 
$$|I_{4}| \leq \frac{d^{2}k^{2}}{(d-1)(d+\g+2)}\int_{\R^{d}}\langle v\rangle^{k-2}f^{p}(t,v)\bm{c}_{\g+2}[f](v)\d v,$$
which proves the result thanks to \eqref{eq:dtMs}.
\end{proof}
We now estimate the various terms 
$$
\int_{\R^{d}}\langle v\rangle^{k-i}\bm{c}_{\g+i}[f(t)](v)f^{p}(t,v)\d v, \qquad i=0,1,2\,,
$$
and observe that the less favourable estimate corresponds to the most singular case $i=0$.

\begin{lem}\label{lem:compa} Consider $d\in\N$. For any $-d < \g <0$,  $p \in (1,\infty)$, $k\in \R_+$, and suitable functions $f\ge 0$ and $h$,
the following estimate holds:
\begin{multline*}
  \sum_{i=0}^{2}\int_{\R^{d}}\langle v\rangle^{k-i}\bm{c}_{\g+i}[f](v)\,|h(v)|^{p}\d v  \leq C_{d,\g}\int_{\R^{d}}\langle v\rangle^{k}\overline{\bm{c}}_{\g}[f](v)|h(v)|^{p}\d v+c_{d,\g}\int_{\R^{d}} \langle v\rangle^{k} |h(v)|^{p}\d v ,
\end{multline*}
where  $C_{d,\g}:=3\frac{\g+d+2}{\g+d}$, $c_{d,\g}:=3(d-1)\,(\g+d+2)\, \|f\|_{L^1_2 }$, and
%for some $C_{d,\g},c_{d,\g} >0$ depending only on $d,\g$, $\|f\|_{L^1_2(\R^d)}$, and
 where
$$\overline{\bm{c}}_{\g}[f] :=(d-1)(d+\g)\int_{|v-\vet|\leq1}|v-\vet|^{\g}f(\vet)\d\vet.$$
In the Coulomb case $\g=-d$, the estimate becomes:
\begin{multline}\label{eq:lemcompar}
\sum_{i=0}^{2}\int_{\R^{d}}\langle v\rangle^{k-i}\bm{c}_{\g+i}[f](v)\,|h(v)|^{p}\d v \leq c_{d}\int_{\R^{d}} \langle v\rangle^{k}\,f(v) \,|h(v)|^{p}\d v\\
+\tilde{C}_{d} \int_{\R^{d}}\langle v\rangle^{k-1}\overline{\bm{c}}_{1-d}[f](v)\,|h(v)|^{p}\d v+   \tilde{c}_{d} \int_{\R^{d}} \langle v\rangle^{k-1} |h(v)|^{p}\d v ,
\end{multline}
for {$\tilde{C}_{d} := 3(d-1)$, $\tilde{c}_{d} := 3(d-1)\,\|f_{\rm in}\|_{1}$. }
\end{lem}

\begin{proof}  For $i=0,1,2$, one simply writes
$$\bm{c}_{\g+i}[f](v)=(d-1) \,(\g+d+i)\int_{\R^{d}}f(\vet)|v-\vet|^{\g+i}\d\vet , $$
and splits the above integral according to $|v-\vet| >1$ and $|v-\vet| \leq 1$. Clearly,
\begin{multline*}
\int_{|v-\vet| \leq 1}f(\vet)|v-\vet|^{\g+i}\d\vet \leq \int_{|v-\vet| \leq 1}f(\vet)|v-\vet|^{\g}\d\vet
\leq \frac{1}{(d-1)(\g+d)}\overline{\bm{c}}_{\g}[f]\,.
\end{multline*}
Whereas, given that $\g <0$,
\begin{multline*} 
\int_{|v-\vet| > 1}f(\vet)|v-\vet|^{\g+i}\d\vet  \leq \int_{|v-\vet|>1}f(\vet)|v-\vet|^{i}\d\vet
\leq \langle v\rangle^{i}\int_{\R^{d}}f(\vet)\langle \vet\rangle^{i}\d\vet ,
\end{multline*}
since $|v-\vet| \leq \langle v\rangle\langle \vet\rangle$.  Recalling that $i=0,1,2$, 
$$
\int_{\R^{d}}f(\vet)\langle\vet\rangle^{i}\d\vet \leq \int_{\R^{d}}f(\vet)\langle\vet\rangle^{2}\d\vet = \|f\|_{L^1_2} .
$$
Thus, 
$$\bm{c}_{\g+i}[f] \leq \frac{\g+d+i}{\g+d}\,\overline{\bm{c}}_{\g}[f] + (d-1) \,(\g+d+i)\,\langle v\rangle^{i} \, \|f\|_{L^1_2 }, $$
and the result follows by defining $C_{d,\g}$ and $c_{d,\g}$ as in the statement.

In the case $\g=-d$, the proof is almost identical and one simply needs to  estimate the sum 
$$\sum_{i=1}^{2} \int_{\R^{d}} \langle v\rangle^{k-i}\bm{c}_{i-d}[f](v) \,|h(v)|^{p}\d v.$$
Then we observe that for $i=1,2$, 
%and some $c_{d}^{*}>0$ depending only on $d$, 
{$$ \int_{|v-\vet| \leq 1}f(\vet)|v-\vet|^{i-d}\d\vet \leq  \overline{\bm{c}}_{1-d}[f], \qquad \text{ 
while } \int_{|v-\vet| \geq 1}f(\vet)|v-\vet|^{i-d}\d\vet \leq  \|f\|_{1}.$$
Therefore, $\bm{c}_{i-d}[f] \le (d-1)\,i\,  \bigg(   \overline{\bm{c}}_{1-d}[f] + \|f\|_{1} \bigg),$
which gives \eqref{eq:lemcompar}. }
 \end{proof}

 \bibliographystyle{plainnat-linked}
		
\end{document}